\documentclass[12pt]{amsart}
\usepackage[a4paper, bottom=30mm, left=30mm, right=30mm, top=30mm]{geometry}
\usepackage{amsmath, amsfonts, amssymb, amsthm, tikz-cd,amscd}
\usepackage[mathscr]{euscript}
\usepackage[colorlinks=true, linkcolor=blue, citecolor=blue]{hyperref}

\newcommand{\mc}[1]{{\mathcal{#1}}}
\newcommand{\mb}[1]{{\mathbb{#1}}}
\newcommand{\mf}[1]{{\mathfrak{#1}}}
\newcommand{\ms}[1]{{\mathcal{#1}}}
\newcommand{\mrm}[1]{{\mathrm{#1}}}
\newcommand{\mit}[1]{{\mathit{#1}}}
\newcommand{\mbf}[1]{{\mathbf{#1}}}
\newcommand{\bslash}{
  \mathchoice
    {\reflectbox{$\displaystyle{/}$}}
    {\reflectbox{$\textstyle{/}$}}
    {\reflectbox{$\scriptstyle{/}$}}
    {\reflectbox{$\scriptscriptstyle{/}$}}}
\newcommand{\Db}{{\ms{D}\mit{b}}}
\newcommand{\Mod}{\mathrm{Mod}}
\newcommand{\id}{\mathrm{id}}
\newcommand{\mhm}{\mrm{MHM}}
\renewcommand{\hom}{\mrm{Hom}}
\newcommand{\shom}{{\ms{H}\mit{om}}}

\DeclareMathOperator{\spec}{Spec}
\DeclareMathOperator{\coker}{coker}
\DeclareMathOperator*{\Res}{Res}
\DeclareMathOperator{\codim}{codim}
\newcommand{\contract}{\mathrel{\lrcorner}}

\newcommand{\longrightleftarrows}{\mathrel{\ooalign{\raise.2em\hbox{$\longrightarrow$}\cr\raise-.2em\hbox{$\longleftarrow$}}}}

\newcommand{\saito}[1]{{{#1}^{\mathrm{Saito}}}}

\theoremstyle{plain}
\newtheorem{thm}{Theorem}[section]
\newtheorem{prop}[thm]{Proposition}
\newtheorem{lem}[thm]{Lemma}
\newtheorem{cor}[thm]{Corollary}

\theoremstyle{definition}

\newtheorem{defn}[thm]{Definition}
\newtheorem{rmk}[thm]{Remark}

\numberwithin{equation}{section}
\newcommand{\nc}{\newcommand}

\nc{\bA}{{\mathbb A}}
\nc{\bB}{{\mathbb B}}
\nc{\bC}{{\mathbb C}}
\nc{\bD}{{\mathbb D}}
\nc{\bE}{{\mathbb E}}
\nc{\bF}{{\mathbb F}}
\nc{\bG}{{\mathbb G}}
\nc{\bH}{{\mathbb H}}
\nc{\bI}{{\mathbb I}}
\nc{\bJ}{{\mathbb J}}
\nc{\bK}{{\mathbb K}}
\nc{\bL}{{\mathbb L}}
\nc{\bM}{{\mathbb M}}
\nc{\bN}{{\mathbb N}}
\nc{\bO}{{\mathbb O}}
\nc{\bP}{{\mathbb P}}
\nc{\bQ}{{\mathbb Q}}
\nc{\bR}{{\mathbb R}}
\nc{\bS}{{\mathbb S}}
\nc{\bT}{{\mathbb T}}
\nc{\bU}{{\mathbb U}}
\nc{\bV}{{\mathbb V}}
\nc{\bW}{{\mathbb W}}
\nc{\bZ}{{\mathbb Z}}
\nc{\bX}{{\mathbb X}}
\nc{\bY}{{\mathbb Y}}

\nc{\cA}{{\mathcal A}}
\nc{\cB}{{\mathcal B}}
\nc{\cC}{{\mathcal C}}
\nc{\cD}{{\mathcal D}}
\nc{\cE}{{\mathcal E}}
\nc{\cF}{{\mathcal F}}
\nc{\cG}{{\mathcal G}}
\nc{\cH}{{\mathcal H}}
\nc{\cI}{{\mathcal I}}
\nc{\cJ}{{\mathcal J}}
\nc{\cK}{{\mathcal K}}
\nc{\cL}{{\mathcal L}}
\nc{\cM}{{\mathcal M}}
\nc{\cN}{{\mathcal N}}
\nc{\cO}{{\mathcal O}}
\nc{\cP}{{\mathcal P}}
\nc{\cQ}{{\mathcal Q}}
\nc{\cR}{{\mathcal R}}
\nc{\cS}{{\mathcal S}}
\nc{\cT}{{\mathcal T}}
\nc{\cU}{{\mathcal U}}
\nc{\cV}{{\mathcal V}}
\nc{\cW}{{\mathcal W}}
\nc{\cZ}{{\mathcal Z}}
\nc{\cX}{{\mathcal X}}
\nc{\cY}{{\mathcal Y}}

\nc{\fA}{{\mathfrak A}}
\nc{\fB}{{\mathfrak B}}
\nc{\fC}{{\mathfrak C}}
\nc{\fD}{{\mathfrak D}}
\nc{\fE}{{\mathfrak E}}
\nc{\fF}{{\mathfrak F}}
\nc{\fG}{{\mathfrak G}}
\nc{\fH}{{\mathfrak H}}
\nc{\fI}{{\mathfrak I}}
\nc{\fJ}{{\mathfrak J}}
\nc{\fK}{{\mathfrak K}}
\nc{\fL}{{\mathfrak L}}
\nc{\fM}{{\mathfrak M}}
\nc{\fN}{{\mathfrak N}}
\nc{\fO}{{\mathfrak O}}
\nc{\fP}{{\mathfrak P}}
\nc{\fQ}{{\mathfrak Q}}
\nc{\fR}{{\mathfrak R}}
\nc{\fS}{{\mathfrak S}}
\nc{\fT}{{\mathfrak T}}
\nc{\fU}{{\mathfrak U}}
\nc{\fV}{{\mathfrak V}}
\nc{\fW}{{\mathfrak W}}
\nc{\fZ}{{\mathfrak Z}}
\nc{\fX}{{\mathfrak X}}
\nc{\fY}{{\mathfrak Y}}
\nc{\fa}{{\mathfrak a}}
\nc{\fb}{{\mathfrak b}}
\nc{\fc}{{\mathfrak c}}
\nc{\fd}{{\mathfrak d}}
\nc{\fe}{{\mathfrak e}}
\nc{\ff}{{\mathfrak f}}
\nc{\fg}{{\mathfrak g}}
\nc{\fh}{{\mathfrak h}}
\nc{\fiI}{{\mathfrak i}}  
\nc{\ffi}{{\mathfrak i}}  
\nc{\fj}{{\mathfrak j}}
\nc{\fk}{{\mathfrak k}}
\nc{\fl}{{\mathfrak{l}}}
\nc{\fm}{{\mathfrak m}}
\nc{\fn}{{\mathfrak n}}
\nc{\fo}{{\mathfrak o}}
\nc{\fp}{{\mathfrak p}}
\nc{\fq}{{\mathfrak q}}
\nc{\fr}{{\mathfrak r}}
\nc{\fs}{{\mathfrak s}}
\nc{\ft}{{\mathfrak t}}
\nc{\fu}{{\mathfrak u}}
\nc{\fv}{{\mathfrak v}}
\nc{\fw}{{\mathfrak w}}
\nc{\fz}{{\mathfrak z}}
\nc{\fx}{{\mathfrak x}}
\nc{\fy}{{\mathfrak y}}

\nc{\al}{{\alpha }}
\nc{\be}{{\beta }}
\nc{\ga}{{\gamma }}
\nc{\de}{{\delta }}
\nc{\del}{{\partial }}
\nc{\ep}{{\varepsilon }}
\nc{\vap}{{\epsilon }}

\nc{\ze}{{\zeta }}
\nc{\et}{{\eta }}

\nc{\oh}{{\operatorname {H}}}
\nc{\ucd}{{(\fh^{univ})^*}}
\nc{\std}[3]{{\cM(#1,#2,#3)_* }}
\nc{\stdd}[3]{{\cM(#1,#2,#3)_! }}
\nc{\irr}[3]{{\cI(#1,#2,#3)}}
\nc{\uc}{{\fh^{univ}}}
\nc{\od}{{\operatorname {D}}}
\nc{\oD}{{\operatorname {D}}}
\nc{\oK}{{\operatorname {K}}}

\nc{\la}{{\lambda }}

\nc{\Gr}{{\operatorname{Gr}}}
\nc{\LVM}{{\mathbf Q}}
\nc{\DR}{{\operatorname{DR}}}

\nc{\mon}{{\mathit{mon}}}

\newenvironment{rouge}
{\relax\color{red}}
{\hspace*{.3ex}\relax}

\newcommand{\ber}{\begin{rouge}}
\newcommand{\er}{\end{rouge}}

\makeatletter
\newcommand*\bigcdot{\mathpalette\bigcdot@{.5}}
\newcommand*\bigcdot@[2]{\mathbin{\vcenter{\hbox{\scalebox{#2}{$\m@th#1\bullet$}}}}}
\makeatother

\title[Mixed Hodge modules and real groups]{Mixed Hodge modules and real groups}
\author{Dougal Davis}\address[DD, Corresponding author]{School of Mathematics, University of Edinburgh, James Clerk Maxwell Building, Peter Guthrie Tait Road, Edinburgh EH9 3FD, United Kingdom} \address[Present address]{School of Mathematics and Statistics, University of Melbourne, VIC 3010, Australia}\email{dougal.davis1@unimelb.edu.au}
\thanks{DD was supported by the EPSRC programme grant EP/R034826/1 and the starter grant ``Categorified Donaldson-Thomas theory'' No.~759967 of the European Research Council.}

\author{Kari Vilonen}\address[KV]{School of Mathematics and Statistics, University of Melbourne, VIC 3010, Australia, also Department of Mathematics and Statistics, University of Helsinki, Helsinki, Finland}
\email{kari.vilonen@unimelb.edu.au, kari.vilonen@helsinki.fi}
\thanks{KV was supported in part by the ARC grants DP180101445,  FL200100141 and the Academy of Finland}


\begin{document}

\begin{abstract}
Let $G$ be a complex reductive group, $\theta \colon G \to G$ an involution, and $K = G^\theta$. In \cite{schmid-vilonen}, W. Schmid and the second named author proposed a program to study unitary representations of the corresponding real form $G_\mb{R}$ using $K$-equivariant twisted mixed Hodge modules on the flag variety of $G$ and their polarizations. In this paper, we make the first significant steps towards implementing this program. Our first main result gives an explicit combinatorial formula for the Hodge numbers appearing in the composition series of a standard module in terms of the Lusztig-Vogan polynomials. Our second main result is a polarized version of the Jantzen conjecture, stating that the Jantzen forms on the composition factors are polarizations of the underlying Hodge modules. Our third main result states that, for regular Beilinson-Bernstein data, the minimal $K$-types of an irreducible Harish-Chandra module lie in the lowest piece of the Hodge filtration of the corresponding Hodge module. An immediate consequence of our results is a Hodge-theoretic proof of the signature multiplicity formula of \cite{ALTV}, which was the inspiration for this work.
\end{abstract}
\maketitle

\tableofcontents
\section{Introduction}

In this paper, we study mixed Hodge modules on flag varieties, their polarizations, and their applications to the representation theory of reductive Lie groups. Our main results are a mixed Hodge module version of the character identities of Lusztig and Vogan \cite{LV}, an extension of the work of Beilinson and Bernstein on Jantzen filtrations \cite{beilinson-bernstein} to polarized Hodge modules, and a result linking polarizations with the $c$-forms studied in the representation theory of real groups.

The study of representations of real reductive groups via mixed Hodge modules has been proposed by Schmid and the second author~\cite{schmid-vilonen}, who outlined a program using this theory to determine the unitary dual of a reductive Lie group, i.e., the set of its irreducible unitary representations. This program is based on a series of very general conjectures, beyond the scope of representation theory, concerning global sections of twisted mixed Hodge modules on flag varieties. These conjectures reduce the problem of determining unitarity of a representation to the problem of computing the Hodge filtration on the corresponding mixed Hodge module. The latter is expected to be much more amenable to general conceptual arguments, thanks to the deep functoriality properties enjoyed by mixed Hodge modules.

In this paper, we take the first major steps in implementing this program and link it with the work of Adams, van Leeuwen, Trapa and Vogan~\cite{ALTV}, who have made significant progress in understanding the unitary dual using different methods. Those authors present an explicit algorithm, implemented in the ``Atlas of Lie Groups and Representations'' software package \cite{atlas}, that decides whether a given irreducible representation is unitary. A key innovation is the reduction of the question of unitarity to the computation of the signatures of the so-called $c$-forms (particular Hermitian forms introduced by the authors) on the Harish-Chandra modules of irreducible representations with real infinitesimal character. The algorithm computes these signatures by deforming the infinitesimal character to a region where the signatures are known, keeping track of changes along the way using a remarkable signature character formula \cite[Theorem 20.6]{ALTV}. Our results allow the $c$-forms to be deduced from the polarizations of the corresponding Hodge modules, and imply the signature character formula as a corollary. Although in the end we prove much more, our results may be motivated as the minimal package required to deduce these crucial ingredients using Hodge theory. We build on this in the companion paper \cite{DV3} to prove the conjectures in \cite{schmid-vilonen} pertaining to real groups.

We now summarize our main results: a more detailed introduction is given in \S\S\ref{sec:main theorem 1}--\ref{sec:main theorem 3} below. We fix throughout a complex reductive group $G$ and an involution $\theta \colon G \to G$. We write $K = G^\theta$ for the subgroup of fixed points and $\mc{B}$ for the flag manifold of $G$. 

Our first main theorem is a Hodge-theoretic upgrade of the classical Lusztig-Vogan theory, which calculates a change of basis matrix in the Grothendieck group of $K$-equivariant (twisted) $\mc{D}$-modules on $\ms{B}$ \cite{LV}. The formula we obtain can be thought of as a Hodge version of \cite[Theorem 20.6]{ALTV}. For now this is a purely geometric problem; we recall the relevance to representation theory after the statement.

As is well-known, the group $K$ acts on $\mc{B}$ with finitely many orbits. Thus, the Grothendieck group $\mrm{K}(\Mod_K(\mc{D}_{\mc{B}}))$ is a free abelian group of finite rank, with three distinguished bases: a ``standard'' basis $\{[j_!\gamma]\}$, a ``costandard'' basis $\{[j_*\gamma]\}$ and an ``irreducible'' basis $\{[j_{!*}\gamma]\}$, where $\gamma$ runs over the finite set of irreducible $K$-equivariant local systems on $K$-orbits in $\mc{B}$ and $j_!\gamma$ (resp., $j_*\gamma$, $j_{!*}\gamma$) denotes the $!$ (resp., $*$, IC) extension to $\mc{B}$. The original Lusztig-Vogan theory calculates the change of basis between the bases $\{[j_!\gamma]\}$ (or equivalently, $\{[j_*\gamma]\}$) and $\{[j_{!*}\gamma]\}$, a matrix with integer entries given by the multiplicities of the composition factors of the reducible $\mc{D}$-modules $j_!\gamma$. The calculation proceeds (as in Kazhdan-Lusztig theory, which is a special case) by passing to the analogous problem for mixed perverse sheaves over a finite field: in this more structured setting, each composition factor in $j_!\gamma$ appears with a certain \emph{weight}. Recording these weights yields a ``mixed'' change of basis matrix, whose entries are polynomials, which may be calculated using Hecke algebra combinatorics.
 
We perform an analogous calculation over $\mb{C}$ with a category of \emph{mixed Hodge modules} in place of mixed perverse sheaves. More precisely, let $H$ be the abstract Cartan of $G$, $\mf{h} = \mrm{Lie}(H)$ its Lie algebra, and $\lambda \in \mf{h}^*_\mb{R} = \mb{X}^*(H) \otimes_{\mb{Z}} \mb{R}$. Then there is a category $\mhm_\lambda(K \bslash \mc{B})$ of $K$-equivariant $\lambda$-twisted complex\footnote{Unless otherwise specified, we always work with the complex mixed Hodge modules of \cite{SS} and \cite{schmid-vilonen} instead of the more standard rational or real ones.} mixed Hodge modules on $\mc{B}$ (denoted $\mrm{CMHM}(\mc{B})_\lambda$ in \cite{schmid-vilonen}). The objects are equivariant modules $\mc{M}$ over a sheaf $\mc{D}_\lambda$ of twisted differential operators on $\mc{B}$ (see \S\ref{subsec:twisted D-modules}) equipped with extra structures such as a weight filtration $W_{\bigcdot} \mc{M}$ and a Hodge filtration $F_{\bigcdot} \mc{M}$; see \S\ref{subsec:modules on K mod B} for more details. As above, we have Hodge versions of the standard, costandard and irreducible objects $j_!\gamma$, $j_*\gamma$ and $j_{!*}\gamma$, parametrized by twisted local systems on $K$-orbits. (The objects $j_{!*}\gamma$ and $j_*\gamma$ are denoted by $\mc{I}(Q, \lambda, \gamma)$ and $\mc{M}(Q, \lambda, \gamma)$ respectively in \cite{schmid-vilonen}, where $Q$ is the underlying $K$-orbit.) Each composition factor $j_{!*}\gamma'$ of $j_!\gamma$ now appears with some Hodge structure, which may differ from the standard one by tensoring with a $1$-dimensional complex Hodge structure $\mb{C}^{p, q}$ of weight $p + q$. Recording this with a coefficient $t_1^pt_2^q$ defines a Hodge multiplicity polynomial
\[ \LVM^h_{\gamma', \gamma}(t_1, t_2) \in \mb{Z}[t_1^{\pm 1}, t_2^{\pm 1}] \]
characterized by
\[
[j_!\gamma] = \sum_{\gamma'} \LVM^h_{\gamma', \gamma}(t_1, t_2)[j_{!*}\gamma'] \in \oK(\mhm_\lambda(K \bslash \ms{B})),
\]
where the variables $t_1$ and $t_2$ correspond to tensoring with $\mb{C}^{1, 0}$ and $\mb{C}^{0, 1}$ respectively. We also write
\[ \LVM^m_{\gamma', \gamma}(u) := \LVM^h_{\gamma', \gamma}(u^{\frac{1}{2}}, u^{\frac{1}{2}}) \]
for the ``mixed'' polynomials remembering only the weights and forgetting the Hodge structures. Unlike $\LVM^h$, the polynomials $\LVM^m$ are well known: since weight filtrations satisfy the same formal properties for mixed Hodge modules as for mixed sheaves over finite fields, $\LVM^m$ is equal (up to an explicit power of $u$) to the Lusztig-Vogan (aka Kazhdan-Lusztig) multiplicity polynomial $\LVM$ of \cite{LV, ic3, ic4}, see Proposition \ref{prop:mixed to LV}. Our first main theorem explains how to recover $\LVM^h$ from $\LVM^m$.

\begin{thm} \label{thm:intro theorem 1}
We have
\[ \LVM^h_{\gamma', \gamma}(t_1, t_2) = (t_1t_2^{-1})^{\frac{1}{2}(\ell_H(\gamma') - \ell_H(\gamma))}\LVM^m_{\gamma', \gamma}(t_1t_2).\]
Here $\ell_H(\gamma)$ and $\ell_H(\gamma')$ are the explicit Hodge shifts of Definition \ref{defn:hodge shift}.
\end{thm}

Theorem \ref{thm:intro theorem 1} is restated as Theorem \ref{thm:main theorem 1} in the text; we direct the reader to \S\ref{sec:main theorem 1} for a more detailed explanation of the statement and to \S\ref{proof:main theorem 1} for the proof. As well as being an effective computation, the result is striking for the following reason. As a rule, most Hodge structures arising in geometric representation theory are \emph{Tate}: that is, they are extensions of the diagonal Hodge structures $\mb{C}^{n, n}$ for $n \in \mb{Z}$. This property is quite useful: for example, it is used to construct graded lifts in the study of Koszul duality \cite[\S 4.5]{BGS}. If this were the case here, we would have
\[ \LVM^h_{\gamma', \gamma}(t_1, t_2) = \LVM^m_{\gamma', \gamma}(t_1t_2).\]
The extra factor in Theorem \ref{thm:intro theorem 1} indicates that our Hodge structures are \emph{not} Tate, but in a mild way that can be removed with a suitable change of normalization for the Hodge structure on each local system. This observation may be of significance for the development of a Hodge-theoretic approach to Koszul duality for real groups as outlined in \cite[\S 6]{BV}.

Let us now explain what Theorem \ref{thm:intro theorem 1} has to do with representation theory and the program of \cite{schmid-vilonen}. Writing $\mf{g} = \mrm{Lie}(G)$, the Beilinson-Bernstein localization theory defines a global sections functor
\[ \Gamma \colon \Mod_K(\mc{D}_\lambda) \to \Mod(\mf{g}, K)_\lambda\]
from the category of $K$-equivariant $\mc{D}_\lambda$-modules on $\mc{B}$ to the category of $(\mf{g}, K)$-modules on which the center $Z(U(\mf{g}))$ of the universal enveloping algebra acts by a fixed character $\chi_\lambda$ determined by $\lambda \in \mf{h}^*$. We fix our conventions so that $\mc{D}_0 = \mc{D}_{\mc{B}}$ is the sheaf of ordinary (untwisted) differential operators, which corresponds to the infinitesimal character of the trivial representation. Then $\Gamma$ is an exact quotient functor (resp., an equivalence) if $\lambda + \rho$ is integrally dominant (resp., integrally dominant and regular), where $\rho$ is half the sum of the positive roots of $G$. In this case, the global sections of the irreducible $\mc{D}_\lambda$-modules $j_{!*}\gamma$ are either irreducible or zero, and the global sections of the $j_!\gamma$ are the standard $(\mf{g}, K)$-modules arising in the Langlands classification \cite{HMSW}.

Now, if we fix a \emph{real} infinitesimal character $\chi$ then there is a unique $\lambda \in \mf{h}^*_\mb{R}$ such that $\lambda + \rho$ is dominant and $\chi = \chi_\lambda$. We call an irreducible twisted local system $\gamma$ \emph{relevant} if the twist $\lambda$ satisfies this condition and $\Gamma(j_{!*}\gamma) \neq 0$. The global sections functor defines a bijection between the set of relevant local systems and the set of irreducible $(\mf{g}, K)$-modules with real infinitesimal character. Moreover, for $\gamma$ relevant, the Hodge filtration on $j_{!*}\gamma$ coming from its lift to an object in $\mhm_\lambda(K \bslash \mc{B})$ determines a canonical Hodge filtration on the irreducible $(\mf{g}, K)$-module $\Gamma(j_{!*}\gamma)$. The main conjecture in \cite{schmid-vilonen} is that this Hodge filtration controls the signature of a natural Hermitian form coming from the Hodge-theoretic notion of \emph{polarization}.

In classical Hodge theory, a polarization on a pure Hodge structure (i.e., one with a single weight) is a Hermitian form satisfying an explicit signature condition with respect to the Hodge decomposition. Similarly, there is a notion of polarization on a pure Hodge module: a Hermitian form on the underlying $\mc{D}$-module satisfying an analogous sign condition with respect to the Hodge module structure. We refer the reader to \S\ref{subsec: phm} for a more detailed explanation, including the precise definition of a Hermitian form on a $\mc{D}$-module. For now, suffice it to say that an irreducible Hodge module $\mc{M}$ has a unique polarization (up to multiplying by a positive real scalar), and that if $S$ is a polarization on $\mc{M}$ then $(-1)^q S$ is a polarization on $\mc{M} \otimes \mb{C}^{p, q}$. The authors of \cite{schmid-vilonen} observed that the polarization on $j_{!*}\gamma$ can be integrated to a Hermitian form on the associated $(\mf{g}, K)$-module and conjectured that this form also satisfies a natural sign condition with respect to the global sections of the Hodge filtration.

Theorem \ref{thm:intro theorem 1} describes in particular how the Hodge filtrations on irreducible and standard modules are related. In view of the above conjecture, one would hope that this might have consequences for Hermitian forms. Our second main theorem shows that this is indeed the case. To formulate the theorem, we first recall how to define Hermitian forms on composition factors of standard modules.

Suppose that $\gamma$ is an equivariant $\lambda$-twisted local system on a $K$-orbit $Q \subset \ms{B}$ and $f \in \mrm{H}^0(\bar{Q}, \mc{L}_\varphi)^K$ is an equation for the boundary of $Q$. Here $\mc{L}_\varphi \in \mrm{Pic}(\mc{B})$ is the line bundle corresponding to a character $\varphi \in \mb{X}^*(H)$. Then there is an associated $(\lambda + s \varphi)$-twisted deformation $\gamma_{s\varphi}$ of $\gamma$ for all $s \in \mb{R}$ (\S\ref{subsec:mhm deform}), equipped with a morphism
\begin{equation} \label{eq:intro jantzen}
 j_!\gamma_{s\varphi} \to j_*\gamma_{s\varphi}
 \end{equation}
with image $j_{!*}\gamma_{s\varphi}$. The morphism \eqref{eq:intro jantzen} is an isomorphism for generic $s$, so filtering by its order of vanishing at $s = 0$ determines a Jantzen filtration $J_{\bigcdot} j_!\gamma$ indexed in negative degrees. The polarization $S$ on $j_{!*}\gamma$ deforms to a perfect pairing between $j_!\gamma_{s\varphi}$ and $j_*\gamma_{s\varphi}$ for all $s$ and hence determines a Jantzen form $s^{-n}\mrm{Gr}_{-n}^J(S)$ on $\mrm{Gr}_{-n}^Jj_!\gamma$ for all $n \geq 0$ (see \S\ref{subsec:mhm jantzen}).

\begin{thm} \label{thm:intro theorem 2}
For all $n \geq 0$, the object $\mrm{Gr}_{-n}^J j_!\gamma$ is a pure Hodge module, and the Jantzen form $s^{-n}\mrm{Gr}_{-n}^J(S)$ is a polarization.
\end{thm}

This is stated as Theorem \ref{thm:main theorem 2} in the text. The theorem is a special case of a completely general result about extensions of polarized Hodge modules across principal divisors (see Theorem \ref{thm:jantzen polarization}). We give a full explanation of the result in \S\ref{sec:main theorem 2}. The proof is given in \S\ref{proof:main theorem 2}.

Since polarizations change sign by $(-1)^q$ under tensoring with $\mb{C}^{p, q}$, Theorems \ref{thm:intro theorem 1} and \ref{thm:intro theorem 2} together give a formula for the signature of the Jantzen forms on $\Gamma(\Gr^J_{-n} j_!\gamma) = \Gr^J_{-n}\Gamma(j_!\gamma)$ in terms of the signatures of the polarizations on the $\Gamma(j_{!*}\gamma')$. To turn this into a purely representation-theoretic statement, however, we need to identify the integral of the polarization in terms of representation theory. Our final main result lets us do so by linking the Hodge filtration with the structure of the Harish-Chandra module.

\begin{thm} \label{thm:intro theorem 3}
Assume the twisted local system $\gamma$ is relevant. Then all minimal $K$-types of $\Gamma(j_{!*}\gamma)$ lie in the lowest piece of the Hodge filtration.
\end{thm}

This is stated as Theorem \ref{thm:main theorem 3} in the text. We refer to \S\ref{sec:main theorem 3} for further details, including a recollection of Vogan's notion of minimal $K$-type. The proof, which is given in \S\ref{mk}, proceeds by characterizing sections in the lowest piece of the Hodge filtration in terms of their behaviour under deformations and nearby cycles and reducing to known properties of minimal $K$-types. This result removes a long-standing sticking point in implementing the program of \cite{schmid-vilonen}: previously, there was not even a known bound on the minimal $p$ such that $\Gamma(F_p j_{!*}\gamma) \neq 0$.

Our results have the following consequences concerning \cite{ALTV}. First, Theorem \ref{thm:intro theorem 3}, together with a general result on positivity of polarizations (Proposition \ref{prop:hodge positivity}) implies the following corollary, which was obtained up to sign in \cite[Proposition 5.10]{schmid-vilonen}.

\begin{cor} \label{cor:intro 1}
Let $\gamma$ be a relevant twisted local system. Then the $c$-form on $\Gamma(j_{!*}\gamma)$ coincides with the integral of the polarization on $j_{!*}\gamma$.
\end{cor}

This is restated in the text as Corollary \ref{forms coincide}. The new part of the corollary is that the sign of the polarization (which is normalized using Hodge theory) agrees with the sign of the $c$-form (which is normalized using the minimal $K$-types). Combined with Theorem \ref{thm:intro theorem 2}, this implies the following.

\begin{cor} \label{cor:intro 2}
If $\gamma$ and $\gamma'$ are relevant twisted local systems, then
\[ \LVM^c_{\gamma', \gamma}(u, \zeta) = u^{\frac{1}{2}(\ell(\gamma') - \ell_I(\gamma') - (\ell(\gamma) - \ell_I(\gamma)))} \LVM^h_{\gamma', \gamma}(u^{\frac{1}{2}}, \zeta u^{\frac{1}{2}}).\]
\end{cor}

This is also stated as Corollary \ref{cor:hodge signature}. The left hand side is the signature multiplicity polynomial of \cite{ALTV}, which is a generalization of the Lusztig-Vogan multiplicity polynomial $\LVM$ involving an additional parameter $\zeta$ with $\zeta^2 = 1$, keeping track of the signature of the Jantzen forms. The power of $u$ in Corollary \ref{cor:intro 2} is not integral to the theory: it is merely a consequence of the conventions used in the definition of $\LVM^c$. Finally, applying Theorem \ref{thm:intro theorem 1} yields:

\begin{cor}[{\cite[Theorem 20.6]{ALTV}}] \label{cor:intro 3}
We have
\[ \LVM_{\gamma', \gamma}^c(u, \zeta) = \zeta^{\frac{1}{2}(\ell_o(\gamma) - \ell_o(\gamma'))}\LVM_{\gamma', \gamma}(\zeta u),\]
where $\ell_o(\gamma)$ is the orientation number of \cite{ALTV}.
\end{cor}

This is restated as Corollary \ref{ALTV}. We explain these corollaries and their proofs in \S\ref{subsec:ALTV}.

Beyond \cite{schmid-vilonen}, the idea of studying real group representations using Hodge theory has also been considered by Adams, Trapa and Vogan in the draft \cite{ATV}, where a number of conjectures are stated that would enable the Hodge filtrations of irreducible modules $\Gamma(j_{!*}\gamma)$ to be calculated by a version of the unitarity algorithm. (This conjectural algorithm has been implemented in ``Atlas''.) Our results imply some of these conjectures. In particular Theorem \ref{thm:intro theorem 1} is essentially \cite[Conjecture 6.8]{ATV} and Corollary \ref{cor:intro 2} is essentially \cite[Conjecture 7.2 (c)]{ATV}.

Throughout this paper, we will work with algebraic left complex mixed Hodge modules on smooth quasi-projective algebraic varieties, following the treatment in \cite{schmid-vilonen} and \cite{SS}. Where conventions clash, we will generally follow the more complete reference \cite{SS}. We will recall aspects of the theory as needed; the relevant sections are \S\S\ref{subsec:modules on K mod B}, \ref{subsec: phm}, \ref{subsec:hodge filtration}, \ref{subsec:SS MHM} and \ref{subsec:SS polarization}.

The reader may object that, at the time of writing, there is no complete reference for the theory of complex mixed Hodge modules in the literature; \cite{schmid-vilonen} contains only a very rough sketch, and the more comprehensive book project \cite{SS} is still unfinished. To assuage any such concerns, we have included an appendix explaining how to reconstruct the complex theory from Saito's original \cite{S1, S2} by fleshing out the trick outlined in the appendix to \cite{schmid-vilonen}, and providing proofs of the non-trivial facts we require in this setting\footnote{Strictly speaking, only the case of rational $\lambda$ can be accessed directly from Saito's theory, but the modification required for all real $\lambda$ is relatively minor. See \S\ref{subsec:Saito MHM}.}.

The outline of the paper is as follows. In \S\S\ref{sec:main theorem 1}, \ref{sec:main theorem 2} and \ref{sec:main theorem 3}, we recall some background and give the statements of our main results Theorems \ref{thm:intro theorem 1}, \ref{thm:intro theorem 2} and \ref{thm:intro theorem 3} (=\ref{thm:main theorem 1}, \ref{thm:main theorem 2} and \ref{thm:main theorem 3}) respectively. We also discuss some consequences of our results and the connection with \cite{ALTV} in \S\ref{sec:main theorem 3}. We give the proof of Theorem \ref{thm:intro theorem 1} in \S\ref{proof:main theorem 1}, the proof of Theorem \ref{thm:intro theorem 3} in \S\ref{proof:main theorem 3} and the proof of the more technical Theorem \ref{thm:intro theorem 2} in \S\ref{proof:main theorem 2}. Finally, in Appendix \ref{appendix}, we explain how the theory of complex mixed Hodge modules we use here is related to the more well-established theory of real mixed Hodge modules.

\subsection{Acknowledgements}

The authors would like to thank Jeff Adams and David Vogan for an email exchange with the second author in 2019 that planted the idea for this project, and for answering our many questions about real groups since. We also thank Claude Sabbah for help with the correct sign conventions for polarized Hodge modules and the anonymous referees for their valuable comments.

We began working on this project together following a learning seminar at the University of Melbourne, and would therefore like to thank all participants for their involvement.

\section{The first main theorem}
\label{sec:main theorem 1}

In this section, we explain our first main theorem (Theorem \ref{thm:main theorem 1}) in more detail.

We begin by recalling the notions of twisted and monodromic $\ms{D}$-modules in general in \S\ref{subsec:twisted D-modules}, and our conventions in the case of the flag variety in \S\ref{subsec:twisted on flag}. In \S\ref{subsec:modules on K mod B} we review some aspects of the theory of twisted mixed Hodge modules, and the $K$-equivariant ones on the flag variety in particular, and write down the standard and irreducible bases for the Grothendieck group in terms of twisted local systems on $K$-orbits. We recall how these local systems are parametrized in \S\ref{subsec:parametrization}. With the parametrization in hand, we write down the various explicit numbers (length, integral length, orientation number and Hodge shift) associated with these local systems in \S\ref{subsec:numerics}. Finally, in \S\ref{subsec:first main theorem}, we introduce the Hodge, mixed and Lusztig-Vogan multiplicity polynomials, and state Theorem \ref{thm:main theorem 1}.

\subsection{Twisted \texorpdfstring{$\ms{D}$-modules}{D-modules}} 
\label{subsec:twisted D-modules}

In this subsection, we recall briefly the general definition of twisted $\ms{D}$-modules and the monodromic construction of the rings of twisted differential operators of \cite[\S 2.5]{beilinson-bernstein}.

Let $X$ be a smooth complex variety, $H$ an algebraic torus with Lie algebra $\mf{h}$, and $\pi \colon \tilde{X} \to X$ a left $H$-torsor. The sheaf $\ms{D}_{\tilde{X}}$ of differential operators on $\tilde{X}$ is naturally $H$-equivariant; we let $\tilde{\ms{D}}_X$ be the sheaf of algebras on $X$ given by
\[ 
\tilde{\ms{D}}_X = \pi_*(\ms{D}_{\tilde{X}})^H.
\]
Differentiating the action $(h \cdot f)(x) = f(h^{-1}x)$ for $h \in H$ and $f \in \mc{O}_{\tilde{X}}$ gives a map $a \colon \mf{h} \to \ms{T}_{\tilde{X}}$, where $\ms{T}_{\tilde{X}} \subset \ms{D}_{\tilde{X}}$ is the tangent sheaf. As $\fh$ is abelian we will use the {\it inverse} $-a$ to induce a map  $U(\mf{h}) \to \ms{D}_{\tilde{X}}$, which then descends to an algebra homomorphism
\[ 
U(\mf{h}) \to \tilde{\ms{D}}_X \subseteq \pi_*\ms{D}_{\tilde{X}}. 
\]
Given $\lambda \in \mf{h}^*$, we set
\[
 \ms{D}_{X, \lambda} = \tilde{\ms{D}}_X \otimes_{U(\mf{h})} \mb{C}_\lambda,
\]
where $h \in \mf{h}$ acts on $1 \in \mb{C}_\lambda$ by $h \cdot 1 = \lambda(h)$. Note that since $U(\mf{h})$ is central in $\tilde{\ms{D}}_X$ and $U(\mf{h}) \to \mb{C}_\lambda$ is an algebra homomorphism, $\ms{D}_{X, \lambda}$ is a sheaf of algebras on $X$. A $\lambda$-twisted $\ms{D}$-module on $X$ is by definition a $\ms{D}_{X, \lambda}$-module such that the underlying sheaf of $\mc{O}_X$-modules is quasi-coherent.

If $\ms{M}$ is a $\ms{D}_{X, \lambda}$-module, then its $\cO$-module pullback $\pi^\circ\ms{M}$ is naturally a $\ms{D}_{\tilde{X}}$-module equipped with a weak $H$-action (i.e., an action of $H$ by $\ms{D}_{\tilde{X}}$-linear maps) such that
\begin{equation} \label{eq:twisted condition}
h \cdot m + a(h)m + \lambda(h)m = 0
\end{equation}
for $h \in \mf{h}$ and $m \in \pi^\circ\ms{M}$. The category $\Mod_{\ms{D}_{X, \lambda}}$ is naturally identified with such weakly $H$-equivariant $\ms{D}_{\tilde{X}}$-modules; we will pass freely between these two descriptions where convenient.

One may also consider the larger category $\Mod^\mon(\ms{D}_{X, \lambda})$ of \emph{$\lambda$-monodromic $\ms{D}$-modules}; these are defined to be weakly $H$-equivariant $\ms{D}_{\tilde{X}}$-modules such that the operator defined by the left hand side of \eqref{eq:twisted condition} is nilpotent rather than zero. Every monodromic $\ms{D}$-module is an iterated extension of twisted ones, so the two notions are indistinguishable at the level of Grothendieck groups.

Forgetting the $H$-action defines a full faithful functor \cite[Lemma 2.5.4 (i)]{beilinson-bernstein}
\begin{equation} \label{eq:twisted to monodromic}
 \od^b(\Mod^\mon(\ms{D}_{X, \lambda})) \to \od^b(\Mod(\ms{D}_{\tilde{X}})).
\end{equation}
Where convenient, we will identify $\Mod(\ms{D}_{X, \lambda})$ and $\Mod^\mon(\ms{D}_{X, \lambda})$ with their images under \eqref{eq:twisted to monodromic}.

\subsection{Twisted \texorpdfstring{$\ms{D}$-modules}{D-modules} on the flag variety}
\label{subsec:twisted on flag}

The flag variety $\ms{B}$ of $G$ is defined to be the variety parametrizing Borel subgroups of $G$. Given a point $x \in \ms{B}$, write $B_x$ for the corresponding Borel subgroup. There is a canonical $G$-action on $\ms{B}$ given by $B_{gx} = gB_xg^{-1}$ for $g \in G$ and $x \in \ms{B}$, which has the property that $B_x = \mrm{Stab}_G(x)$ for all $x \in \ms{B}$. The $G$-action on $\ms{B}$ is transitive, so for any choice of base point $x$ corresponding to a choice of Borel $B = B_{x}$, the map
\begin{align*}
G/B &\to \ms{B} \\
gB &\to g \cdot x
\end{align*}
is an isomorphism. The universal Borel subgroup is the group scheme over $\ms{B}$ given by
\[ B^{uni} = \{(g, x) \in G \times \ms{B} \mid gx = x\} \to \ms{B}. \]
If we write $N^{uni} \subset B^{uni}$ for the unipotent radical, then the universal Cartan $H$ is the unique algebraic torus such that $B^{uni}/N^{uni} \cong H \times \ms{B}$ as group schemes over $\ms{B}$. In particular, $H$ is equipped with a canonical isomorphism $H \cong B/N$ for every Borel subgroup $B \subset G$ with unipotent radical $N$. The torus $H$ comes equipped with a canonical root datum $\Phi \subset \mb{X}^*(H)$, $\Phi^\vee \subset \mb{X}_*(H)$ and a choice $\Phi = \Phi_+ \cup \Phi_-$ of positive and negative roots as follows. For any maximal torus $T$ and choice of a fixed point $x$ of the $T$-action on $\ms{B}$ we obtain an isomorphism $\tau_x\colon T \to H$. The roots are the non-zero weights of $T$ acting on $\mf{g} = \mrm{Lie}(G)$, and the negative roots are those appearing in $\mf{b}_x = \mrm{Lie}(B_x)$.

If $\ms{L}$ is a $G$-equivariant line bundle on $\ms{B}$, then the group scheme $B^{uni}$ acts on $\ms{L}$ by restriction of the $G$-action. Since $\ms{L}$ is a line bundle, the universal unipotent radical $N^{uni}$ acts trivially, so the $B^{uni}$-action factors through an $H$-action on $\ms{L}$ over $\ms{B}$. This construction defines an isomorphism of abelian groups
\[ \mrm{Pic}^G(\ms{B}) \overset{\sim}\to \mb{X}^*(H) \]
sending a $G$-equivariant line bundle $\ms{L}$ to the character of $H$ acting on $\ms{L}$. Write $\lambda \mapsto \ms{L}_\lambda$ for the inverse of this isomorphism. Then there is a canonical $G$-equivariant right $H$-torsor $\tilde{\ms{B}} \to \ms{B}$ defined by the property that $\ms{L}_\lambda = \tilde{\ms{B}} \times^H \mb{C}_\lambda$. If we choose a Borel subgroup $B \subset G$ then we can write $\ms{L}_\lambda$ explicitly as
\[ \ms{L}_\lambda = G \times^B \mb{C}_\lambda = G/N \times^H \mb{C}_\lambda \to G/B \cong \ms{B},\]
so $\tilde{\ms{B}} \cong G/N$. The $\tilde{\ms{B}}$ is often called the base affine space or the enhanced flag variety. 

Let us write $\mf{h}$ for the Lie algebra of $H$ and
\[ \mf{h}^*_\mb{R} = \mb{X}^*(H) \otimes_\bZ \mb{R} \subset \mf{h}^* = \mb{X}^*(H) \otimes_\bZ \mb{C}. \]
The $H$-torsor $\tilde{\ms{B}} \to \ms{B}$ (with left $H$-action given by $h \cdot x := xh^{-1}$) defines a sheaf of twisted differential operators $\ms{D}_{\ms{B}, \lambda}$ on $\ms{B}$ for every $\lambda \in \mf{h}^*$. According to our conventions, the sheaves $\ms{D}_{\ms{B}, \lambda}$ are defined so that $\ms{L}_\lambda = \pi_*(\mc{O}_{\tilde{\ms{B}}} \otimes \mb{C}_\lambda)^H$ is a $\ms{D}_{\ms{B}, \lambda}$ module for all $\lambda \in \mb{X}^*(H)$. Note that this convention differs from the one in \cite{schmid-vilonen} by a $\rho$-shift: the sheaf $\ms{D}_{\ms{B}, \lambda}$ in our notation is the sheaf $\ms{D}_{\lambda + \rho}$ in the notation of that paper. With our conventions the $\ms{L}_\lambda$ are ample for $\la \in \mb{X}^*(H)$ regular dominant. 

\subsection{Twisted mixed Hodge modules on the flag variety}
\label{subsec:modules on K mod B}

We next review the equivariant twisted mixed Hodge modules on $\ms{B}$.

Recall that we work always with complex mixed Hodge modules. A complex mixed Hodge module on a point is a complex mixed Hodge structure, i.e., a finite dimensional complex vector space $V$ equipped with three finite increasing filtrations $(W_{\bigcdot}, F_{\bigcdot}, \bar{F}_{\bigcdot})$ satisfying
\[ \Gr^W_n V = \bigoplus_{p + q = -n} F_p \Gr^W_n V \cap \bar{F}_q \Gr^W_n V.\]
As in the more traditional case of real or rational Hodge structures, in case $V$ is pure of weight $w$, i.e., $\mrm{Gr}^W_w V = V$, the Hodge structure is equivalent to a decomposition
\[ V = \bigoplus_{p + q = -w} V_{p, q} \quad \text{where} \quad V_{p, q} = V^{-p, -q} = F_p V \cap \bar{F}_q V.\]
We write $h_{p,q}=\dim V_{p,q}$ for the Hodge numbers and $\bC_{p,q} = \bC^{-p, -q}$ for the one dimensional Hodge structure with $h_{p,q}=1$. We write $(1)$ for the Tate twists, i.e., for the operation of tensoring by $\bC_{1,1}$ and $(m)$ for its iterates.

For a general smooth quasi-projective variety $X$, a complex mixed Hodge module on $X$ consists of an algebraic regular holonomic $\ms{D}_X$-module $\ms{M}$ (in the sense of Beilinson-Bernstein, see, for example, the book~\cite{borel:1987}) equipped with extra data including a weight filtration $W_{\bigcdot}$ by $\ms{D}_X$-submodules and a Hodge filtration $F_{\bigcdot}$ compatible with the filtration on $\ms{D}_X$ by order of differential operator. There is also a more subtle conjugate Hodge filtration $\bar{F}_{\bigcdot}$, the precise discussion of which we defer to \S\ref{subsec:SS MHM} (as the weight and Hodge filtrations are sufficient for $\oK$-group calculations). The category $\mhm(X)$ of mixed Hodge modules on $X$ is naturally tensored over the category $\mhm(\spec \mb{C})$ of mixed Hodge structures.

Now suppose $\tilde{X} \to X$ is an $H$-torsor and $\lambda \in \mf{h}^*_\mb{R}$. Then we have a category $\mhm_\lambda(X) \subset \mhm(\tilde{X})$ of $\lambda$-twisted mixed Hodge modules on $X$ defined as the full subcategory of $\mhm(\tilde{X})$ consisting of objects whose underlying $\ms{D}$-module is the pullback of a $\lambda$-twisted $\ms{D}$-module on $X$. (If $\lambda \not\in \mf{h}^*_\mb{R}$, this category is zero, since the monodromy of a mixed Hodge module around a copy of $\mb{C}^\times$ always has eigenvalues of absolute value $1$---this is an immediate consequence, for example, of the ``quasi-unipotent and regular'' condition \cite[(5.1.5.1)]{S1} in Saito's formulation, or the ``$\mb{R}$-specializable'' condition \cite[Definition 14.2.2 (2)]{SS} in Sabbah and Schnell's.) Note that when the torsor $\tilde{X}$ is trivial, we have $\mhm_\lambda(X) \cong \mhm(X)$ by results of T. Saito \cite{S3}, so this definition is sensible. If $K$ acts on $\tilde{X} \to X$, then we also have corresponding categories
\[ \mhm_\lambda(K \bslash X) \subset \mhm(K \bslash \tilde{X}) \]
of $K$-equivariant (twisted) mixed Hodge modules.

Return now to the setting of a reductive group $G$, $\theta \colon G \to G$ an involution, and $K = G^\theta$ the subgroup of fixed points. Assuming that $\lambda \in \mf{h}^*_\mb{R}$ is real, let us consider the category 
\[ 
\mhm_\lambda(K \bslash \ms{B})
 \]
of $K$-equivariant $\lambda$-twisted mixed Hodge modules on $\ms{B}$. Unless otherwise specified, we will always assume $\lambda \in \mf{h}^*_\mb{R}$ from now on.

Given a $K$-orbit $Q \subset \ms{B}$ and a $K$-equivariant $\lambda$-twisted local system (i.e., a vector bundle with a $\ms{D}_\lambda$-structure) $\gamma$ on $Q$, the corresponding local system on the preimage $\tilde Q$ in $\tilde{\mc{B}}$ is necessarily unitary. Indeed, if we fix $\tilde{x} \in {\tilde Q}$ then the pullback of $\gamma$ along the action map $K \times H \to \tilde Q$ is unitary since it is $K$-equivariant and $\lambda \in \mf{h}^*_\mb{R}$. Any flat inner product automatically descends to $(K \times H)/\mrm{Stab}_{K \times H}(\tilde{x})^\circ$, so averaging with respect to the finite group $\pi_0(\mrm{Stab}_{K \times H}(\tilde x))$ produces a flat inner product on $\gamma$. We may therefore regard $\gamma$ as a Hodge module so that the pointwise Hodge numbers satisfy $h_{p, q} = 0$ for $(p, q) \neq (0, 0)$; in particular, $F_{-1}\gamma = 0$ and $F_0 \gamma = \gamma$. We will follow this convention throughout the paper. 

\begin{rmk} \label{rmk:weights}
As an artefact of our decision to define twisted mixed Hodge modules as mixed Hodge modules on an $H$-torsor, the object $\gamma \in \mhm_\lambda(Q)$ as defined above has weight $\dim \tilde{Q} = \dim Q + \dim H$.
\end{rmk}

For each $K$-orbit $Q$, write $j_Q \colon Q \to \ms{B}$ for the inclusion. The general theory of mixed Hodge modules equips us with functors of $!$, $*$ and intermediate pushforward
\[ 
j_{Q !}, j_{Q *}, j_{Q !*} \colon \mhm_\lambda(K \bslash Q) \to \mhm_\lambda(K \bslash \ms{B}),
\]
where $j_{Q !*}(\ms{M}) := \mrm{im}(j_{Q!}\ms{M} \to j_{Q*}\ms{M})$. (Really, these are defined as pushforwards along the corresponding map of $H$-torsors.) Recall  the inclusion $Q \hookrightarrow \ms{B}$ is affine and so the functor $j_{Q*}$, and hence $j_{Q!}$, do not require passage to the derived category. To shorten the notation we will from now on simply write $j_{!}\gamma$ for $j_{Q !}\gamma$, and similarly for the other functors if $\gamma$ is a twisted local system on $Q$.

Observe that the Grothendieck group of mixed Hodge structures is given by
\[
 \oK(\mhm(\spec \mb{C})) = \mb{Z}[t_1^{\pm 1}, t_2^{\pm 1}],
 \]
where we have written $t_1 = [\mb{C}_{-1, 0}]$ and $t_2 = [\mb{C}_{0, -1}]$.

\begin{prop} \label{prop:bases}
The Grothendieck group $\oK(\mhm_\lambda(K \bslash \ms{B}))$ is a free module over $\mb{Z}[t_1^{\pm 1}, t_2^{\pm 1}]$. It has three distinct bases
\[ 
\{[j_{!} \gamma]\} , \quad \{[j_{*}\gamma]\}, \quad \text{and} \quad \{[j_{!*}\gamma]\},
\]
where in each case $\gamma$ ranges over irreducible $\lambda$-twisted local systems on $K$-orbits of $\ms{B}$.
\end{prop}
We will write $\LVM^h$ for the change of basis matrix expressing the basis $\{[j_{!}\gamma]\}$ in terms of the basis $ \{[j_{!*}\gamma]\}$, i.e., we have
\[
[j_!\gamma] = \sum_{\gamma'} \LVM^h_{\gamma',\gamma} [j_{!*}\gamma']\,.
\]
The explicit determination of this matrix is the main result of this section.

\subsection{Parametrization of the local systems} \label{subsec:parametrization}

Let us recall the following combinatorial parametrization of $K$-equivariant $\la$-twisted local systems on orbits in $\ms{B}$. Consider a $K$-orbit $Q$ and a $K$-equivariant $\la$-twisted local system $\gamma$ on it. Choose a point $x\in Q$ and a $\theta$-stable torus $T\subset B_x$. The identification $\tau_x \colon T \to H$ defines an involution
\[ \theta_Q = \theta_x := \tau_x\theta\tau_x^{-1} \colon H \to H,\]
independent of $x \in Q$. Consider the Harish-Chandra pair $(\mf{h}, H^{\theta_Q}) \cong (\mf{h}, T^\theta)$. The morphism $\tau_x$ induces an isomorphism between $(B_x \cap K)/(N_x \cap K)^\circ$ and $H^{\theta_Q}$, so $\gamma$ is determined by the representation $(\lambda,\Lambda)$ of $(\fh, H^{\theta_Q})$ given by the action of $T^\theta$ on the fiber of $\gamma$ over $x$. Thus

\begin{prop} \label{prop:orbit Hodge modules}
The category $K$-equivariant $\lambda$-twisted local systems on the $K$-orbit $Kx$ is equivalent to the category of $(\mf{h}, H^{\theta_x})$-modules where $\mf{h}$ acts by $\lambda$.
\end{prop}

Note that since $H^{\theta_x}$ is commutative, Proposition \ref{prop:orbit Hodge modules} implies that all irreducible $\gamma$ are of rank $1$. When a twisted local system $\gamma$ arises from a character $(\lambda, \Lambda)$ in this way, we write
\[
\gamma = \cO_Q(\la, \Lambda) = \cO(\la,\Lambda,x)\,.
\]
Note that when we regard $\Lambda$ as a character of the subgroup $H^{\theta_Q}$ of the universal Cartan, $\mc{O}(\lambda, \Lambda, x)$ depends on $x$ only through $Q = Kx$.

\subsection{Lengths, orientation numbers and Hodge shifts} \label{subsec:numerics}

In this subsection, we recall the definitions of the length $\ell(\gamma)$, the integral length $\ell_I(\gamma)$, and the orientation number $\ell_o(\gamma)$ of \cite{ALTV} associated with a twisted local system $\gamma$. We also give the definition of the closely related quantity $\ell_H(\gamma)$ appearing in Theorem \ref{thm:main theorem 1}.

Recall first the following classification of roots according to the behaviour of $\theta$ with respect to a fixed orbit. 

\begin{defn} \label{defn:root classification}
Let $Q \subset \ms{B}$ be a $K$-orbit. Fix $x \in Q$, a $\theta$-stable torus $T$ fixing $x$, identified with $H$ via $\tau_x$, and write $\mf{g} = \mf{h} \oplus \bigoplus_{\alpha \in \Phi} \mf{g}_\alpha$ for the corresponding root space decomposition. For $\alpha \in \Phi$, we say that
\begin{enumerate}
\item $\alpha$ is \emph{real} if $\theta\alpha = -\alpha$,
\item $\alpha$ is \emph{compact imaginary} if $\theta\alpha = \alpha$ and $\theta|_{\mf{g}_\alpha} = +1$,
\item $\alpha$ is \emph{non-compact imaginary} if $\theta\alpha = \alpha$ and $\theta|_{\mf{g}_\alpha} = -1$, and
\item $\alpha$ is \emph{complex} if $\theta\alpha \not\in \{\pm \alpha\}$.
\end{enumerate}
These notions depend only on $Q$ (and not on $x$). We will sometimes say that $\alpha$ is \emph{$Q$-real} etc to emphasize the dependence on the orbit.
\end{defn}

In the definitions below, for a fixed $\lambda \in \mf{h}^*$, we say that a root $\alpha \in \Phi$ is \emph{$\lambda$-integral} (or simply \emph{integral}) if $\langle \lambda, \check\alpha \rangle \in \mb{Z}$.

\begin{defn} \label{defn:hodge shift}
Let $Q \subset \ms{B}$ be a $K$-orbit and $\gamma = \mc{O}_Q(\lambda, \Lambda)$ an equivariant twisted local system on $Q$ with $\lambda \in \mf{h}^*_\mb{R}$.
\begin{enumerate}
\item The \emph{length} of $\gamma$ is
\[ \ell(\gamma) = \dim Q.\]
\item The \emph{integral length} of $\gamma$ is
\begin{align*}
\ell_I(\gamma) &= \#\{\mbox{$\alpha \in \Phi_+$ $\lambda$-integral}\} - \#\left\{\begin{matrix}\mbox{$\alpha \in \Phi_+$ non-compact} \\ \mbox{imaginary}\end{matrix}\right\} \\
& \quad - \frac{1}{2}\#\left\{\begin{matrix}\mbox{$\alpha \in \Phi_+$ complex $\lambda$-integral}\\\mbox{such that $\theta_Q \alpha \in \Phi_+$}\end{matrix}\right\}.
\end{align*}
\item The \emph{orientation number} \cite[Definition 20.5]{ALTV} of $\gamma$ is
\begin{align*}
 \ell_o(\gamma) &= \#\left\{\begin{matrix}\mbox{$\alpha \in \Phi_+$ real non-integral such that} \\ (-1)^{\lfloor \langle \lambda + \rho_{\mb{R}}, \alpha^\vee\rangle\rfloor} = \Lambda(m_\alpha)\end{matrix}\right\} \\
&\qquad \qquad + \frac{1}{2}\#\left\{\begin{matrix}\mbox{$\alpha \in \Phi_+$ complex non-integral} \\ \mbox{such that $-\theta_Q \alpha \in \Phi_+$}\end{matrix}\right\},
\end{align*}
where $m_\alpha = \check{\alpha}(-1) \in H^{\theta_Q}$, $\rho_{\mb{R}} = \rho_{Q, \mb{R}}$ is half the sum of the positive $Q$-real roots, and $\lfloor a \rfloor$ denotes the largest integer less than or equal to the real number $a$, as usual. 
\item The \emph{Hodge shift} of $\gamma$ is
\[ \ell_H(\gamma) = \frac{1}{2}\sum_{\substack{\alpha \in \Phi_+ \\ \text{real non-integral}}} (-1)^{\lfloor \langle \lambda + \rho_\mb{R}, \check\alpha\rangle\rfloor}\Lambda(m_\alpha). \]
\end{enumerate}
\end{defn}

\begin{rmk}
Regarding integral length, if $\lambda$ is integral (i.e., if all roots are $\lambda$-integral), then $\ell_I(\gamma) = \ell(\gamma) = \dim Q$. The formula for $\ell_I(\gamma)$ is obtained by writing down a formula for $\dim Q$ in terms of root counts and discarding all non-integral roots. Note also that our definition of integral length differs from the one in~\cite[Definition 18.1]{ALTV} by a constant. Only differences of integral lengths appear in our statements, however, so the constant does not play any role. 
\end{rmk}

The quantities of Definition \ref{defn:hodge shift} are related as follows.

\begin{prop} \label{prop:orientation vs hodge}
We have
\[ \ell_H(\gamma) = \ell_o(\gamma) + \ell_I(\gamma) - \ell(\gamma) + \frac{1}{2}\#\{\alpha \in \Phi_+ \text{non-integral}\}.\]
\end{prop}
\begin{proof}
Write out each term of the right hand side in terms of root counting and cancel.
\end{proof}

\subsection{The first main theorem} \label{subsec:first main theorem}

Let us consider the change of basis matrix $\LVM^h$ between the bases $\{[j_{!}\gamma]\}$ and $\{[j_{!*}\gamma]\}$ in $\oK(\mhm_\lambda(K \bslash \ms{B}))$. If we disregard the Hodge structure and only consider the weights we obtain a ``mixed'' change of basis matrix
\[ \LVM^m_{\gamma', \gamma}(u) := \LVM^h_{\gamma', \gamma}(u^{\frac{1}{2}}, u^{\frac{1}{2}}) \in \mb{Z}[u^{\pm \frac{1}{2}}, u^{\pm \frac{1}{2}}].\]
Such a change of basis matrix has been calculated explicitly by Lusztig and Vogan~\cite{LV} for $\lambda \in \mf{h}^*_\mb{Z}$ integral (working with mixed sheaves over a finite field instead of mixed Hodge modules) and implicitly by Adams, Barbasch and Vogan \cite{ABV} for general $\lambda$.

\begin{prop}[{cf., \cite[Theorem 16.22]{ABV}}] \label{prop:mixed to LV}
The matrix $\LVM^m$ is given by
\[
\LVM^m_{\gamma', \gamma}(u)=
u^{\frac{1}{2}(\ell_I(\gamma') - \ell(\gamma') - (\ell_I(\gamma) - \ell(\gamma)))}\LVM_{\gamma', \gamma}(u) \,,
\]
where $\LVM \in \mb{Z}[u]$ are the Lusztig-Vogan multiplicity polynomials.
\end{prop}

The matrix $\LVM$ of Lusztig-Vogan multiplicity polynomials is by definition the inverse to the matrix of Lusztig-Vogan polynomials defined by Lusztig and Vogan in \cite{LV, ic3, ic4}.\footnote{Vogan uses the term Kazhdan-Lusztig polynomials, by analogy with the case of a complex group.} The matrix $\LVM$ can be regarded as a mixed change of basis matrix for the endoscopic group determined by the $\lambda$-integral coroots. Note that for non-integral $\lambda$, $\LVM^m$ need only be a polynomial in $u^{\frac{1}{2}}$.

Our main result tells us how to pass from $\LVM^m$  to $\LVM^h$:
\begin{thm} 
\label{thm:main theorem 1}
For any $\lambda \in \mf{h}^*_\mb{R}$, we have
\[
\LVM^h_{\gamma', \gamma}(t_1, t_2) =  (t_1t_2^{-1})^{\frac{1}{2}(\ell_H(\gamma')-\ell_H(\gamma))}\LVM^m_{\gamma', \gamma}(t_1t_2) \,.
\]
\end{thm}
We defer the proof of this theorem to \S\ref{proof:main theorem 1}.

\begin{rmk}
As suggested by Theorem \ref{thm:main theorem 1}, we will often identify the parameter $u$ in $\LVM^m$ with the class of the inverse Tate structure
\[ u = [\mb{C}_{-1, -1}] = t_1t_2 \in \oK(\mhm(\spec \mb{C})).\]
\end{rmk}

\section{The second main theorem}
\label{sec:main theorem 2}

In this section, we discuss our second main theorem (Theorem \ref{thm:main theorem 2}) in some detail.
We begin in \S\ref{subsec: phm} by recalling the notion of polarization on a complex Hodge module. In \S\ref{subsec:mhm deform} we briefly review the deformations of twisted local systems. We then recall Beilinson and Bernstein's result on the Jantzen filtration for the deformation associated with a boundary equation, and state our Theorem \ref{thm:main theorem 2} extending this to polarized Hodge modules in \S\ref{subsec:mhm jantzen}.

\subsection{Polarized Hodge modules} \label{subsec: phm}

In this subsection, we briefly summarize the theory of polarizations on complex mixed Hodge modules, and spell out explicitly the structure of the polarized Hodge modules $j_{!*}\gamma$ on the flag variety.

Let $X$ be a smooth variety. Recall that the analytification functor $\ms{M} \mapsto \ms{M}^{an}$ realizes the algebraic regular holonomic $\ms{D}_X$-modules as the full subcategory of analytic regular holonomic $\ms{D}$-modules $\ms{M}$ such that the perverse sheaf $\DR(\ms{M})$ has algebraically constructible cohomology (see for example \cite[Proposition 7.8]{brylinski}). Given a regular holonomic $\ms{D}_X$-module $\ms{M}$, its Hermitian dual $\ms{M}^h$ is the unique regular holonomic $\ms{D}_X$-module with analytification
\[ (\ms{M}^h)^{an} = \shom_{\ms{D}_{\bar{X}}}(\overline{\ms{M}}, \ms{D}\mit{b}_{X}). \]
Here we write $\Db_X$ for the sheaf of complex-valued distributions (defined as duals of compactly supported top degree differential forms) on $X$ viewed as a real manifold, and $\bar{X}$ for the conjugate complex structure on $X$. Note that $\ms{M}^h$ is indeed a regular holonomic $\ms{D}_X$-module, satisfying $\mb{D}\overline{\DR(\cM)} = \DR(\cM^h)$ by~\cite{kashiwara2}. If $\ms{M}$ and $\ms{M}'$ are regular holonomic $\ms{D}_X$-modules, then a sesquilinear pairing between $\ms{M}$ and $\ms{M}'$ is a morphism $\ms{M} \to (\ms{M}')^h$, or equivalently, a $\ms{D}_X \otimes \ms{D}_{\bar{X}}$-linear map
\[ \ms{M} \otimes \overline{\ms{M}'} \to \Db_X.\]

The operation of Hermitian duality lifts to complex mixed Hodge modules, sends pure Hodge modules of weight $w$ to pure Hodge modules of weight $-w$, and interchanges the $!$ and $*$ functors. If $\ms{M}$ is a pure Hodge module of weight $w$ then a Hermitian isomorphism $S \colon \ms{M} \to \ms{M}^h(-w)$ (i.e., one satisfying $S^h = S$) is called a polarization if the pair $(\ms{M}, S)$ 
satisfies the inductive conditions of \cite[Definition 14.2.2]{SS}. When $X$ is a point and $\ms{M} = V$ is a Hodge structure, $S$ may be identified with a Hermitian form on $V$ with respect to which the decomposition $V= \oplus  V_{p,q}$ is orthogonal, and the condition for $S$ to be a polarization is
\begin{equation} \label{eq:polarization normalization}
(-1)^q S|_{V_{p,q}} \qquad \text{is positive definite}\,.
\end{equation}
More generally, if $\ms{M} = \ms{V}$ is a variation of Hodge structure on $X$, then the sesquilinear pairing underlying $S$ is necessarily valued in smooth functions on $X$, and $S$ is a polarization if and only if the pairing
\[ S_x \colon \ms{V}_x \otimes \overline{\ms{V}_x} \to \mb{C} \]
is a polarization of the Hodge structure $\ms{V}_x$ for all $x \in X$.

More generally still, suppose that $j \colon Q \to X$ is a locally closed immersion and $\ms{V}$ is a variation of Hodge structure on $Q$. Then the polarizations on the pure Hodge module $j_{!*}\ms{V}$ are in bijection with the polarizations on $\ms{V}$ as follows. Fix a polarization $S_Q$ of $\ms{V}$ (valued in smooth functions on $Q$ as above), and factor $j \colon Q \to X$ as $j' \circ i$, where $i \colon Q \to U$ is a closed immersion and $j' \colon U \to X$ is an open immersion, which we will assume for simplicity is the complement of a divisor.

First, form the closed pushforward $i_*\ms{V}$, which is defined by
\[ i_*\ms{V} = i_{\bigcdot}(\ms{D}_{U \leftarrow Q} \otimes_{\ms{D}_Q} \ms{V}),\]
where the pushforward on the right is the usual sheaf-theoretic pushforward, and $\ms{D}_{U \leftarrow Q}$ is the $(i^{-1}\ms{D}_U, \ms{D}_Q)$-bimodule
\[ \ms{D}_{U \leftarrow Q} = i^{-1}\ms{D}_U \otimes_{i^{-1}\ms{O}_U} \omega_{Q/U} \]
for $\omega_{Q/U} = \det N_{Q/U}$ the determinant of the normal bundle. The polarization $S_Q$ induces a polarization $i_*(S_Q)$ on $i_*\ms{V}$ determined explicitly by the formula
\begin{equation} \label{eq:polarization closed pushforward}
\begin{aligned}
\langle \eta, i_*S_Q(\xi \otimes m, \overline{\xi' \otimes m'})\rangle &= (-1)^{c(c - 1)/2}(2\pi i)^c\langle  (\xi \wedge \overline{\xi'})\contract \eta |_Q, S_Q(m, \overline{m'}) \rangle \\
& := (-1)^{c(c - 1)/2}(2 \pi i)^c \int_{\tilde{Q}} (\xi \wedge \overline{\xi'}) \contract \eta|_Q \cdot S_Q(m, \overline{m'})
\end{aligned}
\end{equation}
for a compactly supported top form $\eta$ on $U$, $\xi, \xi' \in \omega_{Q/U}$, $m, m' \in \ms{V}$ and $c = \codim Q$. Here we define the contraction with exterior products of vector fields by
\[ (\xi_1 \wedge \xi_2 \wedge \cdots \wedge \xi_{2k}) \contract \eta = \xi_1 \contract \xi_2 \contract \cdots \contract \xi_{2k} \contract \eta \]
for vector fields $\xi_i$, where
\[ \xi_i \contract (\eta_1 \wedge \eta_2 \wedge \cdots \wedge \eta_m) := \sum_j (-1)^{j - 1} \langle \xi_i, \eta_j \rangle \eta_1 \wedge \cdots \wedge \widehat{\eta_j} \wedge \cdots \wedge \eta_m\]
for $1$-forms $\eta_j$. The sign in \eqref{eq:polarization closed pushforward} comes from unpacking \cite[\S\S 0.2, 12.3.3, 12.4.a]{SS}; it is arranged so that if $\eta$ is a volume form on $U$ and $\xi \in \omega_{Q/U}$ is a holomorphic section, then $(-1)^{c(c - 1)/2}(2\pi i)^c(\xi \wedge \overline{\xi})\contract \eta |_Q$ is a volume form on $Q$.

The induced polarization $S = j_{!*}S_Q$ of the Hodge module $j_{!*}\ms{V} = j'_{!*}i_*\ms{V}$ is given implicitly as the unique Hermitian form restricting to $i_*S_Q$ on $U$. Note that the existence and uniqueness of such a form follows formally from Kashiwara's theorem \cite{kashiwara2} that Hermitian duality is an anti-equivalence of the category of regular holonomic $\mc{D}$-modules commuting with restriction to opens. To compute it explicitly, we work locally and assume $U = f^{-1}(\mb{C}^\times)$ for some regular function $f \colon X \to \mb{C}$. For a top form $\eta$ and $m \otimes \xi, m' \otimes \xi' \in j_{!*}\ms{V} \subset j_*\ms{V}$, we may then consider the function
\begin{equation} \label{eq:analytic continuation}
 s \mapsto \langle \eta |f|^{2s}, (i_*S)(m \otimes \xi, \overline{m' \otimes \xi'}) \rangle,
\end{equation}
which is a well-defined analytic function for $\operatorname{Re} s \gg 0$. By \cite[Proposition 12.5.4]{SS}, \eqref{eq:analytic continuation} may be analytically continued to a meromorphic function of $s \in \mb{C}$; the unique extension $S = j_{!*}S_Q$ is given by
\begin{equation} \label{eq:intermediate polarization}
\begin{aligned}
\langle \eta, j_{!*}S_Q(m \otimes \xi, &\overline{m' \otimes \xi'})\rangle = \Res_{s = 0}s^{-1}\langle \eta |f|^{2s}, i_*S_Q(m \otimes \xi, \overline{m' \otimes \xi'})\rangle  \\
&= (-1)^{c(c - 1)}(2 \pi i)^c \Res_{s = 0} s^{-1} \int_{\tilde{Q}} (\xi \wedge \overline{\xi'}) \contract \eta|_Q \cdot |f|^{2s}S_Q(m, \overline{m'}).
\end{aligned}
\end{equation}
This is clearly a Hermitian form extending $i_*S_Q$, so it is the unique such.

In the setting of an $H$-torsor $\tilde{X} \to X$ and $\lambda \in \mf{h}^*_\mb{R}$, we recall that the $\lambda$-twisted mixed Hodge modules on $X$ are defined as a full subcategory of the mixed Hodge modules on $\tilde{X}$, and define polarizations on these accordingly. Note that under our assumption that $\lambda$ is real, the Hermitian dual of a $\lambda$-twisted $\ms{D}$-module is again a $\lambda$-twisted $\ms{D}$-module.

Now consider $\gamma$ a rank one $K$-equivariant $\lambda$-twisted local system on an orbit $Q \subset \cB$. Recall that we regard the rank one local system $\gamma$ on $Q$ as variation of Hodge structure on the $H$-bundle $\tilde{Q} \to Q$ of Hodge type $(0,0)$, i.e., $F_{-1}\gamma = 0$ and $F_0\gamma = \gamma$. As a variation of Hodge structure, it carries a unique polarization $S_Q$ up to a positive scalar, which takes values in smooth functions on $\tilde{Q}$ as above. Locally on a sufficiently small open subset $U$ of $\tilde{Q}$ (in the analytic topology), the restriction $\gamma|_U$ can be identified with the $\cD_U$-module $\cO_U$. Under such an identification the polarization $S_Q$ is given by 
\[
S_Q (f_1,\bar{f}_2)\ = \ f_1\bar f_2 \qquad \text{for} \ \ f_1,f_2\in \cO_U\,.
\]
The polarization of the Hodge module $j_{!*}\gamma$ is $S = j_{!*}S_Q$, as in the general construction above.

\subsection{Deformations of twisted local systems} \label{subsec:mhm deform}

In this subsection we recall the space of deformations of the twisted local systems on $K$-orbits in $\ms{B}$ in terms of both parameters and geometry.

Let us consider a $K$-orbit $Q$ on $\cB$ and recall from \S\ref{subsec:parametrization} the associated involution $\theta = \theta_Q \colon H \to H$ and the bijection between one dimensional Harish-Chandra $(\mf{h}, H^\theta)$-modules and twisted $K$-equivariant rank one $\cD$-modules $\gamma$ on the orbit $Q$. For $\lambda\in \fh^*$ and $\Lambda$ a character of $H^\theta$ we wrote $\gamma \ = \ \cO_Q(\la,\Lambda)$ for the corresponding $\cD_\lambda$-module. Let us write $\fa$ for the $(-1)$-eigenspace of $\theta$ on $\fh$. Then we have $\fh=\fa \oplus \fh^\theta$. Let us now fix the character $\Lambda$. Then the set of the twisted $K$-equivariant rank one $\cD$-modules $\gamma$ with fixed $\Lambda$ is a torsor over $\nu\in\fa^*$. Thus, if we fix a particular $\gamma$ associated to the data $(\lambda,\Lambda)$, then all its deformations are given by $(\lambda+\nu,\Lambda)$; we write $\gamma_\nu = \mc{O}_Q(\lambda + \nu, \Lambda)$ for the corresponding family of local systems.

Geometrically, the deformations $\gamma_\nu$ can be constructed as follows. Suppose first that $\varphi \in \mb{X}^*(H)$ and that there exists a non-zero $K$-invariant element $f_\varphi \in H^0(Q, \mc{L}_\varphi)^K$. Then $f_\varphi$ defines a non-vanishing function $f_\varphi \colon \tilde{Q} \to \mb{C}$ on the preimage $\tilde{Q}$ of $Q$ in $\tilde{\ms{B}}$, on which $H$ acts by $\varphi$. For all $s$, the deformed twisted local system $\gamma_{s\varphi}$ is given by the monodromic $\ms{D}$-module
\[ \gamma_{s\varphi} = f_\varphi^s \gamma,\]
where $f_\varphi^s\gamma = \gamma$ as $\mc{O}_{\tilde{Q}}$-modules and $\ms{D}_{\tilde{Q}}$ acts by
\begin{equation} \label{eq:f to the s action}
 \partial (f_\varphi^s m) = f_\varphi^s\left(\partial m + s\frac{\partial f_\varphi}{f_\varphi} m \right)
\end{equation}
for a vector field $\partial$ on $\tilde{Q}$ and a local section $m$ of $\gamma$. More generally, given $\varphi_1, \ldots, \varphi_n \in \mb{X}^*(H)$ and non-zero sections $f_i \in H^0(Q, \mc{L}_{\varphi_i})^K$, we have
\[ \gamma_{s_1\varphi_1 + \cdots + s_n \varphi_n} = f_1^{s_1} \cdots f_n^{s_n} \gamma.\]
The characters $\varphi$ for which such $f_{\varphi}$ exist form a sublattice in $\mb{X}^*(H) \cap \mf{a}^*$ of full rank, so any $\gamma_\nu$ may be realized in this way.

\subsection{Jantzen conjecture with mixed Hodge modules} \label{subsec:mhm jantzen}

In this subsection we recall the results of Beilinson and Bernstein from~\cite{beilinson-bernstein} about Jantzen filtrations. We work in the context of Hodge modules, which carry more information than the original treatment. At the end of this section we state our second main theorem.

Fix $\lambda \in \mf{h}^*_\mb{R}$, a $K$-orbit $Q \subset \ms{B}$ and an irreducible $K$-equivariant $\lambda$-twisted local system $\gamma$ on $Q$. By \cite[Lemma 3.5.2]{beilinson-bernstein}, there exists $\varphi \in \mb{X}^*(H)$ and a $K$-invariant section $f_\varphi \in H^0(\bar{Q}, \ms{L}_\varphi)$ such that $f_\varphi^{-1}(0) \cap \bar{Q} = {\partial}Q$. (The existence of such an $f_\varphi$ may be regarded as a weak positivity condition on $\varphi \in \mb{X}^*(H) \cap \mf{a}^*$.) Fixing these, we may form the $K$-equivariant $(\lambda + s_0\varphi)$-twisted local system $\gamma_{s_0\varphi} = f_\varphi^{s_0}\gamma$ on $Q$ for any real number $s_0$.

Consider the tautological morphisms
\begin{equation} \label{eq:jantzen taut morphism}
 j_!f_\varphi^{s_0}\gamma \to j_*f_\varphi^{s_0}\gamma.
\end{equation}
To study the behavior near $s_0 = 0$, observe first that the module $j_*f_\varphi^{s_0}\gamma$ is the quotient by $s - s_0$ of the $\mc{D}_{\tilde{\mc{B}}}[s]$-module $j_+f_\varphi^s\gamma[s]$, where $f_\varphi^s\gamma[s]$ given by $\gamma[s]$ as a sheaf of $\mc{O}_{\tilde Q}[s]$-modules with $\mc{D}_{\tilde Q}$-action defined by \eqref{eq:f to the s action} and $j_+$ is the naive $\mc{D}$-module pushforward to $\tilde{\mc{B}}$. Similarly, one can write down a sheaf of $\mc{D}_{\tilde X}[s]$-submodules
\[ j_!f_\varphi^s\gamma[s] \subset j_+f_\varphi^s\gamma[s] \]
such that the quotient by $s - s_0$ recovers \eqref{eq:jantzen taut morphism} for $s_0$ sufficiently close to $0$. Since this is for us only motivation, we will not make a precise definition of $j_!f_\varphi^s\gamma[s]$; one can take, for example, the module $j_!^{(0)}f^s\gamma[s]$ defined in \cite[\S 3.3]{DV3}. This realizes \eqref{eq:jantzen taut morphism} as an algebraic family. Taking the formal completion at $s = 0$, we obtain a morphism
\begin{equation} \label{eq:jantzen map}
 j_!f_\varphi^s\gamma[[s]] \to j_*f_\varphi^s\gamma[[s]].
\end{equation}
Note that both sides are now well-defined by the formulas
\[ j_!f_\varphi^s\gamma[[s]] := \varprojlim_n j_!\frac{f_\varphi^s\gamma[s]}{s^n f_\varphi^s\gamma[s]} \quad \text{and} \quad j_*f_\varphi^s\gamma[[s]] := \varprojlim_n j_*\frac{f_\varphi^s\gamma[s]}{s^n f_\varphi^s\gamma[s]},\]
where we observe that the quotients by $s^n$ are holonomic $\mc{D}$-modules, so the $!$ and $*$ extensions are well-defined for them.

Our assumption on $\varphi$ implies that \eqref{eq:jantzen map} is an isomorphism after inverting $s$. Hence it is an injection and it induces Jantzen filtrations $J_{\bigcdot}$ on the domain and codomain defined by
\begin{equation} \label{eq:jantzen filtration shriek}
 J_n j_!\gamma = (j_!f_\varphi^s\gamma[[s]] \cap s^{-n} j_*f_\varphi^s\gamma[[s]])/(s)
\end{equation}
and
\begin{equation} \label{eq:jantzen filtration star}
J_n j_*\gamma = (s^{-n}j_!f_\varphi^s\gamma[[s]] \cap j_*f_\varphi^s\gamma[[s]])/(s)
\end{equation}
and isomorphisms
\[ s^n \colon \mrm{Gr}^J_{n} j_*\gamma \overset{\sim}\to \mrm{Gr}^J_{-n} j_!\gamma.\]

\begin{thm}(\cite{beilinson-bernstein}) \label{thm:beilinson-bernstein}
The Jantzen filtrations coincide with the weight filtrations up to a shift and the $s$ corresponds to the Tate twist, i.e., 
\[
J_n j_*\gamma \ = \ W_{d+n}  j_*\gamma \ \   J_{-n} j_!\gamma  \ = \ W_{d-n} j_!\gamma\, \ \ \text{and}\ \  \mrm{Gr}^J_{n} j_*\gamma \overset{\sim}\to \mrm{Gr}^J_{-n} j_!\gamma (-n)\,,
\]
where $d=\dim Q + \dim H$ (cf., Remark \ref{rmk:weights}).
\end{thm}
In~\cite{beilinson-bernstein}, this theorem is proved in the context of mixed sheaves over finite base fields. This requires the assumption that $\lambda$ is rational, although this is not a very restrictive hypothesis. 

In this paper we extend Theorem \ref{thm:beilinson-bernstein} to the context of mixed Hodge modules, avoiding the detour through finite fields. A key difference, in addition to being able to work with arbitrary real $\lambda$ is that we are keeping track of more structure: the Hodge filtration and the polarization. This gives us substantially more information.

Recall that the polarization on $j_{!*}\gamma$ is of the form $S = j_{!*}S_Q$, for $S_Q \colon \gamma \to \gamma^h(-d)$ a polarization on $Q$. Since the functor $j_!$ is fully faithful, the polarization extends uniquely to a non-degenerate sesquilinear pairing
\[  j_!\gamma \to j_!\gamma^h(-d) = (j_*\gamma)^h(-d),\]
which for simplicity we also denote by $S$. The Jantzen filtrations are dual under $S$ (this is clear from the formulas in \S\ref{subsec:proof of theorem 2}, for example), so we obtain (dual) perfect pairings
\[ \mrm{Gr}_{-n}^J(S) \colon \mrm{Gr}_{-n}^J(j_!\gamma) \to \mrm{Gr}_n^J(j_*\gamma)^h(-d) \]
and
\[ \mrm{Gr}_n^J(S) \colon \mrm{Gr}_n^J(j_*\gamma) \to \mrm{Gr}_{-n}^J(j_!\gamma)^h(-d).\]
The \emph{Jantzen forms} are the non-degenerate Hermitian forms
\[ s^{-n}\mrm{Gr}_{-n}^J(S) \colon \mrm{Gr}_{-n}^J(j_!\gamma) \to \mrm{Gr}_{-n}^J(j_!\gamma)^h(-d + n)\]
and
\[ s^n\mrm{Gr}_n^J(S) \colon \mrm{Gr}_n^J(j_*\gamma) \to \mrm{Gr}_n^J(j_*\gamma)^h(-d - n).\]
Our second main theorem is the following.

\begin{thm} \label{thm:main theorem 2}
Let $S$ be a polarization of $\gamma$. Then for all $n$, $\mrm{Gr}^J_{-n}j_!\gamma$ is a pure Hodge module of weight $d - n$, and the form
\[ s^{-n}\mrm{Gr}_{-n}^J(S) \colon \mrm{Gr}_{-n}^Jj_!\gamma \overset{\sim}\to (\mrm{Gr}_{-n}^Jj_!\gamma)^h(-d + n) \]
is a polarization.
\end{thm}

Section~\ref{proof:main theorem 2} is devoted to the proof of this theorem. 

\begin{rmk}
Although we do not use this here, the Jantzen forms admit the following interpretation in terms of the polynomial families $j_!f_\varphi^s\gamma[s]$ and $j_+f_\varphi^s\gamma[s]$. Consider the $\mb{C}[s]$-bilinear pairing
\[ S_Q \colon f_\varphi^s\gamma[s] \otimes \overline{f_\varphi^s \gamma[s]} \to \Db_{\tilde{Q}}[s]^{\mathit{hol}}, \]
valued in the sheaf $\Db_{\tilde{Q}}[s]^{\mathit{hol}}$ of distributions on $Q$ depending holomorphically on $s$, defined by the formula
\[ S_Q(f_\varphi^s m(s), \overline{f_\varphi^s m'(s)}) = |f_\varphi|^{2s} S_Q(m(s), \overline{m'(s)}),\]
where on the right hand side we adopt the convention $\bar{s} = s$. By the argument of \cite[Proposition 12.5.4]{SS}, this extends further to a pairing
\[ S \colon j_+f_\varphi^s\gamma[s] \otimes \overline{j_+f_\varphi^s\gamma[s]} \to \Db_{\tilde{\mc{B}}}(s)^{\mathit{mer}},\]
where $\Db_{\tilde{\mc{B}}}(s)^{\mathit{mer}}$ now denotes the sheaf of distributions depending meromorphically on $s$. One can show that the restriction to $j_!f_\varphi^s\gamma[s] \otimes \overline{j_+f_\varphi^s\gamma[s]}$ has no poles in a neighborhood of $s = 0$ and that evaluating at sufficiently small non-zero $s_0$ gives a polarization on $j_!f_\varphi^{s_0}\gamma = j_*f_\varphi^{s_0}\gamma = j_{!*}f_\varphi^{s_0}\gamma$. We think of this as a family of polarizations degenerating at $s = 0$. Completing, we get a pairing
\[S \colon j_!f_\varphi^s\gamma[[s]] \otimes \overline{j_*f_\varphi^s\gamma[[s]]} \to \Db_{\tilde{\mc{B}}}[[s]] \]
from which we can recover the Jantzen forms as
\[ s^{-n}\mrm{Gr}_{-n}^J(S)(m(0), \overline{m'(0)}) = S(m(s), s^{-n}\overline{m'(s)})|_{s = 0},\]
for $m, m' \in j_!f_\varphi^s\gamma[[s]] \cap s^n j_*f_\varphi^s\gamma[[s]]$ and
\[ s^n\mrm{Gr}_n^J(S)(m(0), \overline{m'(0)}) = S(m(s), s^n\overline{m'(s)})|_{s = 0}\]
for $m, m' \in s^{-n}j_!f_\varphi^s\gamma[[s]] \cap j_*f_\varphi^s\gamma[[s]]$.
\end{rmk}

\section{The third main theorem}
\label{sec:main theorem 3}

In this section, we make contact between our study of Hodge modules on the flag variety and representations of real reductive groups, via their incarnations as Harish-Chandra modules. We recall the localization theory of Beilinson-Bernstein in \S\ref{subsec:localization}. In \S\ref{subsec:mk statement}, we recall the notion of minimal $K$-types and state our third main theorem linking these to the Hodge filtration. With a view to our applications to real groups, we recall the Schmid-Vilonen construction of Hermitian forms from polarizations in \S\ref{global_polar}. Finally, we explain in \S\ref{subsec:ALTV} how our results imply one of the main theorems of \cite{ALTV}.

\subsection{Beilinson-Bernstein localization} \label{subsec:localization}

In this subsection, we recall the Beilinson-Bernstein localization theory in our context.

We write $\Mod(\fg,K)_\la$ for the category of Harish-Chandra $(\fg,K)$-modules with infinitesimal character (i.e., the action of the center of $U(\mf{g})$) induced by $\la$ via the Harish-Chandra homomorphism. Recall that we have not included a $\rho$-shift in our notation, so, for example, the trivial representation has infinitesimal character $\lambda = 0$.

Assume that $\lambda+\rho$ is dominant as well as real, i.e., that $\langle \lambda+\rho,\check\alpha_i \rangle \geq 0$ for simple positive coroots $\check\alpha_i$. If $\lambda + \rho$ is also regular, then according to Beilinson and Bernstein~\cite{beilinson-bernstein} the global sections functor induces an equivalence of categories
\begin{equation*}
\Gamma \colon \Mod_K(\cD_\lambda)\cong \Mod(\fg,K)_\la\,. 
\end{equation*}
If $\lambda + \rho$ is dominant but not regular, then $\Gamma$ is exact and $\Mod(\fg,K)_\la$ is a Serre quotient of $\Mod_K(\cD_\lambda)$. In particular, each irreducible module in $\Mod(\fg,K)_\la$ arises as global sections of a unique irreducible module in $\Mod_K(\cD_\lambda)$. The above statements also apply in the monodromic version of the theory: for $\lambda + \rho$ dominant, the global sections gives an exact functor
\[ \Gamma \colon \Mod_K^\mon(\ms{D}_\lambda) \to \Mod(\mf{g}, K)_{\widehat{\lambda}},\]
which is an equivalence of categories for $\lambda + \rho$ regular, where the target is the category of Harish-Chandra modules with generalized infinitesimal character $\lambda$.

\begin{rmk}
The global sections $\Gamma$ here are taken over $\ms{B}$. In terms of weakly $H$-equivariant $\ms{D}$-modules on the base affine space, the functor is given by
\[ \Gamma(\ms{M}) = \Gamma(\tilde{\ms{B}}, \ms{M})^H.\]
\end{rmk}

It will be convenient to introduce the following terminology. Let $\gamma$ be a $K$-equivariant $\lambda$-twisted local system of rank $1$ on a $K$-orbit $Q$ for some $\lambda \in \mf{h}^*_\mb{R}$. We say that $\gamma$ is \emph{relevant} if
\begin{enumerate}
\item $\lambda + \rho$ is dominant  and
\item $\Gamma(j_{!*}\gamma) \neq 0$.
\end{enumerate}
This corresponds to the notion of regular Beilinson-Bernstein data of \cite[\S 3]{chang}. The global sections functor defines a bijection between isomorphism classes of irreducible modules in $\Mod(\mf{g}, K)_\lambda$ and relevant local systems $\gamma$. If $\gamma = \mc{O}_Q(\lambda, \Lambda)$, then we have the following criterion for relevance due to Beilinson and Bernstein:

\begin{prop}[{\cite[Theorem 3.15]{chang}}]
The local system $\gamma = \mc{O}_Q(\lambda, \Lambda)$ is relevant if and only if $\lambda+\rho$ is dominant and the following are satisifed.
\begin{enumerate}
\item There is no compact simple root $\alpha$ with $\langle \lambda, \check\alpha \rangle = -1$.
\item There is no complex simple root $\alpha$ with $\theta \alpha \in \Phi_-$ and $\langle \lambda, \check\alpha \rangle = -1$.
\item There is no real simple root $\alpha$ with $\Lambda(m_\alpha) = -1$ and $\langle \lambda, \check\alpha \rangle = -1$.
\end{enumerate}
\end{prop}

Here $m_\alpha = \check\alpha(-1) \in H^{\theta_Q}$ and we refer to Definition \ref{defn:root classification} for the notion of compact, complex and real roots.

\subsection{Minimal \texorpdfstring{$K$}{K}-types} \label{subsec:mk statement}

In this subsection we recall the minimal $K$-types of Vogan and formulate our third main theorem. As we work in a geometric context, it will be convenient to refer to~\cite{chang} for standard results on minimal $K$-types.

Let us consider the irreducible Hodge module $j_{!*}\gamma$ associated to a pair $(Q,\gamma)$ and let us write $c = \codim(Q)$. By definition of the Hodge filtration on a pushforward (see \S\ref{subsec:hodge filtration}, especially the formulas~\eqref{closed embedding formula}, \eqref{open embedding formula} and \eqref{dual open embedding formula}), the first non-zero term of the Hodge filtrations of $j_{!}\gamma$, $j_{!*}\gamma$, and $j_{*}\gamma$ is in the filtered degree $c$.

Let us recall the notion of minimal $K$-type due to Vogan~\cite{Vogan thesis}. A \emph{minimal $K$-type of $\Gamma(\cB,j_{*}\gamma)$} is an irreducible representation $V_\mu$ of $K$ of highest weight $\mu$ which is smallest of the irreducible representations of $K$ appearing in $\Gamma(\cB,j_{*}\gamma)$ in the sense that the length of $\mu+2\rho_K$ is minimal. A costandard representation $\Gamma(\cB,j_{*}\gamma)$ can have several minimal $K$-types. The minimal $K$-types have the following simple but remarkable property due to Vogan.

\begin{thm}[{\cite[Theorems 8.15 and 8.19]{chang}}] \label{thm:minimal K-type in irreducible}
If $\gamma$ is relevant, then every minimal $K$-type of $\Gamma(\ms{B}, j_*\gamma)$ lies in $\Gamma(\ms{B}, j_{!*}\gamma)$.
\end{thm}

We also have the following useful result.

\begin{prop}[{\cite[\S 8]{chang}}] \label{prop:minimal K-type support}
Every minimal $K$-type of $\Gamma(\ms{B}, j_*\gamma)$ lies in the subspace
\[ \Gamma(\ms{B}, j_{\bigcdot}(\gamma \otimes \omega_{Q/\ms{B}})) \subset \Gamma(j_*\gamma).\]
\end{prop}

Our third main theorem is the following. 

\begin{thm} \label{thm:main theorem 3}
If $\gamma$ is relevant, then the minimal $K$-types of $\Gamma(\cB,j_{!*}\gamma)$ lie in $\Gamma(\cB,F_c j_{!*}\gamma)$. In particular, $\Gamma(\cB,F_c j_{!*}\gamma)\neq 0$.
\end{thm}

We defer the proof of this statement to \S\ref{mk}. Remarkably, Theorem \ref{thm:minimal K-type in irreducible} and Proposition \ref{prop:minimal K-type support} are the only facts about minimal $K$-types required.

\subsection{Hermitian forms on representations}
\label{global_polar}

Let $\ms{M}$ be a pure twisted Hodge module in $\mhm_\lambda(K \bslash \ms{B})$ and $S \colon \ms{M} \to \ms{M}^h(-w)$ a polarization. In this subsection, we recall from \cite{schmid-vilonen} the following induced Hermitian form on $\Gamma(\ms{M})$.

By choosing a Borel, we have a real algebraic embedding $U_\bR \subset \tilde\cB$, so that $\tilde{\ms{B}} \cong U_{\mb{R}} \times H_\mb{R}^\circ$ as real manifolds, where $U_\mb{R} \subset G$ is the compact real form, $H_\mb{R} \subset H$ is the split real form with Lie algebra $\mf{h}_\mb{R}$, and $H_\mb{R}^\circ$ is its identity component. If $m, m' \in \Gamma(\ms{M}) = \Gamma(\tilde{\ms{B}}, \ms{M})^H$, then the distribution $S(m, \overline{m'})$ on $\tilde{\ms{B}}$ is annihilated by the operators $h-\lambda(h)$ and $\bar h-\lambda(\bar h)$ for $h\in\fh$, so we may write
\[ S(m, \overline{m'}) = S(m, \overline{m'})|_{U_\mb{R}} \cdot \exp(2\lambda),\]
for some distribution $S(m, \overline{m'})|_{U_\mb{R}}$ on $U_\mb{R}$. Here we have chosen compatible orientations on $H_\mb{R}^\circ$ and $U_\mb{R}$ in order to identify the function $\exp(2\lambda)$ with a distribution on $H_\mb{R}^\circ$. We now define a pairing
\begin{equation}
\label{sesqui}
 \Gamma(S)  \colon \Gamma(\ms{M}) \otimes \overline{\Gamma(\ms{M})} \to \mb{C}
 \end{equation}
 by the formula
\[ 
\Gamma(S)(m, \overline{m'}) = \int_{U_\bR} S(m, \overline{m'})|_{U_\bR},
\]
where the integral is taken with respect to an invariant volume form on $U_\bR$ compatible with its chosen orientation\footnote{We have chosen to consider distributions as duals of smooth compactly supported forms to be compatible with~\cite{SS}. It would perhaps be better to consider them as duals of smooth compactly supported densities.}. 

\begin{prop} (\cite[Proposition 5.10]{schmid-vilonen})
\label{sv-form}
The pairing in~\eqref{sesqui} is $\fu_\bR$-invariant and Hermitian, where $\mf{u}_\mb{R} = \mrm{Lie}(U_\mb{R})$. It depends on the choices made only up to a positive constant multiple.
\end{prop}

The conjectures in~\cite{schmid-vilonen} postulate that the pair $(\Gamma(\ms{M}), \Gamma(S))$, with the natural Hodge filtration on $\Gamma(\ms{M})$, defines an infinite dimensional polarized Hodge structure of weight $w - \dim \tilde{\ms{B}}$.

\subsection{Comparison with ALTV} \label{subsec:ALTV}

We now turn to the connection between the present work and \cite{ALTV}. In particular, we explain why our results imply the signature character formula~\cite[Theorem 20.6]{ALTV}.

Let us say a few words about the history of the unitarity problem addressed in \cite{ALTV}. Harish-Chandra showed that an irreducible representation of the real form $G_\bR \subset G$ associated with $(G, \theta)$ is unitary if and only if the corresponding $(\fg,K)$-module carries a $\fg_\bR$-invariant positive definite form. It is not difficult to see that the classification of unitary representations can be reduced to the case of real infinitesimal character, see, for example, \cite[Theorem 16.10]{knapp1986book} or the discussion in~\cite{ALTV}. Thus, it suffices to work, as we do, under the assumption that the infinitesimal character is real.  It is also not difficult to determine when an irreducible representation carries an invariant Hermitian form. The main difficulty lies in determining when the form is positive (or negative) definite. This has turned out to be a difficult problem. 

The novel idea introduced in~\cite{ALTV} is to consider a different form, which they call a $c$-form. The $c$-form is a $\fu_\bR$-invariant Hermitian form. It is not difficult to show that a $\fu_\bR$-invariant form exists and is unique up a scalar if the Harish-Chandra module is irreducible with real infinitesimal character. The $c$-form on irreducible representations is fixed up to positive scalar by requiring it to be positive definite on the minimal $K$-types \cite[Proposition 10.7]{ALTV}. The problem of unitarity is reduced to determining the signatures of the $c$-forms \cite[Chapter 12]{ALTV}.

Let us consider a relevant local system $\gamma$. From the discussion above, the irreducible module $\Gamma(\cB,j_{!*}\gamma)$ carries a $c$-form $S_\gamma^c$ which is unique up to a positive scalar. On the other hand, by the discussion in \S\ref{global_polar}, the polarization $S$ on $j_{!*}\gamma$ induces another $\fu_\bR$-invariant form $\Gamma(S)$ on $\Gamma(\cB,j_{!*}\gamma)$.

In \S\ref{mk} we will show that 
\begin{prop}
\label{prop:hodge positivity}
The polarization $\Gamma(S)$ restricted to $\Gamma(\cB,F_c j_{!*}\gamma)$ is positive definite. 
\end{prop}

\begin{rmk}
Although we will only discuss it in our context, Proposition \ref{prop:hodge positivity} and its proof generalize easily to polarized Hodge modules on arbitrary varieties. We leave the task of spelling out the precise statement to the interested reader.
\end{rmk}

Combining Proposition~\ref{prop:hodge positivity} with Theorem~\ref{thm:main theorem 3} we conclude:
\begin{cor}
\label{forms coincide}
Let $\gamma$ be a relevant local system, and let $S$ be the polarization on $j_{!*}\gamma$. Then the $c$-form $S_\gamma^c$ and the form $\Gamma(S)$ coincide up to a positive scalar. 
\end{cor}
This was stated as Corollary \ref{cor:intro 1} in the introduction. This corollary allows us to immediately deduce results about $c$-forms from our results on mixed Hodge modules. 

Just as in \S\ref{sec:main theorem 2}, the deformation $\gamma_{s\varphi}$ of \S\ref{subsec:mhm jantzen} induces a Jantzen filtration $J_{\bigcdot}$ on $\Gamma(j_!\gamma)$, and the $c$-form $S_\gamma^c$ induces Jantzen $c$-forms $s^{-n}\mrm{Gr}^J_{-n}(S_\gamma^c)$ on $\mrm{Gr}^J_{-n}\Gamma(j_!\gamma)$ (cf., \cite[Propositions 14.2 and 14.6]{ALTV}). The signature multiplicity polynomial \cite[Definition 20.2]{ALTV} is defined by
\[ \LVM_{\gamma', \gamma}^c(u, \zeta) = u^{\frac{1}{2}(\ell_I(\gamma) - \ell_I(\gamma'))}\sum_n(m^{c, n, +}_{\gamma', \gamma} + \zeta m^{c, n, -}_{\gamma', \gamma})u^{-n/2} \]
where
\[ (\mrm{Gr}_{-n}^J(\Gamma(j_!\gamma)), s^{-n}\mrm{Gr}_{-n}^J(S_\gamma^c)) \cong \bigoplus_{\gamma'} \left((\Gamma(j_{!*}\gamma'), S_{\gamma'}^c)^{m^{c, n, +}_{\gamma', \gamma}} \oplus (\Gamma(j_{!*}\gamma'), -S_{\gamma'}^c)^{m^{c, n, -}_{\gamma', \gamma}}\right).\]

We may now deduce Corollary \ref{cor:intro 2} from the introduction.

\begin{cor} \label{cor:hodge signature}
Assume $\gamma$ and $\gamma'$ are relevant. Then we have
\[ \LVM^c_{\gamma', \gamma}(u, \zeta) = u^{\frac{1}{2}(\ell(\gamma') - \ell_I(\gamma') - (\ell(\gamma) - \ell_I(\gamma)))} \LVM_{\gamma', \gamma}^h(u^{\frac{1}{2}}, \zeta u^{\frac{1}{2}}).\]
\end{cor}
\begin{proof}
To compute the coefficients
\[ w^{c, n}_{\gamma', \gamma}:= m^{c, n, +}_{\gamma', \gamma} + \zeta m^{c, n, -}_{\gamma', \gamma},\]
consider a summand $\mb{C}_{-p, -q} \otimes j_{!*}\gamma'$ of $\mrm{Gr}_{-n}^Jj_!\gamma$, which necessarily satisfies $p + q = \ell(\gamma) - \ell(\gamma') - n$. By Theorem \ref{thm:main theorem 2}, the Jantzen form $\mrm{Gr}_{-n}^J(S_\gamma)$ is a polarization on $\mrm{Gr}_{-n}^Jj_!\gamma$, so its restriction to the summand is given (up to scale) by $(-1)^{q}S_{\gamma'}$, where $S_\gamma$ and $S_{\gamma'}$ are the polarizations on $\gamma$ and $\gamma'$. So if $\LVM_{\gamma', \gamma}^h = \sum_{p, q} a_{p, q}t_1^pt_2^q$, then by Corollary \ref{forms coincide} and the exactness of global sections,
\[ w_{\gamma', \gamma}^{c, n} = \sum_{p + q = \ell(\gamma) - \ell(\gamma') - n} a_{p, q}\zeta^q. \]
So by Theorem \ref{thm:main theorem 1},
\begin{align*}
 \LVM^c_{\gamma', \gamma}(u, \zeta) &= u^{\frac{1}{2}(\ell_I(\gamma) - \ell_I(\gamma'))}\sum_{p, q} a_{p, q}\zeta^q u^{\frac{1}{2}(\ell(\gamma') - \ell(\gamma) + p + q)} \\ 
&= u^{\frac{1}{2}(\ell(\gamma') - \ell_I(\gamma') - (\ell(\gamma) - \ell_I(\gamma)))}\LVM^h_{\gamma', \gamma}(u^{\frac{1}{2}}, \zeta u^{\frac{1}{2}})
\end{align*}
as claimed.
\end{proof}

Combining with Theorem \ref{thm:main theorem 1}, we deduce Corollary \ref{cor:intro 3}, which is the key result in~\cite{ALTV} used to compute signatures of $c$-forms.

\begin{cor}[{\cite[Theorem 20.6]{ALTV}}]
\label{ALTV}
The signature multiplicity polynomial is given by 
\[
\LVM^c_{\gamma', \gamma}(u, \zeta) = \zeta^{\frac{1}{2}(\ell_o(\gamma) - \ell_o(\gamma'))}\LVM_{\gamma', \gamma}(\zeta u)\,,
\]
where $\ell_o$ is the orientation number of Definition \ref{defn:hodge shift}.
\end{cor}
\begin{proof}
By Theorem \ref{thm:main theorem 1},
\begin{align*}
\LVM^c_{\gamma', \gamma}(u, \zeta) &= u^{\frac{1}{2}(\ell(\gamma') - \ell_I(\gamma') - (\ell(\gamma) - \ell_I(\gamma)))}\LVM^h_{\gamma', \gamma}(u^{\frac{1}{2}}, \zeta u^{\frac{1}{2}})\\
&= \zeta^{\frac{1}{2}(\ell_H(\gamma) - \ell_H(\gamma'))}u^{\frac{1}{2}(\ell(\gamma') - \ell_I(\gamma') - (\ell(\gamma) - \ell_I(\gamma)))}\LVM^m_{\gamma', \gamma}(\zeta u).
\end{align*}
Applying Proposition \ref{prop:mixed to LV}, we obtain
\[ \LVM^c_{\gamma', \gamma}(u, \zeta) = \zeta^{\frac{1}{2}(\ell_H(\gamma) + \ell(\gamma) - \ell_I(\gamma) - (\ell_H(\gamma') + \ell(\gamma') - \ell_I(\gamma')))}\LVM_{\gamma', \gamma}(\zeta u).\]
We conclude by applying Proposition \ref{prop:orientation vs hodge}.
\end{proof}

\section{Proof of Theorem \ref{thm:main theorem 1}}
\label{proof:main theorem 1}

In this section, we give the proof of Theorem \ref{thm:main theorem 1}. In \S\ref{subsec:convolution}, we introduce the Hecke operators on $\oK(\mhm_\lambda(K \bslash \ms{B}))$. We then recall how the Hodge filtration on a pushforward of mixed Hodge modules is defined in \S\ref{subsec:hodge filtration}, and use this in \S\ref{subsec:hecke action} to compute the Hecke operators in the $!$-basis explicitly. The key calculation is Lemma \ref{lem:real hodge calculation}, which gives the Hodge structure for a real non-integral reflection: this is the only place where interesting Hodge structures appear. With the aid of these calculations, we give a version of the results of Lusztig and Vogan on the mixed change of basis matrix $\LVM^m$ in \S\ref{subsec:lusztig-vogan}, and explain how to incorporate Hodge structures into their arguments in \S\ref{subsec:hodge combinatorics}.

In addition to the appearance of interesting Hodge structures and non-integral $\lambda$, the astute reader will notice that, in contrast to \cite{LV}, we work in terms of the matrix of multiplicity polynomials $\LVM$ rather than its inverse $\mbf{P}$ and we use the Hermitian dual $(-)^h$ in place of the standard dual $\mb{D}$ when running the Kazhdan-Lusztig algorithm in \S\ref{subsec:lusztig-vogan}. The first point is a matter of taste, as the matrix $\LVM$ appears slightly more natural to us from the point of view of $\ms{D}$-modules, and the second a matter of convenience, as the Hermitian dual preserves the (real) twisting $\lambda$ while $\mb{D}$ negates it.

\subsection{Convolution functors} \label{subsec:convolution}

In this subsection, we recall the construction of convolution functors, and important special cases, the Hecke and translation functors.

First, we remark that twisted mixed Hodge modules have a good theory of equivariant derived categories. More precisely, if an algebraic group $L$ acts on an $H$-torsor $\tilde{X} \to X$, then we may define the equivariant derived category
\[ \oD_L^b(\mhm_\lambda(X)) \]
to be the derived category of complexes of $\lambda$-twisted mixed Hodge modules on the simplicial space $L^{\bigcdot} \times X$ with bounded and Cartesian cohomology. The equivariant derived categories have the same functoriality as the usual derived categories with respect to equivariant morphisms, as well as functors of restriction of structure group.

Consider the $G$-equivariant $H \times H$-torsor
\[ \tilde{\ms{B}} \times \tilde{\ms{B}} \to \ms{B} \times \ms{B} \]
and the associated equivariant derived categories $\oD^b_G(\mhm_{-\lambda, \mu}(\ms{B} \times \ms{B}))$ of twisted mixed Hodge modules, for $(-\lambda, \mu) \in \mrm{Lie}(H \times H)^*_\mb{R} = (\mf{h}^*_\mb{R})^2$. The diagonal on the first and second factors
\[ \Delta_{12} \colon \tilde{\ms{B}} \times \tilde{\ms{B}} \to \tilde{\ms{B}} \times \tilde{\ms{B}} \times \tilde{\ms{B}} \]
and the restriction from $G$ to $K$ provide a functor
\begin{align*}
\Delta_{12}^\circ := \Delta_{12}^*[-\dim \tilde{\ms{B}}] \colon & \oD^b_K(\mhm_\lambda(\tilde{\ms{B}})) \times \oD^b_G(\mhm_{-\lambda, \mu}(\tilde{\ms{B}} \times \tilde{\ms{B}})) \\
& \qquad \to \oD^b_K(\mhm_{0, \mu}(\tilde{\ms{B}} \times \tilde{\ms{B}})),
\end{align*}
where for the sake of clarity we have written the total spaces of the monodromic torsors instead of the bases. The morphism $\pi_1 \colon \tilde{\ms{B}} \times \tilde{\ms{B}} \to \ms{B} \times \tilde{\ms{B}}$ induces an equivalence
\[ \pi_1^\circ := \pi_1^*[\dim H] \colon \oD^b_K(\mhm_{\mu}(\ms{B} \times \tilde{\ms{B}})) \to \oD^b_K(\mhm_{0, \mu}(\tilde{\ms{B}} \times \tilde{\ms{B}})). \]

\begin{defn}
The \emph{convolution action} is
\begin{align*}
-\star- = \mrm{pr}_{2*}(\pi_1^\circ)^{-1}\Delta_{12}^\circ \colon & \oD^b_K(\mhm_\lambda(\tilde{\ms{B}})) \times \oD^b_G(\mhm_{-\lambda, \mu}(\tilde{\ms{B}} \times \tilde{\ms{B}})) \\
& \qquad \to \oD^b_K(\mhm_{\mu}(\tilde{\ms{B}})).
\end{align*}
\end{defn}

\begin{rmk}
The cohomological shifts in $\Delta_{12}^\circ$ and $\pi_1^\circ$ are chosen so that the functors are t-exact, i.e., they send mixed Hodge modules to mixed Hodge modules.
\end{rmk}

Let $w \in W$ be an element of the Weyl group of $G$, and let
\[ X_w = G \cdot (x, wx) \subset \ms{B} \times \ms{B} \]
be the corresponding $G$-orbit, where $x \in \ms{B}$ is an arbitrary element and the notation $wx$ is defined using any maximal torus $T$ fixing $x$ and the corresponding isomorphism $N_G(T)/T \cong W$. Via the construction of \S\ref{subsec:parametrization} applied to $G \times G$, one easily sees that there is a unique rank $1$ $G$-equivariant $(\lambda, \mu)$-twisted local system on $X_w$ if and only if $\lambda + w\mu \in \mb{X}^*(H)$; we will denote this by $\mc{O}_{X_w}(\lambda, \mu)$.

\begin{defn}
The \emph{Hecke functor} at $\lambda$ is
\[ \mb{T}_{w}^\lambda = - \star j_{w !} \mc{O}_{X_{w}}(-\lambda, w \lambda) \colon \oD^b_K(\mhm_\lambda(\ms{B})) \to \oD^b_K(\mhm_{w \lambda}(\ms{B})),\]
where $j_w \colon X_w \to \ms{B} \times \ms{B}$ is the inclusion. We will write $\mb{T}_{w} = \mb{T}_{w}^\lambda$ when $\lambda$ is clear from context.
\end{defn}

Now let $\alpha \in \Phi_+$ be a simple root, and $s_\alpha \in W$ the corresponding simple reflection. In this case, we may describe the Hecke functor $\mb{T}_{s_\alpha}$ explicitly as follows. Let $\ms{P}_\alpha$ be the partial flag variety parametrizing parabolic subgroups of type $\alpha$ and $\pi_\alpha \colon \ms{B} \to \ms{P}_\alpha$ the canonical $\mb{P}^1$-fibration. The orbit $X_{s_\alpha}$ is the complement of the diagonal in the fiber product
\[ \ms{B} \xleftarrow{p_1} \ms{B} \times_{\ms{P}_\alpha} \ms{B} \xrightarrow{p_2} \ms{B}.\]
The $H$-torsors $p_1^{-1}(\tilde{\ms{B}})$ and $p_2^{-1}(\tilde{\ms{B}})$ give rise to categories
\[ \oD^b_K(\mhm_\lambda^{(1)}(\ms{B} \times_{\ms{P}_\alpha} \ms{B})) \quad \text{and} \quad \oD^b_K(\mhm_\lambda^{(2)}(\ms{B} \times_{\ms{P}_\alpha} \ms{B})).\]
There is a unique isomorphism of $H$-torsors
\begin{equation} \label{eq:untwisting iso}
p_1^{-1}(\tilde{\ms{B}})|_{\ms{B} \times_{\ms{P}_\alpha} \ms{B} - \ms{B}} \cong s_\alpha^*(p_2^{-1}(\tilde{\ms{B}}))|_{\ms{B} \times_{\ms{P}_\alpha} \ms{B} - \ms{B}}
\end{equation}
over the complement of the diagonal, where the $H$-action on the right hand side is twisted by the homomorphism $s_\alpha \colon H \to H$, which induces an equivalence of categories
\[ \oD^b_K(\mhm_\lambda^{(1)}(\ms{B} \times_{\ms{P}_\alpha} \ms{B} - \ms{B})) \cong \oD^b_K(\mhm_{s_\alpha \lambda}^{(2)}(\ms{B} \times_{\ms{P}_\alpha} \ms{B} - \ms{B})).\]
The Hecke functor $\mb{T}_{s_\alpha}^\lambda$ is the composition
\begin{align*}
\oD^b_K(\mhm_\lambda(\ms{B})) &\xrightarrow{p_1^\circ} \oD^b_K(\mhm_\lambda^{(1)}(\ms{B} \times_{\ms{P}_\alpha} \ms{B} - \ms{B})) \\
&\cong \oD^b_K(\mhm_{s_\alpha\lambda}^{(2)}(\ms{B} \times_{\ms{P}_\alpha} \ms{B} - \ms{B})) \\
&\xrightarrow{p_{2!}} \oD_K^b(\mhm_{s_\alpha\lambda}(\ms{B})),
\end{align*}
where $p_1^\circ = p_1^*[1]$.

\begin{rmk}
When $\lambda = 0$, the Hecke functors $\mb{T}_{s_\alpha}$ agree with the operators of \cite{LV} up to a cohomological shift by $1$ (after passage from mixed Hodge modules to mixed sheaves over a finite field).
\end{rmk}

\begin{defn}
Let $\lambda \in \mf{h}^*_\mb{R}$ and $\mu \in \mb{X}^*(H)$. The \emph{translation functor}
\[ t_\mu \colon \mhm_\lambda(K \bslash \ms{B}) \to \mhm_{\lambda + \mu}(K \bslash \ms{B}) \]
is given by
\[ t_\mu(\ms{M}) = \mc{O}_{\ms{B}}(\mu) \otimes \ms{M} = \ms{M} \star j_{1!}\mc{O}_{X_1}(-\lambda, \lambda + \mu).\]
\end{defn}

\begin{rmk} \label{rmk:intertwining functors}
From the perspective of representation theory, it is often natural to work with the \emph{intertwining functors}
\begin{equation} \label{eq:intertwining functor}
 \mb{I}_{w} = t_{w \rho - \rho} \mb{T}_{w}
\end{equation}
in place of the Hecke functors $\mb{T}_{w}$ (see \S\ref{subsec:intertwining} for a related discussion). All the results stated in this section for the $\mb{T}_w$ can easily be translated into results for the $\mb{I}_w$ using \eqref{eq:intertwining functor}.
\end{rmk}

The Hecke functors induce operators
\[ \mb{T}_w \colon \oK(\mhm_\lambda(K \bslash \ms{B})) \to \oK(\mhm_{w \lambda}(K \bslash \ms{B})).\]
We conclude this subsection by recording some key relations they satisfy.

\begin{prop} \label{prop:hecke relations}
The Hecke operators $\mb{T}_{s_\alpha}$ satisfy the relations
\begin{equation} \label{eq:hecke relation 1}
\mb{T}_{s_\alpha}^2 := \mb{T}_{s_\alpha}^{s_\alpha \lambda}\mb{T}_{s_\alpha}^\lambda = \begin{cases} u + (1 - u)t_{\lambda - s_\alpha\lambda}\mb{T}_{s_\alpha}, & \mbox{if $\alpha$ is integral,} \\ u, & \mbox{otherwise,}\end{cases}
\end{equation}
\begin{equation} \label{eq:hecke relation 2}
\mb{T}_{w}\mb{T}_{w'} = \mb{T}_{ww'} \quad \mbox{if $\ell(ww') = \ell(w) + \ell(w')$},
\end{equation}
and
\begin{equation} \label{eq:hecke relation 3}
(\mb{T}_{s_\alpha}[\ms{M}])^h = \begin{cases} u^{-1}(\mb{T}_{s_\alpha} + (u - 1)t_{s_\alpha \lambda - \lambda})[\ms{M}^h],& \mbox{if $\alpha$ is integral,} \\ u^{-1}\mb{T}_{s_\alpha}[\ms{M}^h], &\mbox{otherwise}\end{cases}
\end{equation}
for a simple root $\alpha$ and $w, w' \in W$.
\end{prop}
\begin{proof}
Relation \eqref{eq:hecke relation 1} follows by a similar calculation to Proposition \ref{prop:hecke action} \eqref{itm:hecke action 2} and \eqref{itm:hecke action 3}. The proof of relation \eqref{eq:hecke relation 2} is straightforward and well known. To prove \eqref{eq:hecke relation 3}, observe that we have
\begin{align*}
(\mb{T}_{s_\alpha}\ms{M})^h &= (\ms{M} \star j_{s_\alpha !} \mc{O}_{X_{s_\alpha}}(-\lambda, s_\alpha \lambda))^h \\
&= \ms{M}^h \star (j_{s_\alpha !} \mc{O}_{X_{s_\alpha}}(-\lambda, s_\alpha \lambda))^h (-\dim \tilde{\ms{B}} - \dim H) \\
&= \ms{M}^h \star j_{s_\alpha *} \mc{O}_{X_{s_\alpha}}(-\lambda, s_\alpha \lambda) (1),
\end{align*}
where in the passage from the first line to the second we have used the well known identifications $\pi_1^! = \pi_1^*(\dim H)[2 \dim H]$ and $\Delta_{12}^! = \Delta_{12}^* (-\dim \tilde{\ms{B}})[-2 \dim \tilde{\ms{B}}]$ (the latter holds equivariantly because the map of stacks
\[ \Delta_{12} \colon K \bslash (\tilde{\ms{B}} \times \tilde{\ms{B}}) \to K \bslash \tilde{\ms{B}} \times G \bslash (\tilde{\ms{B}} \times \tilde{\ms{B}})\]
is smooth.) The claimed relation now follows from the straightforward calculation
\[ [j_{s_\alpha *}\mc{O}_{X_{s_\alpha}}(-\lambda, s_\alpha \lambda) ] = [j_{s_\alpha !} \mc{O}_{X_{s_\alpha}}(-\lambda, s_\alpha \lambda)] + (u - 1)[j_{1!}\mc{O}_{X_1}(- \lambda, s_\alpha \lambda)],\]
if $\alpha$ is integral, and
\[  [j_{s_\alpha *}\mc{O}_{X_{s_\alpha}}(-\lambda, s_\alpha \lambda)] =  [j_{s_\alpha !} \mc{O}_{X_{s_\alpha}}(-\lambda, s_\alpha \lambda)],\]
otherwise.
\end{proof}

\subsection{The Hodge filtration of a pushforward} \label{subsec:hodge filtration}

Let $f \colon X \to Y$ be a morphism of smooth varieties. In this subsection, we recall how the Hodge filtration of a mixed Hodge module behaves under pushforward by $f$.

If $f \colon X \to Y$ is proper then the Hodge filtrations are pushed forward using a canonically defined functor $f_* \colon \od(\cD_X)_{frh}\to \od(\cD_Y)_{frh}$; here we have written $\od(\cD_X)_{frh}$ for the derived category consisting of complexes of filtered $\cD_X$-modules whose cohomologies are filtered regular holonomic $\cD_X$-modules. Let us briefly recall this general construction. For a detailed treatment, see \cite{laumon1983}. As we can decompose $f$ into a closed embedding of $X$ in $X\times Y$ via its graph followed by a projection $X\times Y \to Y$, it suffices to consider these two special cases. 

Let $i\colon X \to Y$ be closed embedding and let $\cM$ be a filtered $\cD_X$-module. By definition, we have
\[ i_*\ms{M} = i_{\bigcdot}(\ms{D}_{Y \leftarrow X} \otimes_{\ms{D}_X} \ms{M}),\]
as in \S\ref{subsec: phm}.
If we work in local coordinates, we may assume that $X$ is cut out by the equations $y_1=0, \dots ,y_k=0$ and write
\begin{equation}
i_*\cM \ = \  \cM[\partial_{y_1},\dots, \partial_{y_k}]\partial_{y_1} \wedge \cdots \wedge \partial_{y_k}\,,
\end{equation}
where we have inserted $\partial_{y_1} \wedge \cdots \wedge \partial_{y_k}$ as a reminder that the pushforward involves a twist by the top exterior power of the normal bundle. 
We filter it as follows:
\begin{equation}
\label{closed embedding formula}
F_p(i_*\cM) \ = \  \sum_{q+|\alpha|\leq p-k}F_q\cM\, \partial_{y_1}^{\alpha_1}\dots \partial_{y_k}^{\alpha_k}\, \ \ \ |\alpha| = \alpha_1+\cdots + \alpha_k\,.
\end{equation}
Note that if $F_{\bigcdot}\mc{M}$ has lowest non-zero piece $F_p\mc{M}$, then $F_{\bigcdot} i_*\mc{M}$ has lowest non-zero piece $F_{p + k}\mc{M}$.

Let us now consider a projection $\pi \colon X \times Y \to Y$ and let $\cM$ be a filtered $\cD_{X\times Y}$-module such that $\pi$ is proper on the support of $\mc{M}$. Recall that 
\begin{equation}
\pi_*\cM \ = \  R\pi_{\bigcdot} (\cM\otimes\Omega_X^{\bigcdot})[n]\,
\end{equation}
where $n$ is the dimension of $X$, $R\pi_{\bigcdot}$ is the derived push-forward in the sense of sheaves of complex vector spaces, and the shift by $n$ places the complex in degrees $-n, \ldots, 0$. The filtration is given as follows:
\begin{equation}
F_p(\pi_*\cM) \ = \  R\pi_{\bigcdot}(F_p \cM \to F_{p+1}\cM\otimes \Omega_X^{1} \to \dots \to  F_{p+n}\cM\otimes \Omega_X^{n})[n]\,.
\end{equation}

To specify the Hodge filtration of a general pushforward of a mixed Hodge module, it therefore remains to do so for open embeddings $j \colon U \to X$. The question is local, so it suffices to assume that the complement $X-U$ is a principal divisor given by $g=0$. Saito's theory implies that in this case the filtration is given (not surprisingly) by a $\mc{D}$-module version of ``order of pole'', supplied by the Kashiwara-Malgrange $V$-filtration.\footnote{This is not to be confused with more naive notions of pole order such as \cite[(0.8)]{Saito93}, whose relation to the Hodge filtration is much more subtle.}

Let $\cM$ be a mixed Hodge module on $U$. Let us write $i\colon X \hookrightarrow X \times \bC $ for the graph of $g$ and $\tilde i\colon U \to X \times \bC^\times$ for the closed embedding given by restriction of $i$. We write $\tilde j\colon X \times \bC^\times \hookrightarrow X \times \bC$ for the open inclusion and $t$ for the local coordinate on $\bC$. Since $i_*j_*\ms{M}$ is a mixed Hodge module on $X \times \bC$, it has a well-defined decreasing $V$-filtration, indexed by $\mb{R}$, with respect to the divisor given by $t = 0$ (see \S\ref{subsec:nearby cycles}). The Hodge filtration on $j_*\ms{M}$ is determined by
\begin{equation}
 \label{open embedding formula}
F_p i_*j_*\cM \ = \ \sum \partial_t^k (V^{\geq -1}i_*j_*\cM \cap \tilde j_{\bigcdot} F_{p-k}\tilde i_*\cM)
\end{equation}
together with \eqref{closed embedding formula}. This is \cite[Definitions 11.3.1]{SS}. Dually, the mixed Hodge module $i_*j_!\ms{M}$ has a $V$-filtration such that the canonical map
\[ V^{>-1}i_*j_!\ms{M} \to V^{>-1}i_*j_*\ms{M}\]
is an isomorphism. The Hodge filtration on $j_!\ms{M}$ is determined by
\begin{equation}
\label{dual open embedding formula}
F_p i_*j_!\cM \ = \ \sum \partial_t^k (V^{> -1}i_*j_!\cM \cap \tilde j_{\bigcdot} F_{p-k}\tilde i_*\cM).
\end{equation}
This is \cite[Definition 11.4.1]{SS}. This implies in particular that if $F_{\bigcdot} \mc{M}$ has lowest non-zero piece $F_p\mc{M}$, then $F_{\bigcdot} j_!\mc{M}$, $F_{\bigcdot} j_*\mc{M}$ and hence $F_{\bigcdot} j_{!*}\mc{M}$ have lowest non-zero pieces $F_p$ as well.

\subsection{The Hecke functors in the standard basis} \label{subsec:hecke action}

In this subsection, we write down the action of the Hecke functors in the $\{[j_!\gamma]\}$ basis for the $\oK$-group.

In what follows, for any $x \in \ms{B}$, we choose a $\theta$-stable maximal torus $T_x$ fixing $x$. For $w \in W$, we write $wx = \bar{w}x$ for any $\bar{w} \in N_G(T_x)$ lifting $w$; this may depend on the choice of $T_x$. We also recall the notation of \S\ref{subsec:parametrization}.

Fix a simple root $\alpha$ and a line $Z \cong \mb{P}^1 \subset \ms{B}$ of type $\alpha$ (i.e., a fiber of $\pi_\alpha \colon \ms{B} \to \ms{P}_\alpha$).

\begin{prop} \label{prop:orbit geometries}
There is exactly one $K$-orbit $K x$, $x \in Z$, such that $Kx \cap Z \subset Z$ is open. Moreover, one of the following holds.
\begin{enumerate}
\item \label{itm:orbit geometries 1} $Kx \cap Z = Z$ and $\alpha$ is $Kx$-compact imaginary.
\item \label{itm:orbit geometries 2} $Kx \cap Z = Z - \{y\}$, $\alpha$ is both $Kx$-complex and $Ky$-complex, $\theta_y\alpha \in \Phi_+$, $\theta_x \alpha \in \Phi_-$, and $\theta_x = s_\alpha\theta_ys_\alpha$.
\item \label{itm:orbit geometries 3} $Kx \cap Z = Z - \{y_+, y_-\}$ for $y_+ \neq y_-$, $\alpha$ is $Kx$-real and $Ky_\pm$-non-compact imaginary, and we have $\theta_{y_+} = \theta_{y_-} = \theta_xs_\alpha$. Moreover,
\begin{enumerate}
\item \label{itm:orbit geometries 3:I} if $Ky_+ \neq Ky_-$, then $H^{\theta_x} \subset H^{\theta_{y_\pm}}$, and $H^{\theta_{y_\pm}}/H^{\theta_x} \cong \mb{G}_m$, and
\item \label{itm:orbit geometries 3:II} if $Ky_+ = Ky_-$, then $H^{\theta_x} \cap H^{\theta_{y_{\pm}}}$ has index $2$ in $H^{\theta_x}$ and $H^{\theta_{y_\pm}}/H^{\theta_x} \cap H^{\theta_{y_\pm}} \cong \mb{G}_m$.
\end{enumerate}
\end{enumerate}
\end{prop}
Cases \eqref{itm:orbit geometries 3:I} and \eqref{itm:orbit geometries 3:II} are called ``type I'' and ``type II'' in \cite{ic3}.
\begin{proof}
The possible geometries and the corresponding root types are given in \cite[Lemma 5.1]{ic3}. In \eqref{itm:orbit geometries 2}, the relation $\theta_x = s_\alpha\theta_ys_\alpha$ holds because $y = s_\alpha x$. To see the relation $\theta_x = \theta_{y_\pm}s_\alpha$ in \eqref{itm:orbit geometries 3}, we may assume without loss of generality that $y_\pm = gx$ and $T_{y_\pm} = gT_xg^{-1}$ for
\[ g = \phi_{s_\alpha, x}\left(\frac{1}{\sqrt{2}}\left(\begin{matrix} 1 & i \\ i & 1 \end{matrix}\right)\right),\]
where
\[ \phi_{s_\alpha, x} \colon (SL_2, ((-)^t)^{-1}) \to (G, \theta) \]
is the root subgroup at $x$. So $\theta(g)^{-1}g \in s_\alpha T_x$, whence the desired relation follows.

Finally, to see \eqref{itm:orbit geometries 3:I} and \eqref{itm:orbit geometries 3:II}, note that
\[ \tau_x(\mrm{Stab}_K(x) \cap \mrm{Stab}_K(y_\pm)) = \tau_{y_\pm}(\mrm{Stab}_K(x) \cap \mrm{Stab}_K(y_\pm)) = H^{\theta_x} \cap H^{\theta_{y_\pm}}.\]
(The inclusion $\subset$ is obvious, and the inclusion $\supset$ follows from the fact that $T_{y_\pm}^{\theta} \cap T_{y_\pm}^{s_\alpha\theta} \subset T_{y_\pm}^{s_\alpha}$ acts trivially on $Z$.) The desired statements now follow, since
\[ \frac{\mrm{Stab}_K(x)}{\mrm{Stab}_K(x) \cap \mrm{Stab}_K(y_\pm)} \cong Ky_\pm \cap Z \quad \text{and} \quad \frac{\mrm{Stab}_K(y_\pm)}{\mrm{Stab}_K(x) \cap \mrm{Stab}_K(y_\pm)} \cong Kx \cap Z\]
by the orbit-stabilizer theorem.
\end{proof}

Recall for the following proposition that we write $u = t_1t_2 \in \oK(\mhm(\mrm{pt}))$, and that $\Lambda$ is regarded as a character of $H^{\theta_x} \subset H$ in the notation $\mc{O}(\lambda, \Lambda, x)$ for an irreducible twisted local system on $Kx$.

\begin{prop} \label{prop:hecke action}
Suppose we are in the setting of Proposition \ref{prop:orbit geometries}, and fix $\lambda \in \mf{h}^*_\mb{R}$. Then we have the following.
\begin{enumerate}
\item \label{itm:hecke action 1} If $Z$, $\alpha$ are as in Proposition \ref{prop:orbit geometries}, \eqref{itm:orbit geometries 1} then there are no $\lambda$-twisted local systems $\gamma$ on $Kx$ unless $\alpha$ is integral. In this case, we have
\[ \mb{T}_{s_\alpha}[j_!\gamma] = -ut_{s_\alpha \lambda - \lambda}[j_!\gamma]\]
for all such $\gamma$.
\item \label{itm:hecke action 2} If $Z$, $\alpha$ are as in Proposition \ref{prop:orbit geometries}, \eqref{itm:orbit geometries 2} and $\alpha$ is integral, then
\[ \mb{T}_{s_\alpha}[j_!\gamma] = (1 - u)t_{s_\alpha \lambda - \lambda}[j_!\gamma] + u[j_!\gamma'] \quad \text{and} \quad \mb{T}_{s_\alpha}[j_!\gamma'] = [j_!\gamma],\]
for $\gamma = \mc{O}(\lambda, \Lambda, x)$ and $\gamma' = \mc{O}(s_\alpha\lambda, s_\alpha\Lambda, y)$.
\item \label{itm:hecke action 3} If $Z$, $\alpha$ are as in Proposition \ref{prop:orbit geometries}, \eqref{itm:orbit geometries 2} and $\alpha$ is non-integral, then
\[ \mb{T}_{s_\alpha}[j_!\gamma] = u[j_!\gamma'] \quad \text{and} \quad \mb{T}_{s_\alpha}[j_!\gamma'] = [j_!\gamma],\]
for $\gamma = \mc{O}(\lambda, \Lambda, x)$ and $\gamma' = \mc{O}(s_\alpha\lambda, s_\alpha\Lambda, y)$.
\item \label{itm:hecke action 4} If $Z$, $\alpha$ are as in Proposition \ref{prop:orbit geometries}, \eqref{itm:orbit geometries 3:I} and $\alpha$ is integral, then
\[ \mb{T}_{s_\alpha}[j_!\gamma] = (2 - u)t_{s_\alpha\lambda - \lambda}[j_!\gamma] + (u - 1)([j_!\gamma_+'] + [j_!\gamma_-'])\]
and
\[ \mb{T}_{s_\alpha}[j_!\gamma_{\pm}'] = [j_!\gamma] - t_{\lambda - s_\alpha\lambda}[j_!\gamma_{\mp}'],\]
for
\[ \gamma_{\pm}' = \mc{O}(s_\alpha\lambda, \Lambda', y_{\pm}), \quad \text{and} \quad \gamma = \mc{O}(\lambda, \Lambda'|_{H^{\theta_x}}, x).\]
\item \label{itm:hecke action 5} If $Z$, $\alpha$ are as in Proposition \ref{prop:orbit geometries}, \eqref{itm:orbit geometries 3:II} and $\alpha$ is integral, then
\[ \mb{T}_{s_\alpha}[j_!\gamma_\pm] = t_{s_\alpha\lambda - \lambda}((1 - u)[j_!\gamma_\pm] + [j_!\gamma_{\mp}]) + (u - 1)[j_!\gamma'] \]
and
\[ \mb{T}_{s_\alpha}[j_!\gamma'] = [j_!\gamma_+] + [j_!\gamma_-] - t_{\lambda - s_\alpha\lambda}[j_!\gamma']\]
for
\[ \gamma' = \mc{O}(s_\alpha\lambda, \Lambda', y_+) \quad \text{and} \quad \gamma_\pm = \mc{O}(\lambda, \Lambda_\pm, x)\]
where
\[ \Lambda_+ \oplus \Lambda_- = \mrm{Ind}_{H^{\theta_{y_+}} \cap H^{\theta_x}}^{H^{\theta_x}}\left(\Lambda'|_{H^{\theta_{y_+}} \cap H^{\theta_x}}\right).\]
\item \label{itm:hecke action 6} If $Z$, $\alpha$ are as in Proposition \ref{prop:orbit geometries}, \eqref{itm:orbit geometries 3}, $\alpha$ is integral, and $\gamma = \mc{O}(\lambda, \Lambda, x)$ is such that $\Lambda(m_\alpha) \neq (-1)^{\langle \lambda, \check\alpha\rangle}$, then
\[ \mb{T}_{s_\alpha}[j_!\gamma] = t_{s_\alpha \lambda - \lambda}[j_!\gamma].\]
\item \label{itm:hecke action 7} If $Z$, $\alpha$ are as in Proposition \ref{prop:orbit geometries}, \eqref{itm:orbit geometries 3} and $\alpha$ is non-integral, then there are no $\lambda$-twisted local systems on $Ky_\pm$ and
\[ \mb{T}_{s_\alpha}[j_!\gamma] = \begin{cases} t_1[j_!\gamma'], & \mbox{if $(-1)^{\lfloor \langle \lambda, \check\alpha\rangle\rfloor}\Lambda(m_\alpha) = 1$,} \\ t_2[j_!\gamma'], &\mbox{if $(-1)^{\lfloor \langle \lambda, \check\alpha\rangle\rfloor}\Lambda(m_\alpha) = -1$},\end{cases}\]
where
\[ \gamma = \mc{O}(\lambda, \Lambda, x) \quad \text{and} \quad \gamma' = \mc{O}(s_\alpha\lambda, s_\alpha\Lambda \otimes \alpha, x).\]
\end{enumerate}
Any $\lambda$-twisted local system on $Kx$ and any $s_\alpha\lambda$-twisted local system on another $K$-orbit meeting $Z$ appears exactly once in the above list.
\end{prop}

\begin{rmk}
Cases \eqref{itm:hecke action 4} and \eqref{itm:hecke action 5} exactly describe those twisted local systems that extend to twisted local systems on $Z$ in the case when $\alpha$ is real.
\end{rmk}

\begin{proof}
For a $\lambda$-twisted local system $\delta$ on an orbit $Q$ meeting $Z$, we may write
\[ \mb{T}_{s_\alpha}[j_!\delta] = \sum_{Q' \cap Z \neq \emptyset} [j_{Q'!}j_{Q'}^*\mb{T}_{s_\alpha}j_!\delta]. \]
By base change, we may rewrite this as
\begin{equation} \label{eq:hecke action 1}
 \mb{T}_{s_\alpha}[j_!\delta] = \sum_{Q' \cap Z \neq \emptyset} -[j_{Q'!} q_{2!}\xi q_{2}^*\delta],
\end{equation}
where $q_i$ are the projections
\[ Q \xleftarrow{q_1} Q \times_{\ms{P}_\alpha} Q' - Q \cap Q' \xrightarrow{q_2} Q'\]
and
\[ \xi \colon \mhm_\lambda^{(1)}(K \bslash (Q \times_{\ms{P}_\alpha} Q' - Q \cap Q')) \overset{\sim}\to \mhm_{s_\alpha \lambda}^{(2)}(K \bslash (Q \times_{\ms{P}_\alpha} Q' - Q \cap Q')) \]
is the equivalence of categories induced by \eqref{eq:untwisting iso}. Finally, choosing a point $z \in Q'$ for each orbit, we may simplify \eqref{eq:hecke action 1} to
\begin{equation} \label{eq:hecke action 2}
\mb{T}_{s_\alpha} [j_!\delta] = \sum_{Kz \cap Z \neq \emptyset} (-1)^{\dim (Kz \cap Z) - 1} [j_{z!}\pi_{z!}(\delta \otimes \mc{O}(-\lambda z))],
\end{equation}
where $\pi_z \colon Q \cap Z - \{z\} \to \mrm{pt}$ is the projection to a point, we interpret the complex $\pi_{z!}$ as the Harish-Chandra module at $z$ defining an $s_\alpha\lambda$-twisted local system on $Kz$, and we have written the equivalence $\xi$ explicitly as the tensor product with the (unique) rank $1$ $(-\lambda)$-twisted local system $\mc{O}(-\lambda z)$ on $Z - \{z\}$.

With these preliminaries, we observe that the integral cases \eqref{itm:hecke action 1}, \eqref{itm:hecke action 2}, \eqref{itm:hecke action 4}, \eqref{itm:hecke action 5} and \eqref{itm:hecke action 6} are essentially the same calculations as \cite[Lemma 3.5]{LV}, with appropriate modifications to take into account translations and our different sign conventions. (Recall that our bases and Hecke operators are defined using $\ms{D}$-modules and intermediate pullback, which are cohomologically shifted compared with the constructible sheaves and $*$-pullback used in \cite{LV}.) Note that in these cases, the cohomology groups appearing in \eqref{eq:hecke action 2} are either zero, the cohomology of the trivial local system on $Z = \mb{P}^1$ minus $1$ or $2$ points, or the cohomology of a non-trivial local system on $\mb{P}^1$ minus $3$ points with trivial monodromy around one of the punctures; these are all Hodge-Tate, hence the coefficients are all powers of $u$.

For case \eqref{itm:hecke action 3}, we have since $Ky \cap Z - \{y\} = \emptyset$,
\[ \mb{T}_{s_\alpha}[j_!\gamma'] = [j_{x!}\pi_{x!} (\gamma \otimes \mc{O}(-s_\alpha\lambda x))]\]
with $\pi_x \colon \{y\} \to \mrm{pt}$. If we write $y = s_\alpha x$, then $T_x^\theta = T_y^\theta$ acts on this $1$-dimensional vector space through the character $(s_\alpha\Lambda) \tau_y = \Lambda\tau_x$, so $\mb{T}_{s_\alpha}[j_!\gamma'] = [j_!\gamma]$ as claimed. Similarly,
\[ \mb{T}_{s_\alpha}[j_!\gamma] = [j_{x!}\pi_{x!} (\gamma \otimes \mc{O}(-\lambda x))] - [j_{y!}\pi_{y!}(\gamma \otimes \mc{O}(-\lambda y))].\]
The local system $\gamma \otimes \mc{O}(-\lambda x)$ on $Kx \cap Z - \{x\} \cong \mb{G}_m$ is non-trivial, so it has vanishing cohomology. The local system $\gamma \otimes \mc{O}(-\lambda y)$ on $Kx \cap Z - \{y\} = Kx \cap Z \cong \mb{A}^1$ is trivial, so its $\pi_!$ is $\mb{C}(-1)[-1]$. Computing the $H^{\theta_y}$-action as above gives $\mb{T}_{s_\alpha}[j_!\gamma] = u[j_!\gamma']$ as claimed.

For case \eqref{itm:hecke action 7}, the non-existence of $\lambda$-twisted local systems on $Ky_\pm$ is obvious since $\check\alpha(\mb{G}_m) \subset H^{\theta_{y_\pm}}$. To compute the Hecke action, we write
\[ \mb{T}_{s_\alpha}[j_!\gamma] = [j_{s_\alpha x!} \pi_{s_\alpha x !}(\gamma \otimes \mc{O}(-\lambda s_\alpha x))],\]
where we note that $Kx = Ks_\alpha x$. Choose a coordinate $z$ on $Z - \{s_\alpha x\} \cong \mb{A}^1$, and write
\[ \gamma \otimes \mc{O}(-\lambda s_\alpha x) = \mb{C}[z, (z - a)^{-1}, (z - b)^{-1}](z - a)^{\mu_1}(z - b)^{\mu_2}\]
as in Lemma \ref{lem:real local system} below. By Lemma \ref{lem:real hodge calculation}, we obtain that
\begin{align*}
 \pi_{s_\alpha x !}(\gamma \otimes \mc{O}(-\lambda s_\alpha x)) &= \pi_{s_\alpha x *}(\gamma \otimes \mc{O}(-\lambda s_\alpha x)) \\
&= \begin{cases} \mb{C}_{-1, 0}, & \mbox{if $(-1)^{\lfloor \langle \lambda, \check\alpha \rangle \rfloor} \Lambda(m_\alpha) = 1$,} \\ \mb{C}_{0, -1}, & \mbox{if $(-1)^{\lfloor \langle \lambda, \check\alpha \rfloor} \Lambda(m_\alpha) = -1$},\end{cases}
\end{align*}
as Hodge structures.

To compute the action of $T_{s_\alpha x}^\theta = T_x^\theta$, we observe that
\begin{equation} \label{eq:hecke action 3}
(z - a)^{\mu_1}(z - b)^{\mu_2}dz \in H^0(Kx \cap Z - \{s_\alpha \lambda\}, \gamma \otimes \mc{O}(-\lambda s_\alpha x) \otimes \Omega^1_Z)
\end{equation} 
is a generator for the non-zero cohomology group of $\pi_{s_\alpha x *}$. The character $\alpha\tau_x \colon T^\theta_x \to \mb{G}_m$ factors through $\{\pm 1\}$ (since $\theta_x \alpha = -\alpha$), and its kernel is $T^\theta_x \cap \mrm{Stab}(y_{\pm})$. This kernel acts trivially on $Z$ (and hence on $dz$), so must act on the above generator \eqref{eq:hecke action 3} by the character $\Lambda\tau_x$. On the other hand, if $g \in T^\theta_x$ and $\alpha \tau_x(g) = -1$, then $gy_+ = y_-$, so $g \cdot dz = - dz$ and $g \cdot (z - a)^{\mu_1}(z - b)^{\mu_2} = A (z - a)^{\mu_2}(z - b)^{\mu_1}$ for some $A \in \mb{C}$. We may choose the coordinate $z$ so that $a = 1$, $b = -1$ and $z = 0$ is the coordinate of $x$; with this choice, evaluating at $z = 0$ gives $A = (-1)^{\mu_1 - \mu_2}\Lambda(g)$. Finally, we have
\[ [(z - a)^{\mu_2}(z - b)^{\mu_1}dz] = (-1)^{\mu_1 - \mu_2}[(z - a)^{\mu_1}(z - b)^{\mu_2}dz] \]
in cohomology. So $g$ acts by $-\Lambda\tau_x(g) = (\Lambda \otimes \alpha)\tau_x(g)$. Combining the above, we have that $T^\theta_x$ acts by $(\Lambda \otimes \alpha)\tau_x = (s_\alpha \Lambda \otimes \alpha)\tau_{s_\alpha x}$, which completes the proof.
\end{proof}

\begin{lem} \label{lem:real local system}
In the setting of Proposition \ref{prop:hecke action} \eqref{itm:hecke action 7}, if we choose a coordinate $z$ on $Z - \{s_\alpha x\} \cong \mb{A}^1$, then
\[ \gamma \otimes \mc{O}(-\lambda s_\alpha x) \cong \mb{C}[z, (z - a)^{-1}, (z - b)^{-1}](z - a)^{\mu_1}(z - b)^{\mu_2}\]
where $a$ and $b$ are the $z$-coordinates of $y_{\pm}$, $\mu_i \in \mb{R}$ satisfy $\mu_1 + \mu_2 = \langle \lambda, \check\alpha\rangle$, $\mu_1 - \mu_2 \in \mb{Z}$, and $(-1)^{\mu_1 - \mu_2} = \Lambda(m_\alpha)$.
\end{lem}
\begin{proof}
Consider the pullback $\ms{M}$ of $\gamma$ along the morphism
\[ \mb{C}^2 - \{0\} = SL_2/N \to G/N = \tilde{\ms{B}}\]
given by the $\alpha$-root subgroup at $y_+$. (Note that $Z$ is the image of $\mb{C}^2 - \{0\}$ in $\ms{B}$.) Choose a section $m$ of $\ms{M}$ such that $m$ (resp., the line $\mb{C}m$) is invariant under the weak $H$-action (resp., the $T_{y_+}^\theta$-action); let $\Theta$ be the character of $T_{y_+}^\theta$ acting on $m$. We have
\[ \ms{M} = \mb{C}[u, v, u^{-1}, v^{-1}]m \]
where $u$ and $v$ are coordinates on $\mb{C}^2$ such that $y_+$ is the image of $u = 0$, $y_-$ is the image of $v = 0$ and $s_\alpha x$ is the image of $u = v$. The conditions on the weak $H$-action and strong $K$-action imply
\[ (u \partial_u + v \partial_v) m = \langle \lambda, \check\alpha \rangle m \quad \text{and} \quad (u \partial_u - v \partial_v) m = \langle \Theta, \tau_{y_+}^{-1}\check\alpha\rangle m.\]
Tensoring with $\mc{O}(-\lambda s_\alpha x)$ is equivalent to pulling back along the map
\begin{align*}
Z - \{s_\alpha x\} = \mb{A}^1 &\to \mb{C}^2 - \{0\} \\
z &\mapsto (z - a, z - b),
\end{align*}
so $\gamma \otimes \mc{O}(-s_\alpha\lambda) = \mb{C}[z, (z - a)^{-1}, (z - b)^{-1}]m$, where $m$ satisfies
\[
\partial_z m = \left(\frac{1}{z - a} u\partial_u + \frac{1}{z - b}v\partial_v\right) m = \left(\frac{\mu_1}{z - a} +  \frac{\mu_2}{z - b}\right)m,
\]
where
\[ \mu_1 + \mu_2 = \langle \lambda, \check\alpha\rangle \quad \text{and} \quad \mu_1 - \mu_2 = \langle \Theta, \tau_{y_+}^{-1}\check\alpha\rangle \in \mb{Z}.\]
So the expression in the statement holds, with $m$ identified with $(z - a)^{\mu_1}(z - b)^{\mu_2}$. To deduce that $\Lambda(m_\alpha) = (-1)^{\mu_1 - \mu_2}$, note that $T_{y_+}^\theta \cap \mrm{Stab}_K(x) = (T_{y_+}^\theta)^{s_\alpha}$ acts on the fiber at $x$ by the restriction of $\Theta$ on the one hand, and by the restriction of $\Lambda \tau_x$ on the other. Hence
\[ \Theta|_{(T_{y_+}^\theta)^{s_\alpha}} = \Lambda\tau_x|_{(T_{y_+}^\theta)^{s_\alpha}} = \Lambda s_\alpha \tau_{y_+}|_{(T_{y_+}^\theta)^{s_\alpha}} = \Lambda \tau_{y_+}|_{(T_{y_+}^\theta)^{s_\alpha}},\]
where we have used the fact that $\tau_x = s_\alpha \tau_{y_+}$ on the intersection $B_x \cap B_{y_+}$. The desired claim now follows since $m_\alpha \in (H^\theta_{y_+})^{s_\alpha}$.
\end{proof}

\begin{lem} \label{lem:real hodge calculation}
Let $a \neq b \in \mb{A}^1$ and let $\ms{M}$ be the $\ms{D}$-module on $\mb{A}^1 - \{a, b\}$ given by $\ms{M} = \mb{C}[z, (z - a)^{-1}, (z - b)^{-1}](z - a)^{\mu_1}(z - b)^{\mu_2}$ equipped with its standard Hodge structure, for $\mu_1, \mu_2 \in \mb{R}$ with $\mu_1 + \mu_2 \not\in \mb{Z}$ and $\mu_1 - \mu_2 \in \mb{Z}$. Then we have $j_!\ms{M} = j_*\ms{M} = j_{!*}\ms{M}$, and
\[ H^i(\pi_*j_!\ms{M}) = \begin{cases} 0, & \mbox{if $i \neq 0$}, \\ \mb{C}_{-1, 0}, & \mbox{if $i = 0$, $(-1)^{\lfloor \mu_1 + \mu_2 \rfloor} = (-1)^{\mu_1 - \mu_2}$}, \\ \mb{C}_{0, -1}, &\mbox{otherwise,} \end{cases} \]
where $j \colon \mb{A}^1 - \{a, b\} \to \mb{P}^1$ is the inclusion and $\pi \colon \mb{P}^1 \to \mrm{pt}$ the projection.
\end{lem}
\begin{proof}
The conditions on $\mu_1, \mu_2$ imply that $\ms{M}$ has non-trivial monodromy around each puncture $a, b, \infty$, so $j_!\ms{M} = j_*\ms{M}$ as claimed.

Next, observe that since the local system $\ms{M}$ is non-trivial, the space $H^{-1}(\pi_* j_*\ms{M})$ of flat sections is zero. Since $j_*\ms{M} = j_{!}\ms{M}$, we have $(j_*\ms{M})^h \cong j_*\ms{M}$ as $\ms{D}$-modules, so we have
\[ H^1(\pi_* j_*\ms{M}) = H^{-1}(\pi_*j_*\ms{M})^h = 0\]
also. Moreover, since $j_*\ms{M}$ is a pure Hodge module of weight $1$, the mixed Hodge structure $H^0(\pi_*j_*\ms{M})$ is also pure of weight $1$, so it remains to compute its Hodge filtration.

We can assume without loss of generality that $-1 < \mu_i < 0$, since modifying either of the $\mu_i$ by an integer does not change the $\ms{D}$-module $\ms{M}$ or any of the conditions in the statement. Note that this assumption implies $\mu_1 = \mu_2$. By \eqref{open embedding formula}, the Hodge filtration $F_{\bigcdot}$ on $j_*\ms{M}$ is given in terms of the $V$-filtration by
\[ F_p j_*\ms{M} = \begin{cases} 0, & \mbox{if $p < 0$}, \\ V^{>-p -1} j_*\ms{M}, & \mbox{if $p \geq 0$}.\end{cases}\]
Writing $w = z^{-1}$, we have $(z - a)^{\mu_1}(z - b)^{\mu_2} = (1 - aw)^{\mu_1}(1 - bw)^{\mu_2}w^{-\mu_1 - \mu_2}$, so we deduce
\[ F_p j_*\ms{M} = \mc{O}(pa + pb + (p - \lfloor \mu_1 + \mu_2 \rfloor)\infty)(z - a)^{\mu_1}(z - b)^{\mu_2}  \quad \mbox{for $p \geq 0$}. \]
The Hodge filtration on $H^0(\pi_* j_*\ms{M})$ is given by
\[
 F_p H^0(\pi_* j_*\ms{M}) = \mb{H}^0(\mb{P}^1, [F_p j_*\ms{M} \to F_{p + 1} j_*\ms{M} \otimes \Omega_{\mb{P}^1}^1]).
\]
For $p \geq 0$, this is 
\begin{align*}
F_pH^0(\pi_*j_*\ms{M})&= \coker(H^0(\mb{P}^1, \mc{O}(3p - \lfloor \mu_1 + \mu_2 \rfloor)) \\
& \qquad \qquad \to H^0(\mb{P}^1, \mc{O}(3p - \lfloor \mu_1 + \mu_2 \rfloor + 1)) \cong \mb{C},
\end{align*}
since $\lfloor \mu_1 + \mu_2 \rfloor \in \{-1, -2\}$. For $p = -1$, we have
\[ F_{-1} H^0(\pi_*j_*\ms{M}) = H^0(\mb{P}^1, \mc{O}(- \lfloor \mu_1 + \mu_2\rfloor - 2)) = \begin{cases} \mb{C}, & \mbox{if $-2 < \mu_1 + \mu_2 < -1$,} \\ 0, & \mbox{if $-1 < \mu_1 + \mu_2 < 0$,} \end{cases}\]
and for $p < -1$ clearly $F_pH^0(\pi_*j_*\ms{M}) = 0$. Since $H^0(\pi_*j_*\ms{M})$ is a Hodge structure of weight $1$, the conclusion of the lemma now follows.
\end{proof}

\subsection{The mixed change of basis matrix} \label{subsec:lusztig-vogan}

In this subsection, we present a small tweak of the results of \cite{LV} and \cite{ic3} adapted to our context, which give an algorithm for computing the mixed change of basis matrix $\LVM^m$.

Consider the mixed $\oK$-groups
\[ \oK^m(\mhm_\lambda(K \bslash \ms{B})) := \mb{Z}[u^{\pm \frac{1}{2}}] \otimes_{\mb{Z}[t_1^{\pm 1}, t_2^{\pm 1}]} \oK(\mhm_\lambda(K \bslash \ms{B})) \]
with $t_1$ and $t_2$ sent to $u^{\frac{1}{2}}$, and the Hecke functors
\[ \mb{T}_{s_\alpha} = \mb{T}_{s_\alpha}^\lambda \colon \oK^m(\mhm_\lambda(K \bslash \ms{B})) \to \oK^m(\mhm_{s_\alpha \lambda}(K \bslash \ms{B})) \]
for each simple root $\alpha$. Suppose we are given an abelian group $A$ with an element $u^{\frac{1}{2}}$ of infinite order and an involution $D \colon A \to A$ sending $u^{\frac{1}{2}}$ to $u^{-\frac{1}{2}}$. Set
\[ M_\lambda = \mb{Z}[A] \otimes_{\mb{Z}[u^{\pm \frac{1}{2}}]} \oK^m(\mhm_\lambda(K \bslash \ms{B}))\]
and write
\[ B_\gamma = u^{\frac{1}{2}(\ell_I(\gamma) - \ell(\gamma) - \dim H)}[j_!\gamma] \in M_\lambda,\]
where we recall from Remark \ref{rmk:weights} that the Hodge module $\gamma$ has weight $\ell(\gamma) + \dim H$ according to our conventions. We also consider the closure partial order $\gamma' = (Q', \gamma') < \gamma = (Q, \gamma)$ if $Q' \subset \bar{Q} - Q$.

\begin{prop} \label{prop:our duality}
In the setup above, there exists a unique system of $\mb{Z}$-linear maps $D \colon M_\lambda \to M_\lambda$ with matrix
\[ D(B_\gamma) = u^{-\ell_I(\gamma)}\sum_{\gamma'} \mbf{R}_{\gamma', \gamma} B_{\gamma'}, \quad \mbf{R}_{\gamma', \gamma} \in \mb{Z}[A] \]
such that
\begin{enumerate}
\item \label{itm:our duality 1} $D(a m) = D(a) D(m)$ for $a \in A$, $m \in M_\lambda$,
\item \label{itm:our duality 2} if $\alpha$ is a simple root then
\[ D(\mb{T}_{s_\alpha}m) = \begin{cases} u^{-1}(\mb{T}_{s_\alpha} + (u - 1)t_{s_\alpha \lambda - \lambda})D(m), & \mbox{if $\alpha$ is integral}, \\ u^{-1}\mb{T}_{s_\alpha}D(m), &\mbox{otherwise,}\end{cases} \]
\item \label{itm:our duality 2a} if $\mu \in \mb{X}^*(H)$, then $D(t_\mu m) = t_\mu D(m)$,
\item \label{itm:our duality 3} $\mbf{R}_{\gamma, \gamma} = 1$, and
\item \label{itm:our duality 4} $\mbf{R}_{\gamma', \gamma} = 0$ unless $\gamma' \leq \gamma$ in the closure partial order.
\end{enumerate}
Moreover, the $\mbf{R}_{\gamma', \gamma}$ are polynomials in $u$ of degree at most $\ell_I(\gamma) - \ell_I(\gamma')$.
\end{prop}

We give a proof of Proposition \ref{prop:our duality} at the end of the subsection. We simply note here that the map $D$ is given by Hermitian duality
\[ (-)^h \colon \oK^m(\mhm_\lambda(K \bslash \ms{B})) \to \oK^m(\mhm_{\lambda}(K \bslash \ms{B}))\]
extended via \eqref{itm:our duality 1}.

Given Proposition \ref{prop:our duality}, we deduce the following by the usual Kazhdan-Lusztig method. Recall that $\LVM^h_{\gamma', \gamma}(t_1, t_2) \in \mb{Z}[t_1^{\pm 1}, t_2^{\pm 1}]$ is the change of basis matrix defined by the relation
\[ [j_!\gamma] = \sum_{\gamma'} \LVM_{\gamma', \gamma}^h[j_{!*}\gamma'] \quad \mbox{in $\oK(\mhm_0(K \bslash \ms{B}))$}\]
and
\[ \LVM^m(u) := \LVM^h(u^{\frac{1}{2}}, u^{\frac{1}{2}}) \in \mb{Z}[u^{\pm \frac{1}{2}}].\]

\begin{prop} \label{prop:our poly}
Assume in the setting of Proposition \ref{prop:our duality} that we are given a multiplicative norm
\[ | \cdot | \colon A \to \mb{R}_{> 0} \]
satisfying $|u| > 1$ and $|D(a)| = |a|^{-1}$. Then there exists a unique system of bases $\{C_\gamma\}$ for $M_\lambda$, with
\[ B_\gamma = \sum_{\gamma'} \LVM_{\gamma', \gamma} C_{\gamma'}, \quad \LVM_{\gamma', \gamma} \in \mb{Z}[A], \]
such that
\begin{enumerate}
\item $D(C_\gamma) = u^{-\ell_I(\gamma)}C_{\gamma'}$,
\item $\LVM_{\gamma, \gamma} = 1$,
\item $\LVM_{\gamma', \gamma} = 0$ unless $\gamma' \leq \gamma$ in the closure partial order, and
\item \label{itm:our poly 4} if $\gamma' \neq \gamma$, then $\LVM_{\gamma', \gamma}$ is a linear combination of elements of $A$ of norms at most $|u|^{(\ell_I(\gamma) - \ell_I(\gamma') - 1)/2}$.
\end{enumerate}
Moreover, the elements $\LVM_{\gamma', \gamma}$ are computable polynomials in $u$, of degree at most $(\ell_I(\gamma) - \ell_I(\gamma') - 1)/2$ for $\gamma \neq \gamma'$, satisfying
\begin{equation} \label{eq:non-integral mixed}
 \LVM^m_{\gamma', \gamma} = u^{(\ell(\gamma) - \ell_I(\gamma)) - (\ell(\gamma') - \ell_I(\gamma'))/2}\LVM_{\gamma', \gamma}.
\end{equation}
\end{prop}
\begin{proof}
Given the defining properties of $D$ from Proposition \ref{prop:our duality}, uniqueness and polynomiality follow by the usual Kazhdan-Lusztig yoga \cite{KL}. For existence, observe that the elements
\[ C_\gamma = u^{\frac{1}{2}(\ell_I(\gamma) - \ell(\gamma) - \dim H)}[j_{!*}\gamma] \]
have the desired properties; the only difficult property \eqref{itm:our poly 4} is just the statement that $j_!$ decreases weights and that $j_{!*}$ is its top weight quotient. This also proves \eqref{eq:non-integral mixed}.
\end{proof} 

A posteriori, given \eqref{eq:non-integral mixed} and Proposition \ref{prop:mixed to LV}, the polynomials $\LVM_{\gamma', \gamma}$ are precisely the Lusztig-Vogan multiplicity polynomials.

We now turn to the proof of Proposition \ref{prop:our duality}. To clarify the argument, let us endow the $\mb{Z}[A]$-module $M_\lambda$ with the $A$-grading given by
\[ \deg(a) = a, \; a \in A \quad \text{and} \quad \deg B_\gamma = u^{\ell_I(\gamma)}.\]
The degree bound on $\mbf{R}_{\gamma', \gamma}$ in the statement is equivalent to the claim that, for all $\gamma$, $D(B_\gamma)$ is a linear combination of terms of degree $u^n$ for $n \leq 0$ with respect to this grading. We also make a note of the following lemma, which follows immediately from the explicit equations of Proposition \ref{prop:hecke action}.

\begin{lem} \label{lem:our duality}
Fix $\lambda \in \mf{h}^*_\mb{R}$ and let $\alpha$ be a $\lambda$-integral (resp., non-integral) simple root. Then the operator $\mb{T}_{s_\alpha}^\lambda$ (resp., $u^{-\frac{1}{2}}\mb{T}_{s_\alpha}^\lambda$) from $M_\lambda$ to $M_{s_\alpha \lambda}$ is a linear combination of terms of degrees $u^n$ for integers $n \leq 1$ (resp., homogeneous of degree $0$) preserving the $\mb{Z}[u]$ span of the $B_\gamma$.
\end{lem}

\begin{proof}[Proof of Proposition \ref{prop:our duality}]
Existence is clear, since we can take $D$ to be given by Hermitian duality on $K^m$ and extend by linearity and \eqref{itm:our duality 1}. The proof of uniqueness and the polynomiality of the $\mbf{R}_{\gamma', \gamma}$ is a slight variation on \cite[Lemma 6.8]{ic3}.

Fix a local system $\gamma$ on a $K$-orbit $Q$, and assume by induction that the $\mbf{R}_{\delta', \delta}$ are known polynomials satisfying the conditions of the proposition for $\ell(\delta) < \ell(\gamma)$. Assume first that there exists a positive simple complex root $\alpha \in \Phi_+$ such that $\theta_Q\alpha \in \Phi_-$. Then we are in the setting of Proposition \ref{prop:orbit geometries} \eqref{itm:orbit geometries 2}, so we may write
\[ B_\gamma = \begin{cases} \mb{T}_{s_\alpha}B_\delta, & \mbox{if $\alpha$ is integral}, \\ u^{-\frac{1}{2}}\mb{T}_{s_\alpha}B_\delta, & \mbox{if $\alpha$ is non-integral}\end{cases} \]
using Proposition \ref{prop:hecke action} \eqref{itm:hecke action 2} and \eqref{itm:hecke action 3}, where $\ell(\delta) = \ell(\gamma) - 1$.
We may therefore use relation \eqref{itm:our duality 2} to compute $D(B_\gamma)$ in terms of $D(B_\delta)$. The desired properties for $D(B_\gamma)$ clearly follow from those for $D(B_\delta)$ by Lemma \ref{lem:our duality}.

We may therefore assume that $\theta_Q\alpha \in \Phi_+$ for all complex simple roots $\alpha$. In this case, the set of $Q$-real roots determines a standard Levi subgroup, with associated partial flag variety $\ms{P}_Q$, such that the image $\pi_Q(Q) \subset \ms{P}_Q$ under the canonical fibration $\pi_Q \colon \ms{B} \to \ms{P}_{Q}$ is closed. Fix $\gamma' = (Q', \gamma') < \gamma$, and assume by descending induction on $\ell(\gamma')$ that the $\mbf{R}_{\delta', \delta}$ are known and satisfy the required conditions for $\ell(\delta) = \ell(\gamma)$ and $\ell(\delta') > \ell(\gamma')$. Then, since $Q' \subset \bar{Q} - Q$, there must therefore exist a simple root $\alpha$ such that $\alpha$ is real for $Q$ and either non-compact imaginary or complex for $Q'$ with $\theta_{Q'}\alpha \in \Phi_+$.

If $\alpha$ is non-integral, then $\alpha$ is necessarily complex for $Q'$, and we may use Proposition \ref{prop:hecke action} \eqref{itm:hecke action 3} and \eqref{itm:hecke action 7} to write
\[ B_{\gamma} = u^{-\frac{1}{2}}\mb{T}_{s_\alpha}B_\delta \quad \text{and} \quad B_{\gamma'} = u^{-\frac{1}{2}}\mb{T}_{s_\alpha}B_{\delta'} \]
where $\ell_I(\delta) = \ell_I(\gamma)$, $\ell(\delta) = \ell(\gamma)$, $\ell_I(\gamma') = \ell_I(\delta')$ and $\ell(\delta') = \ell(\gamma') + 1$. Moreover, $B_{\gamma'}$ does not appear in $\mb{T}_{s_\alpha} B_{\delta''}$ for any $\delta'' \neq \delta'$, so we have $\mbf{R}_{\gamma', \gamma} = \mbf{R}_{\delta', \delta}$, which satisfies the required conditions by induction.

If $\alpha$ is integral, then the proof of \cite[Lemma 6.8]{ic3} applies, so we are done.

\end{proof}

\subsection{Combinatorics for the Hodge filtration} \label{subsec:hodge combinatorics}

In this subsection, we complete the proof of Theorem \ref{thm:main theorem 1}. The key additional ingredient is the following.

\begin{prop} \label{prop:mixed to hodge}
The isomorphisms
\begin{align}
\mb{Z}[t_1^{\pm\frac{1}{4}}, t_2^{\pm\frac{1}{4}}] \otimes_{\mb{Z}[u^{\pm\frac{1}{2}}]} \oK^m(\mhm_\lambda(K \bslash \ms{B})) &\to \mb{Z}[t_1^{\pm\frac{1}{4}}, t_2^{\pm\frac{1}{4}}] \otimes_{\mb{Z}[t_1^{\pm 1}, t_2^{\pm 1}]} \oK(\mhm_\lambda(K \bslash \ms{B}))  \nonumber \\
[j_!\gamma] &\mapsto (t_1t_2^{-1})^{\frac{1}{2}\ell_H(\gamma)}[j_!\gamma] \label{eq:mixed to hodge}
\end{align}
for $\lambda \in \mf{h}^*_\mb{R}$ commute with the action of the Hecke functors $\mb{T}_{s_\alpha}^\lambda$ on either side, where $\ell_H(\gamma)$ is the Hodge shift of Definition \ref{defn:hodge shift}.
\end{prop}
\begin{proof}
Fix a simple root $\alpha$ and a twisted local system $\gamma$, and write
\[ \mb{T}_{s_\alpha}[j_!\gamma] = \sum_{\gamma'} a_{\gamma', \gamma}[j_!\gamma'].\]
We need to show that
\begin{equation} \label{eq:mixed to hodge 2}
(t_1t_2^{-1})^{\frac{1}{2}(\ell_H(\gamma) - \ell_H(\gamma'))}a_{\gamma', \gamma} \in \mb{Z}[u^{\pm \frac{1}{2}}]
\end{equation}
for all $\gamma'$. This can be checked case-by-case from the formulas in Proposition \ref{prop:hecke action}. Cases \eqref{itm:hecke action 1}, \eqref{itm:hecke action 2}, \eqref{itm:hecke action 3} and \eqref{itm:hecke action 6} are straightforward, so we omit them.

Case \eqref{itm:hecke action 7} is also straightforward, but informative. We may suppose that $\gamma = \mc{O}(\lambda, \Lambda, x)$ and $\gamma' = \mc{O}(s_\alpha \lambda, s_\alpha \Lambda \otimes \alpha, x)$. Then
\begin{align*}
 \ell_H(\gamma') &= \frac{1}{2}\sum_{\substack{\beta \in \Phi_+ \text{ real}\\ s_\alpha\lambda\text{-non-integral}}} (-1)^{\lfloor \langle s_\alpha\lambda + \rho_\mb{R}, \check\beta\rangle\rfloor}(s_\alpha\Lambda \otimes \alpha)(m_\beta) \\
&= \frac{1}{2}\sum_{\substack{\beta \in \Phi_+ \text{ real}\\ s_\alpha\lambda\text{-non-integral}}} (-1)^{\lfloor \langle \lambda + \rho_\mb{R}, s_\alpha \check\beta \rangle \rfloor}\Lambda(m_{s_\alpha \beta}) \\
&= \frac{1}{2}\sum_{\substack{\beta \in \Phi_+ - \{\alpha\} \text{ real} \\ \lambda\text{-non-integral}}} (-1)^{\lfloor \langle \lambda + \rho_\mb{R}, \check\beta\rangle \rfloor}\Lambda(m_\beta) - \frac{1}{2}(-1)^{\lfloor \langle \lambda + \rho_\mb{R}, \check\alpha\rangle\rfloor} \Lambda(m_\alpha) \\
&= \ell_H(\gamma) + (-1)^{\lfloor \langle \lambda, \check\alpha\rangle\rfloor}\Lambda(m_\alpha),
\end{align*}
where in passing from the second line to the third we have used that $s_\alpha$ permutes the real roots in $\Phi_+ - \{\alpha\}$ and for non-integers $x$, $\lfloor -x \rfloor = -\lfloor x \rfloor - 1$. So
\[ \ell_H(\gamma) - \ell_H(\gamma') = \begin{cases} -1, & \mbox{if $(-1)^{\lfloor \langle \lambda, \check\alpha \rangle \rfloor}\Lambda(m_\alpha) = 1$,} \\ 1, & \mbox{if $(-1)^{\lfloor \langle \lambda, \check\alpha \rangle \rfloor} \Lambda(m_\alpha) = -1$} \end{cases} \]
and
\[ a_{\gamma', \gamma} = \begin{cases} t_1, & \mbox{if $(-1)^{\lfloor \langle \lambda, \check\alpha \rangle \rfloor}\Lambda(m_\alpha) = 1$,} \\ t_2, &\mbox{if $(-1)^{\lfloor \langle \lambda, \check\alpha \rangle\rfloor} \Lambda(m_\alpha) = -1$,} \end{cases} \]
by Proposition \ref{prop:hecke action}, so
\[ (t_1t_2^{-1})^{\frac{1}{2}(\ell_H(\gamma) - \ell_H(\gamma'))}a_{\gamma', \gamma} = t_1^{\frac{1}{2}}t_2^{\frac{1}{2}} = u^{\frac{1}{2}} \]
in either case.

Finally, for cases \eqref{itm:hecke action 4} and \eqref{itm:hecke action 5}, the required statement is Lemma \ref{lem:mixed to hodge} below.
\end{proof}

\begin{lem} \label{lem:mixed to hodge}
In the setting of Proposition \ref{prop:orbit geometries} \eqref{itm:orbit geometries 3}, suppose that $\alpha$ is $\lambda$-integral and that $\gamma = \mc{O}(\lambda, \Lambda, x)$ and $\gamma' = \mc{O}(s_\alpha \lambda, \Lambda', y)$, for $y = y_+$ or $y_-$, are twisted local systems satisfying
\[ \Lambda|_{H^{\theta_x} \cap H^{\theta_{y}}} = \Lambda'|_{H^{\theta_x} \cap H^{\theta_{y}}}.\]
Then $\ell_H(\gamma) = \ell_H(\gamma')$.
\end{lem}
\begin{proof}
Write $\theta_x, \theta_y \colon H \to H$ for the involutions at $x$ and $y$ respectively. Then $\theta_y = \theta_x s_\alpha$. By definition
\begin{equation} \label{eq:simple reflection hodge 1}
\ell_H(\mc{O}(\lambda, \Lambda, x)) = \frac{1}{2}\sum_{\substack{\beta \in \Phi_+ \mbox{\tiny $\theta_x$-real}\\ \mbox{\tiny non-integral}}} (-1)^{\lfloor \langle \lambda + \rho_{\mb{R}}(x), \check\beta\rangle\rfloor}\Lambda(m_\beta)
\end{equation}
and
\begin{equation} \label{eq:simple reflection hodge 2}
\ell_H(\mc{O}(s_\alpha\lambda, \Lambda', y)) = \frac{1}{2}\sum_{\substack{\beta \in \Phi_+ \mbox{\tiny $\theta_y$-real}\\ \mbox{\tiny non-integral}}} (-1)^{\lfloor \langle s_\alpha\lambda + \rho_{\mb{R}}(y), \check\beta\rangle\rfloor}\Lambda'(m_\beta)
\end{equation}
where $\rho_\mb{R}(x)$ and $\rho_\mb{R}(y)$ are half the sum of the positive $\theta_x$-real and $\theta_y$-real roots respectively.

First, we claim that $\beta \in \Phi_+$ is $\theta_y$-real if and only if $\beta$ is $\theta_x$-real and $\langle \alpha, \check\beta \rangle = 0$. To see this, note that if $\beta$ is $\theta_y$-real then
\[ \beta + \theta_x \beta = \beta + s_\alpha \theta_y \beta = \beta - s_\alpha \beta = \langle \beta, \check\alpha\rangle \alpha.\]
Since the left hand side is $\theta_x$-invariant and the right hand side is $\theta_x$-anti-invariant, both sides must be zero, hence $\beta$ is also $\theta_x$-real and $\langle \beta, \check\alpha\rangle = \langle \alpha, \check\beta \rangle = 0$. This proves one direction of the claim, and the converse is clear.

Next, write the sum \eqref{eq:simple reflection hodge 1} in the form
\begin{equation} \label{eq:simple reflection hodge 3}
\ell_H(\mc{O}(\lambda, \Lambda, x)) = \frac{1}{2}\sum_{C} \sum_{\check\beta \in C} (-1)^{\lfloor \langle \lambda + \rho_{\mb{R}}(x), \check\beta\rangle\rfloor}\Lambda(m_\beta),
\end{equation}
where the outer sum is over equivalence classes $C$ of $\theta_x$-real non-integral positive coroots modulo multiples of $\check\alpha$. By the above classification of $\theta_y$-real roots, each equivalence class $C$ contains either a unique $\theta_y$-real coroot (if $|C| = 1$ or $3$) or no $\theta_y$-real coroots (if $|C| = 2$ or $4$). We will show that the contribution of $C$ to \eqref{eq:simple reflection hodge 3} is equal to the contibution of its unique $\theta_y$-real coroot to \eqref{eq:simple reflection hodge 2} if it exists, and zero otherwise.

First suppose that $C$ contains no $\theta_y$-real coroots. In this case, $\langle \alpha, \check\beta \rangle$ is odd for all $\check\beta \in C$, and $s_\alpha$ partitions $C$ into orbits of size two. For all $\check\beta \in C$, we therefore have
\begin{align*}
(-1)^{\lfloor \langle \lambda + \rho_{\mb{R}}(x), s_\alpha\check\beta\rangle \rfloor}\Lambda(m_{s_\alpha \beta}) &= \left((-1)^{\langle \lambda + \rho_\mb{R}(x), \check\alpha\rangle} \Lambda(m_\alpha)\right)^{\langle \alpha, \check\beta\rangle}(-1)^{\lfloor \langle \lambda + \rho_{\mb{R}}(x), \check\beta\rangle \rfloor}\Lambda(m_\beta) \\
&= (-1)^{\langle \alpha, \check\beta\rangle}(-1)^{\lfloor \langle \lambda + \rho_{\mb{R}}(x), \check\beta\rangle \rfloor}\Lambda(m_\beta) \\
&= - (-1)^{\lfloor \langle \lambda + \rho_{\mb{R}}(x), \check\beta\rangle \rfloor}\Lambda(m_\beta),
\end{align*}
so
\[ \sum_{\check\beta \in C} (-1)^{\lfloor \langle \lambda + \rho_{\mb{R}}(x), \check\beta\rangle\rfloor}\Lambda(m_\beta) = 0\]
as claimed.

Suppose on the other hand that $C$ contains a $\theta_y$-real coroot $\check\beta$, and consider the difference
\[ \rho_\mb{R}(x) - \rho_\mb{R}(y) = \frac{1}{2}\sum_{\substack{\mbox{\tiny $\gamma \in \Phi_+$ $\theta_x$-real} \\ \langle \alpha, \check\gamma \rangle \neq 0}} \gamma.\]
We have
\begin{align*}
\langle \rho_\mb{R}(x) - \rho_\mb{R}(y), \check\beta\rangle &= \frac{1}{2} \langle (\rho_\mb{R}(x) - \rho_\mb{R}(y)) - s_\beta(\rho_\mb{R}(x) - \rho_\mb{R}(y)), \check\beta \rangle \\ 
&= \frac{1}{2}\sum_{\gamma \in S} \langle \gamma, \check \beta \rangle,
\end{align*}
where the sum is over the set $S$ of $\gamma \in \Phi_+$ such that $\gamma$ is $\theta_x$-real, $\langle \gamma, \check\alpha\rangle \neq 0$ and $s_\beta \gamma \in \Phi_-$. Note that the abelian group $\{1, s_\alpha, -s_\beta, -s_\alpha s_\beta\}$ acts on $S$ and preserves the function $\langle -, \check\beta\rangle$. The orbit of $\gamma$ is of size $2$ if $\gamma \in \mb{Q}\mbox{-span}\{\alpha, \beta\}$ and of size $4$ otherwise, so
\begin{equation} \label{eq:simple reflection hodge 4}
\langle \rho_\mb{R}(x) - \rho_\mb{R}(y), \check\beta \rangle \equiv \frac{1}{2}\sum_{\gamma \in S'} \langle \gamma, \check\beta\rangle \mod 2
\end{equation}
where
\[ S' = S \cap \{\gamma \in S \cap \mb{Q}\mbox{-span}\{\alpha, \beta\} \mid \langle \gamma, \check\beta \rangle \equiv 1 \mod 2\}.\]
The root system $\Phi_{\alpha, \beta} = \Phi \cap \mb{Q}\mbox{-span}\{\alpha, \beta\}$ is of rank $2$ and contains two orthogonal roots, so it is of type $A_1 \times A_1$, $B_2$ or $G_2$. The claim can now be checked directly in each of these cases using \eqref{eq:simple reflection hodge 4}.
\end{proof}

\begin{proof}[Proof of Theorem \ref{thm:main theorem 1}]
In the context of Propositions \ref{prop:our duality} and \ref{prop:our poly}, let $A$ be the abelian group $t_1^{\frac{1}{4}\mb{Z}}t_2^{\frac{1}{4}\mb{Z}}$ with $u^{\frac{1}{2}} = (t_1t_2)^{\frac{1}{2}}$, duality $D(t_1) = t_2^{-1}$, $D(t_2) = t_1^{-1}$ and norm given by $|t_1| = |t_2| = 2$. Then we may identify
\[ M_\lambda \cong \mb{Z}[t_1^{\pm \frac{1}{4}}, t_2^{\pm \frac{1}{4}}] \otimes_{\mb{Z}[t_1^{\pm 1}, t_2^{\pm 1}]} \oK(\mhm_\lambda(K \bslash \ms{B})) \]
via the isomorphism \eqref{eq:mixed to hodge}, so that
\[ B_\gamma = u^{\frac{1}{2}(\ell_I(\gamma) - \ell(\gamma) - \dim H)}(t_1t_2^{-1})^{\frac{1}{2}\ell_H(\gamma)}[j_!\gamma].\]
By Proposition \ref{prop:mixed to hodge}, this respects the action of the Hecke operators on either side, so the Hermitian duality $(-)^h$ on $\oK(\mhm_\lambda(K \bslash \ms{B}))$ agrees with the map $D$ of Proposition \ref{prop:our duality}. (Note that the pre-factors $t_1t_2^{-1}$ are self-dual under $D$, so $(-)^h$ still satisfies the conditions of Proposition \ref{prop:our duality} under this identification.) The basis 
\[ C_\gamma = u^{\frac{1}{2}(\ell_I(\gamma) - \ell(\gamma) - \dim H)}(t_1t_2^{-1})^{\ell_H(\gamma)/2}[j_{!*}\gamma]\]
therefore satisfies the conditions of Proposition \ref{prop:our poly}, so we have
\[ (t_1t_2^{-1})^{\ell_H(\gamma)/2}[j_!\gamma] = \sum_{\gamma'} \LVM^m_{\gamma', \gamma}(u)(t_1t_2^{-1})^{\ell_H(\gamma')/2}[j_{!*}\gamma']\]
from which we deduce the formula
\[ \LVM^h_{\gamma', \gamma} = (t_1t_2^{-1})^{\frac{1}{2}(\ell_H(\gamma') - \ell_H(\gamma))}\LVM^m_{\gamma', \gamma}.\]
\end{proof}

\section{Proof of Theorem \ref{thm:main theorem 3}}
\label{mk}
\label{proof:main theorem 3}

In this section, we recall some necessary ingredients and give the proofs of Theorem \ref{thm:main theorem 3} and Proposition \ref{prop:hodge positivity}. In \S\ref{subsec:nearby cycles}, we recall two equivalent definitions of the functor of nearby cycles on holonomic $\ms{D}$-modules and the isomorphism between them. In \S\ref{subsec:intertwining}, we recall the intertwining functors and a fundamental exact sequence. In \S\ref{subsec:proof of main theorem 3}, we give the proof of Theorem \ref{thm:main theorem 3} by combining these ingredients with the properties of minimal $K$-types discussed in \S\ref{subsec:mk statement}. Finally, in \S\ref{subsec:proof of hodge positivity} we give the proof of Proposition \ref{prop:hodge positivity} on positivity of polarizations on the lowest piece of the Hodge filtration.

\subsection{Nearby cycles} \label{subsec:nearby cycles}

In this subsection, we recall the construction of unipotent nearby cycles for holonomic $\ms{D}$-modules.

We will work first in the following setting. Let $X$ be a smooth variety and $f \colon X \to \mb{C}$ a regular function. We set $U = f^{-1}(\mb{C}^\times)$ and let $j \colon U \to X$ be the inclusion.

We first give Beilinson's construction of nearby cycles \cite{beilinson}. Given a holonomic $\ms{D}$-module $\ms{M}$ on $U$, we have the $\ms{D}_U$-modules
\[ f^s\ms{M}[[s]] = \varprojlim_n \frac{f^s\ms{M}[s]}{s^nf^s\ms{M}[s]} \quad \text{and} \quad f^s\ms{M}((s)) = \varinjlim_m \varprojlim_n \frac{s^{-m} f^s\ms{M}[s]}{s^nf^s\ms{M}[s]},\]
where $f^s\ms{M}[s] = \ms{M} \otimes \mb{C}[s]$ as $\mc{O}_U$-modules, with the usual twisted action of differential operators. The terms in the above limits are holonomic $\ms{D}_U$-modules (regular if $\mc{M}$ is so); we extend standard operations to such objects in the natural way by setting
\[ j_*f^s\ms{M}[[s]] = \varprojlim_n j_*\left(\frac{f^s\ms{M}[s]}{s^nf^s\ms{M}[s]}\right), \quad j_*f^s\ms{M}((s)) = \varinjlim_m \varprojlim_n j_*\left(\frac{s^{-m}f^s\ms{M}[s]}{s^nf^s\ms{M}[s]}\right).\]
Note that the $b$-function lemma implies that $j_!f^s\ms{M}((s)) \to j_*f^s\ms{M}((s))$ is an isomorphism.

\begin{defn} \label{defn:beilinson functor}
Let $a \in \mb{Z}_{\geq 0}$. The \emph{Beilinson functor $\pi_f^a$} is defined by
\[ \pi_f^a(\ms{M}) = \coker\left(j_!f^s\ms{M}[[s]] \xrightarrow{s^a} j_*f^s\ms{M}[[s]]\right)\]
for $\ms{M}$ a holonomic $\ms{D}_U$-module. The functor $\pi_f^0$ is \emph{Beilinson's unipotent nearby cycles functor}.
\end{defn}

It follows from the $b$-function lemma that the sequence
\[ \coker\left(j_!\frac{f^s\ms{M}[s]}{s^nf^s\ms{M}[s]} \xrightarrow{s^a} j_*\frac{f^s\ms{M}[s]}{s^nf^s\ms{M}[s]} \right) \]
stabilizes to $\pi_f^a(\ms{M})$ for $n \gg 0$. Hence $\pi_f^a(\ms{M})$ is always a holonomic $\ms{D}_X$-module, regular if $\mc{M}$ is so, and the endomorphism $s$ is nilpotent.

We next recall the standard definition of nearby cycles in terms of the $V$-filtration and relate it to the Beilinson construction. This is essentially due to Kashiwara~\cite{kashiwara1}, and is described explicitly (albeit in different notation to ours) in \cite[Th\'eor\`eme 4.7-2]{maisonbe-mebkhout}. 

We will start with the case of a smooth divisor. Let us consider $X'=X\times \bC$. We have natural inclusions $i \colon X=X \times \{0\} \to X'$ and $j'\colon U' = X\times \bC^\times \to X'$. Let $t$ be the coordinate on the $\bC$-factor. Then for any holonomic $\ms{D}$-module $\ms{N}$ on $X'$, we have the decreasing Kashiwara-Malgrange $V$-filtration on $\ms{N}$ with respect to the divisor $t = 0$, indexed so that
\[ t \partial_t - \alpha \colon \mrm{Gr}_V^\alpha\ms{N} \to \mrm{Gr}_V^\alpha \ms{N} \]
is nilpotent. For general holonomic $\mc{D}$-modules, the $V$-filtration is indexed by $\mb{C}$ with, say, the lexicographic ordering. All our examples, however, will underlie mixed Hodge modules, in which case the $V$-filtration can be taken to be indexed by $\alpha \in \mb{R}$. For all $\alpha$, the sheaves $\mrm{Gr}_V^\alpha\ms{N}$ are holonomic $\ms{D}$-modules on $X = X \times \{0\}$. The $V$-filtration construction of unipotent nearby cycles with respect to $t$ is the functor $\mrm{Gr}_V^0$.

Recall that the pushforward of a holonomic $\ms{D}$-module along a closed immersion includes a twist by the determinant of the normal bundle. In the lemma below, $\partial_t$ denotes the vector field generating the normal bundle of $X$ in $X'$ in this twist.

\begin{lem} \label{lem:nearby smooth divisor}
Let $\ms{N}$ be a holonomic $\ms{D}_{U'}$-module. Then
\[ i_*\mrm{Gr}_V^0j'_*\ms{N} \cong \pi_t^0(\ms{\ms{N}})\qquad n \otimes \partial_t \mapsto t^{s-1}n\]
and the operator $-t \partial_t$ on the left is sent to multiplication by $s$ on the right.
\end{lem}
\begin{proof}
By definition we have an exact sequence 
\[ 0 \to  j'_!t^s\ms{N}[[s]] \rightarrow j'_*t^s\ms{N}[[s]]\to \pi_t^0(\ms{N})\to 0\]
and the cokernel is supported on $X$. Thus, $\pi_t^0(\ms{\ms{N}})= i_*\Gr_V^{-1} (\pi_t^0(\ms{N})) \otimes \partial_t^{-1}$ and we have an exact sequence
\[ 0 \to \Gr_V^{-1}  (j'_!t^s\ms{N}[[s]]) \rightarrow \Gr_V^{-1} (j'_*t^s\ms{N}[[s]])\to \Gr_V^{-1} (\pi_t^0(\ms{N}))\to 0.\]
Here $V^{\bigcdot} j'_!t^s\ms{N}[[s]] := \varprojlim_n V^{\bigcdot} j'_!t^s\ms{N}[s]/(s^n)$ etc.
Now,
\[\Gr_V^{-1}  (j'_!t^s\ms{N}[[s]]) \cong \partial_t \Gr_V^{0}(j'_!t^s\ms{N}[[s]]) =  \partial_t \Gr_V^{0}(j'_*t^s\ms{N}[[s]])\]
and  $t \Gr_V^{-1} (j'_*t^s\ms{N}[[s]) = \Gr_V^{0} (j'_*t^s\ms{N}[[s])$. Hence 
\[
 \pi_t^0(\ms{\ms{N}}) = i_*\coker(t\partial_t \colon \mrm{Gr}_V^0(j'_*t^s\ms{N}[[s]]) \to \mrm{Gr}_V^0(j'_*t^s\ms{N}[[s]])) \cong i_*\Gr_V^0 j'_*\cN,
\]
where the last isomorphism is obtained by observing
\[ (\mrm{Gr}_V^0j'_*t^s\ms{N}[[s]], t \partial_t)  = ((\mrm{Gr}_V^0j'_*\ms{N})[[s]], t\partial_t + s).\]
By construction the map $i_*\Gr_V^0 j'_*\cN \to \pi_t^0(\cN)$ is multiplication by $t^{-1} \otimes \partial_t^{-1}$ and sends $-t\partial_t$ to $s$.
\end{proof}

To treat the case of a function $f \colon X\to \bC$ we let $i_f \colon X \to X' = X \times \mb{C}$ be the inclusion of the graph of $f$ and $i_f|_U \colon U \to U'$ its restriction. The nearby cycles are defined by applying the $V$-filtration construction to the pushforward.

\begin{defn}
For $\ms{M}$ a holonomic $\ms{D}_U$-module, the \emph{unipotent nearby cycles} are
\[ \psi_f^{un}\ms{M} := \mrm{Gr}_V^0j'_*(i_{f}|_U)_*\ms{M}.\]
\end{defn}

\begin{rmk}
It may be helpful to note that the pushforward $j'_*(i_{f}|_U)_*\ms{M}$ above is given very explicitly by
\begin{align*}
j'_*(i_f|_U)_*\ms{M} &\cong \frac{\mc{O}_{X'}[(t - f)^{-1}]}{\mc{O}_{X'}} \otimes_{\mc{O}_X} j_*\ms{M} \\
m \otimes \partial_t &\mapsto \frac{m}{t - f}.
\end{align*}
\end{rmk}

We deduce the following from Lemma \ref{lem:nearby smooth divisor}

\begin{prop} \label{prop:beilinson nearby cycles}
In the setting above, we have a natural isomorphism
\[ \psi_f^{un}\ms{M} \cong \pi_f^0(\ms{M}),\]
of $\ms{D}$-modules on $X$ such that
\begin{enumerate}
\item the class of $m \otimes \partial_t \in V^0j'_*(i_f|_U)_*\ms{M}$ on the left is sent to the class of $f^{s - 1}m$ on the right, and
\item the operator $\mrm{N} = -t\partial_t$ on the left is sent to multiplication by $s$ on the right.
\end{enumerate}
\end{prop}
\begin{proof}
We have an obvious isomorphism $i_*\pi_f^0(\ms{M}) = i_{f*}\pi_f^0(\ms{M}) = \pi_t^0((i_{f}|_U)_*\ms{M})$. But $\pi_t^0((i_{f}|_U)_*\ms{M})) = i_*\psi_f^{un}\ms{M}$ by Lemma \ref{lem:nearby smooth divisor}, so we obtain the desired isomorphism by Kashiwara's equivalence.
\end{proof}

Suppose now that $j \colon Q \to X$ is a locally closed immersion of smooth varieties and $f \colon \bar{Q} \to \mb{C}$ is a regular function such that $Q = f^{-1}(\mb{C}^\times)$. (Note that we do not assume that the closure $\bar{Q}$ is smooth.) Then, for $\ms{M}$ a holonomic $\ms{D}$-module on $Q$, we may define $\pi_f^0\ms{M}$ as in Definition \ref{defn:beilinson functor} and
\[ \psi_f^{un}\ms{M} = \mrm{Gr}_V^0j_{f*}\ms{M},\]
where $j_f = (j, f) \colon Q \to X \times \mb{C}$. Locally on $\bar{Q}$, we may write $f$ as the restriction of a regular function $g$ on $X$; in this case, we have
\[ \psi_f^{un}\ms{M} = \psi_g^{un}i_*\ms{M} \qquad \text{and} \qquad \pi_f^0\ms{M} = \pi_g^0i_*\ms{M},\]
where $i \colon Q \to g^{-1}(0)$ is the (closed) inclusion. We deduce the following from Proposition \ref{prop:beilinson nearby cycles}.

\begin{cor} \label{cor:locally closed nearby cycles}
Given $j \colon Q \to X$ and $\ms{M}$ as above, we have a canonical isomorphism
\[ \psi_f^{un}\ms{M} \cong \pi_f^0(\ms{M}) \]
of $\ms{D}$-modules on $X$, sending the class of
\[ m \otimes \xi \partial_t \in V^0j_{f*}\ms{M} \cap j_{f\bigcdot}(\ms{M} \otimes \omega_{Q/X \times \mb{C}}) \]
on the left to the class of $f^{s - 1}m \otimes \xi$ on the right.
\end{cor}

Finally, suppose that $Q$ and $X$ are equipped with commuting actions of algebraic groups $K$ and $H$ such that $f$ is $K$-equivariant and scaled under the action of $H$. Then, for a holonomic $\ms{D}$-module $\ms{M}$ on $Q$ strongly $K$-equivariant and weakly $H$-equivariant, the $\ms{D}_X$-modules $\pi_f^0(\ms{M})$ and $\psi_f^{un}\ms{M}$ have natural actions of $K$ and $H$. We conclude this subsection with the following observation.

\begin{cor} \label{cor:equivariant nearby cycles}
In the equivariant setting above, the isomorphism of Corollary \ref{cor:locally closed nearby cycles} is $K \times H$-equivariant.
\end{cor}

\subsection{Intertwining functors} \label{subsec:intertwining}

In this subsection, we briefly recall the intertwining functors associated with simple reflections and some of their properties.

Let $\lambda \in \mf{h}^*_\mb{R}$, $\alpha \in \Phi_+$ a simple root, and recall from \S\ref{subsec:convolution} the $G$-orbit $X_{s_\alpha} \subset \ms{B} \times \ms{B}$. Let $\mc{O}_{X_{s_\alpha}}(-\lambda, s_\alpha \cdot \lambda)$ be the unique rank $1$ $G$-equivariant $(-\lambda, s_\alpha \cdot \lambda)$-twisted local system on $X_{s_\alpha}$, where the dot action is defined by
\[ s_\alpha \cdot \lambda = s_\alpha(\lambda + \rho) - \rho = s_\alpha \lambda -\alpha.\]
The (dual) intertwining functor is
\[ \mb{I}_{s_\alpha}^* = - \star j_{s_\alpha *}\mc{O}_{X_{s_\alpha}}^{-\lambda, s_\alpha \cdot \lambda} \colon \oD^b_K(\mhm_\lambda(\ms{B})) \to \oD^b_K(\mhm_{s_\alpha \cdot \lambda}(\ms{B})).\]

The key property of the intertwining functors, due to Beilinson and Bernstein, is that $R\Gamma \circ \mb{I}^*_{s_\alpha}(\ms{M}) \cong \Gamma(\ms{M})$ as $(\mf{g}, K)$-modules for $\lambda + \rho$ regular and integrally dominant. This statement can be extended with some care to singular $\lambda$. We will have need of the following very special case, where the isomorphism is realized by a map of mixed Hodge modules.

\begin{lem} \label{lem:intertwining}
Let $Q \subset \ms{B}$ be a $K$-orbit, and let $\alpha \in \Phi_+$ be a $Q$-complex simple root such that $\theta_Q\alpha \in \Phi_+$. Assume $\gamma = \mc{O}_Q(\lambda, \Lambda)$ is such that $\langle \lambda, \check\alpha \rangle = -1$. Then there exists $\gamma' = \mc{O}_{Q'}(\lambda, \Lambda')$ such that
\begin{enumerate}
\item \label{itm:intertwining 1} $\mb{I}_{s_\alpha}^*j_*\gamma = j_*\gamma'$,
\item \label{itm:intertwining 2} $\dim Q' = \dim Q + 1$,
\item \label{itm:intertwining 3} there is a surjection $j_*\gamma' \to j_*\gamma(-1)$ of twisted mixed Hodge modules whose kernel is sent to zero under $R\Gamma$.
\end{enumerate}
\end{lem}
\begin{proof}
Consider the $\mb{P}^1$-fibration $\pi_\alpha \colon \ms{B} \to \ms{P}_\alpha$, where $\ms{P}_\alpha$ is the partial flag variety of parabolics of type $\alpha$. By our assumptions on $\alpha$, Proposition \ref{prop:orbit geometries} shows that $Q \to \pi_\alpha(Q)$ is an isomorphism and $\pi_\alpha^{-1}\pi_\alpha Q = Q \cup Q'$ where $Q'$ is another $K$-orbit with $\dim Q' = \dim Q + 1$. An easy calculation, similar to Proposition \ref{prop:hecke action} \eqref{itm:hecke action 2}, shows \eqref{itm:intertwining 1} with $\gamma' = \mc{O}_{Q'}(\lambda, s_\alpha \Lambda \otimes (-\alpha))$, so this proves \eqref{itm:intertwining 1} and \eqref{itm:intertwining 2}. 

To prove \eqref{itm:intertwining 3}, consider the $G$-orbit closure
\[ \bar{X}_{s_\alpha} = \ms{B} \times_{\ms{P}_\alpha}\ms{B} \subset \ms{B} \times \ms{B}.\]
The twisted local system $\mc{O}_{X_{s_\alpha}}(-\lambda, s_\alpha \cdot \lambda) = \mc{O}_{X_{s_\alpha}}(-\lambda, \lambda)$ extends to a $(-\lambda, \lambda)$-twisted local system $\mc{O}_{\bar{X}_{s_\alpha}}(-\lambda, \lambda)$ on $\bar{X}_{s_\alpha}$, with an exact sequence
\[ 0 \to \bar{j}_{s_\alpha *}\mc{O}_{\bar{X}_{s_\alpha}}(-\lambda, \lambda) \to j_{s_\alpha *} \mc{O}_{X_{s_\alpha}}(-\lambda, \lambda) \to j_{1*} \mc{O}_{X_1}(-\lambda, \lambda)(-1) \to 0,\]
where $\bar{j}_{s_\alpha} \colon \bar{X}_{s_\alpha} \to \ms{B} \times \ms{B}$ is the inclusion. So we have an exact sequence
\[ 0 \to j_*\gamma \star \left(\bar{j}_{s_\alpha *} \mc{O}_{\bar{X}_{s_\alpha}}(-\lambda, \lambda)\right) \to j_*\gamma' \to j_*\gamma(-1) \to 0\]
of the corresponding convolutions. The kernel is the pushforward from $\pi_\alpha^{-1}\pi_\alpha Q$ of a line bundle with degree $-1$ on the fibres of $\pi_\alpha$, so it is sent to zero under $R\Gamma$ as claimed.
\end{proof}

\subsection{Proof of Theorem \ref{thm:main theorem 3}}
\label{subsec:proof of main theorem 3}

We now give the proof of Theorem \ref{thm:main theorem 3}. We will show that the minimal $K$-types of $\Gamma(\ms{B}, j_*\gamma)$ lie in $\Gamma(\ms{B}, F_c j_*\gamma)$ whenever $\lambda + \rho$ is dominant and $\langle \lambda + \rho, \check\alpha \rangle > 0$ for all $Q$-compact imaginary simple roots $\alpha$. This includes all relevant $\gamma$ as in the statement of the theorem, so the result follows.

Let $\mu \subset \Gamma(\ms{B}, j_*\gamma)$ be a minimal $K$-type. According to Proposition \ref{prop:minimal K-type support}, we have
\begin{equation} \label{eq:minimal hodge 1}
\mu \subset \Gamma(\ms{B}, j_{\bigcdot}(\gamma \otimes \omega_{Q/\ms{B}})).
\end{equation}
If $Q$ is closed, then $F_c j_*\gamma = j_{\bigcdot}(\gamma \otimes \omega_{Q/\ms{B}})$ by definition, so we are done. So assume from now on that $Q$ is not closed. Then we can (and will) choose $\varphi \in \mf{h}^*$ and $f_\varphi \in H^0(\bar{Q}, \ms{L}_\varphi)^K$ such that $Q = f_{\varphi}^{-1}(\mb{C}^\times)$ as in \S\ref{subsec:mhm jantzen}. By \eqref{eq:minimal hodge 1}, we have a well-defined $K$-submodule
\[ f_\varphi^a \mu \subset \Gamma(\ms{B}, j_{\bigcdot}(f_\varphi^a \gamma \otimes \omega_{Q/\ms{B}})) \subset \Gamma(\ms{B}, j_*f_\varphi^a\gamma),\]
for all $a \in \mb{R}$, which is also a minimal $K$-type.

Consider the quotient map
\begin{equation} \label{eq:minimal hodge 2}
 j_*f_{\varphi}^{s + a}\gamma[[s]] \to  \pi_{f_\varphi}^0(f_{\varphi}^a\gamma) = \coker (j_!f_{\varphi}^{s + a} \gamma[[s]] \to j_*f_{\varphi}^{s + a} \gamma[[s]])
\end{equation}
and the set
\[
J_\gamma = \{\mbox{$a > 0 \mid f_{\varphi}^{s + a}\mu$ has non-zero image under \eqref{eq:minimal hodge 2}}\}.
\]

\begin{lem} \label{lem:minimal K-types 1}
The $K$-type $\mu$ lies in $\Gamma(\ms{B}, F_cj_*\gamma)$ if and only if $J_\gamma = \emptyset$.
\end{lem}
\begin{proof}
By \eqref{open embedding formula}, $\mu \subset F_c j_*\gamma$ if and only if $\mu \subset V^{\geq -1}j'_*i_*\gamma$, which holds if and only if $t \mu \subset V^{\geq 0}j'_*i_*\gamma$, where $i \colon \tilde{Q} \to \tilde{\ms{B}} \times \mb{C}^\times$ is the graph of $f_\varphi$, $j' \colon \tilde{\ms{B}} \times \mb{C}^\times \to \tilde{\ms{B}} \times \mb{C}$ is the open inclusion, and $t$ is the coordinate on $\mb{C}$. This fails if and only if $t \mu $ has non-zero image in $\mrm{Gr}_V^{-a}j'_*i_*\gamma$ for some $a > 0$. But by Corollary \ref{cor:equivariant nearby cycles}, we have an equivariant isomorphism
\[ 
\mrm{Gr}_V^{- a}j'_*i_*\gamma = \mrm{Gr}_V^0j'_*i_*(f_{\varphi}^a \gamma) = \psi_{f_\varphi}^{un}(f_\varphi^a \gamma) \cong \pi_{f_\varphi}^0(f_{\varphi}^a \gamma)
 \]
sending $t \mu$ to $f_{\varphi}^{s + a}\mu$. This proves the lemma.
\end{proof}

Let us now consider the interval
\[
I= \{\mbox{$a'\geq 0 \mid \lambda + a'\varphi + \rho$ is dominant}\}.
\]

\begin{lem} \label{lem:minimal K-types 2}
We have $I \cap J_{\gamma} = \emptyset$. In particular, for $a \in I$, $\mu \subset \Gamma(\ms{B}, F_cj_*\gamma)$ if and only if $f_{\varphi}^a\mu \subset \Gamma(\ms{B}, F_cj_*f_{\varphi}^a\gamma)$.
\end{lem}
\begin{proof}
The ``in particular'' follows immediately from Lemma \ref{lem:minimal K-types 1} and the observation that $J_{f_{\varphi}^a \gamma} = (J_{\gamma} - a) \cap \mb{R}_{> 0}$. To prove $I \cap J_\gamma = \emptyset$, assume to the contrary that there is an $a\in I$ such that $a > 0$ and $f_{\varphi}^{s + a} \mu$ is non-zero in $\Gamma(\ms{B}, \pi_{f_\varphi}^0f_{\varphi}^a\gamma)$. Since the endomorphism given by multiplication by $s$ is nilpotent and equivariant this means that the $K$-type $\mu$ appears in
\[ \coker(s \colon \Gamma(\ms{B}, \pi_{f_{\varphi}}^0f_{\varphi}^a\gamma) \to \Gamma(\ms{B}, \pi_{f_\varphi}^0f_{\varphi}^a\gamma)) = \Gamma(\ms{B}, \coker(s \colon \pi_{f_\varphi}^0f_{\varphi}^a\gamma \to \pi_{f_\varphi}^0f_{\varphi}^a\gamma)).
 \]
 Here we have made use of the fact that $\lambda + a\varphi + \rho$ is dominant, which guarantees exactness of $\Gamma$. The additional assumption $a > 0$ ensures that $f_\varphi^a\gamma$ is relevant, so by Theorem \ref{thm:minimal K-type in irreducible}, the minimal $K$-type $\mu$ can only appear in $\Gamma(\cB,j_{!*}f_{\varphi}^a\gamma)$. However we have just shown that it also appears in 
 \[
 \coker(s \colon \pi_{f_\varphi}^0 f_{\varphi}^a\gamma \to \pi_{f_\varphi}^0f_{\varphi}^a\gamma) = \coker(j_!f_{\varphi}^a \gamma \to j_*f_{\varphi}^a\gamma)
 \]
which is a contradiction.
\end{proof}

Continuing with the proof of Theorem \ref{thm:main theorem 3}, if our chosen $\varphi$ is dominant, then $I = \mb{R}_{\geq 0}$, so by Lemmas \ref{lem:minimal K-types 1} and \ref{lem:minimal K-types 2} we are done. So it remains to consider the case where $\varphi$ is not dominant. In this case the interval $I$ is finite, so by Lemma \ref{lem:minimal K-types 2} again we may assume without loss of generality that $I = \{0\}$.

Since $I = \{0\}$, there must exist a simple coroot $\check\alpha$ with $\langle \varphi, \check\alpha\rangle < 0$ and $\langle \lambda, \check\alpha \rangle = -1$. Since $f_\varphi$ is a $K$-invariant boundary equation for $Q$, we have necessarily $\langle \varphi, \check\beta \rangle = 0$ if $\beta$ is imaginary and $\langle \varphi, \check\beta \rangle > 0$ if $\beta$ is a positive simple root with $\theta_Q\beta \in \Phi_-$. So the root $\alpha$ must be complex with $\theta_Q \alpha \in \Phi_+$. We are therefore in the setting of Lemma \ref{lem:intertwining}, so there is a $\lambda$-twisted local system $\gamma'$ on a larger orbit $Q'$ and a surjection
\[j_*\gamma' \to j_*\gamma(-1) \]
of twisted mixed Hodge modules, whose kernel is sent to zero by $R\Gamma$. Now
\[ F_{c - 1}j_*\gamma(-1) = F_cj_*\gamma \]
and we have a commutative diagram
\[
\begin{CD}
\Gamma(\ms{B}, j_*\gamma') @= \Gamma(\ms{B}, j_*\gamma) 
\\
@AAA  @AAA
\\ \Gamma(\ms{B}, F_{c - 1}j_*\gamma') @>>>  \Gamma(\ms{B}, F_cj_*\gamma).
\end{CD}
\]
By descending induction on $\dim Q$, we have $\mu \subset \Gamma(\ms{B}, F_{c - 1}j_*\gamma')$ and hence $\mu \subset \Gamma(\ms{B}, F_{c }j_*\gamma)$ as claimed.

This completes the proof of Theorem \ref{thm:main theorem 3}.

\subsection{Proof of Proposition \ref{prop:hodge positivity}} \label{subsec:proof of hodge positivity}

In this subsection, we give the proof of Proposition \ref{prop:hodge positivity}.

Let $\gamma$ be a twisted local system on a $K$-orbit $Q$ as in the statement of the proposition, and let $n \in \Gamma(\ms{B}, F_cj_{!*}\gamma)$ be in the lowest piece of the Hodge filtration. Observe that the integral over $U_\mb{R}$ defining $\Gamma(S)$ (see \S \ref{global_polar}) may be rewritten as the integral
\[ \Gamma(S)(n, \bar{n}) = \langle \eta, S(n, \bar{n})\rangle\]
over $\tilde{\ms{B}}$, where $\eta$ is a compactly supported real top form on $\tilde{\ms{B}}$ positive with respect to the orientation. To show that $\Gamma(S)(n, \bar{n})$ is positive, it suffices via a partition of unity to show that $\langle \eta, S(n, \bar{n})\rangle > 0$ for positive $\eta$ with sufficiently small compact support. We may therefore work locally on $\tilde{\ms{B}}$, and assume that $n = m \otimes \xi$ for some $m \in \Gamma(Q, \gamma)$ and $\xi \in \omega_{Q/\ms{B}}$. As in \eqref{eq:intermediate polarization}, we may therefore write
\begin{equation} \label{eq:hodge positivity 1}
 \langle \eta, S(n, \bar{n}) \rangle =  \Res_{s = 0} s^{-1}\int_{\tilde{Q}} |f|^{2s} S_Q(m, \bar{m}) \eta',
\end{equation}
where
\[ \eta' = (-1)^{c(c - 1)/2} (2 \pi i)^c (\xi \wedge \bar{\xi}) \contract \eta|_{\tilde{Q}}\]
is a positive real top form on $\tilde{Q}$, $S_Q(m, \bar{m})$ is a positive real function on $Q$, $f \colon \tilde{Q} \to \mb{C}$ is a boundary equation, and the right hand side of \eqref{eq:hodge positivity 1} is defined by analytic continuation from $\operatorname{Re} s \gg 0$. It therefore suffices to show that the integral
\begin{equation} \label{eq:hodge positivity 2}
\int_{\tilde{Q}} S_Q(m, \bar{m}) \eta'
\end{equation}
converges (in which case no analytic continuation is required, so \eqref{eq:hodge positivity 1} is manifestly positive).

The convergence of \eqref{eq:hodge positivity 2} follows directly from results of Saito and is most conveniently derived from the formulation given in~\cite[Proposition 3.2]{budur-mustata-saito}. We will briefly indicate the argument here. Let us write $X$ for a resolution of singularities of $\bar Q$ with normal crossings and $\tilde{X} = X \times_{\ms{B}} \tilde{\ms{B}}$. We write, as before, $j\colon \tilde{Q} \to \tilde{\ms{B}}$ and also $\tilde j \colon \tilde{Q} \to \tilde{X}$ and $\pi\colon \tilde{X} \to \tilde{\ms{B}}$ for the obvious maps. As the morphisms $j$ and $\tilde j$ are affine, we see that $j_* = \pi_* \tilde j_*$ and all the functors are exact. Then, by~\cite[Proposition 3.2]{budur-mustata-saito} and a short calculation we conclude that 
\[
V^{>-1} F_c j_*\gamma \ = \ \pi_{\bigcdot} (V^{>-1}F_0 (\tilde j_*\gamma) \otimes \omega_{\tilde X/\tilde{\ms{B}}})\,;
\]
here we recall that $\pi_{\bigcdot}$ stands for push-forward as a $\bC$-sheaf, and we have written $V$ for the filtration induced by the $V$-filtration on the pushforward via the graph of the boundary equation $f$. By~\eqref{open embedding formula} and \eqref{dual open embedding formula} we have that $V^{>-1} F_c j_*\gamma= F_c j_{!*}\gamma$ and similarly $V^{>-1}F_0 (\tilde j_*\gamma) = F_0 (\tilde j_{!*}\gamma)$.

Locally on $\tilde{\ms{B}}$, we may therefore write $n = m \otimes \xi$ where $m \in F_0 \tilde{j}_{!*}\gamma$ and $\xi \in \omega_{X/\ms{B}}$. In the normal crossings case, an easy computation of the Hodge filtration and the polarization $S_Q$ in a local model shows that \eqref{eq:hodge positivity 2} converges as claimed.

This completes the proof of Proposition \ref{prop:hodge positivity}.

\section{Proof of Theorem \ref{thm:main theorem 2}}
\label{proof:main theorem 2}

The main purpose of this section is to give the proof of Theorem \ref{thm:main theorem 2}. The argument, which follows Beilinson and Bernstein's proof of Theorem \ref{thm:beilinson-bernstein} \cite{beilinson-bernstein} that the Jantzen filtration coincides with the weight filtration, works by relating the Jantzen forms to polarizations on nearby cycles via the intermediary of Beilinson's maximal extension. To this end, in \S\ref{subsec:beilinson pairing} we give a construction for induced pairings on Beilinson functors, which, in the case of nearby cycles, we relate to a construction of a polarization by Sabbah and Schnell in \S\ref{subsec:SS polarization}. We complete the proof of Theorem \ref{thm:main theorem 2} in \S\ref{subsec:proof of theorem 2}.

Since we will need to delve a little deeper into the theory of complex mixed Hodge modules in this section, we start off in \S\ref{subsec:SS MHM} with a summary of further details of the theory, following the treatment in the book project \cite{SS} of Sabbah and Schnell. For an alternative route to complex mixed Hodge modules, via Saito's theory \cite{S1, S2}, see Appendix \ref{appendix}, and \S\ref{subsec:Saito polarization} in particular for the proof that Sabbah and Schnell's polarization on nearby cycles used here agrees with Saito's.

\subsection{Complex mixed Hodge modules following Sabbah and Schnell}
 \label{subsec:SS MHM}

In this subsection, we recall the setup for the theory of complex mixed Hodge modules following \cite{SS}. For the sake of convenience, we will restrict to the algebraic case (in particular, all Hodge modules, polarizations, etc., are assumed to extend to projective compactifications). So we will consider quasi-projective complex algebraic varieties and algebraic regular holonomic left $\ms{D}$-modules on them. 

The data defining a mixed Hodge module on a variety $X$ are a regular holonomic $\ms{D}_X$-module $\ms{M}$, a weight filtration $W_{\bigcdot}\ms{M}$ by $\ms{D}_X$-submodules, and Hodge and conjugate Hodge filtrations $F_{\bigcdot}$ and $\bar{F}_{\bigcdot}$ specified as follows. The Hodge filtration, which varies holomorphically, is given by a filtration $F_{\bigcdot}\ms{M}$ on the $\ms{O}_X$-module $\ms{M}$ such that the $\ms{D}_X$-action is compatible with the order filtration on $\ms{D}_X$. The conjugate Hodge filtration, which varies anti-holomorphically, may be specified either by a similar filtration $\bar{F}_{\bigcdot} \ms{M}^c$ on the conjugate $\ms{D}$-module $\ms{M}^c$ determined by
\[ 
(\ms{M}^c)^{an} = \shom_{\ms{D}_{\bar{X}}}(\mb{D}(\overline{\ms{M}}), \Db_{X}) 
\]
or by its dual $F_{\bigcdot} \ms{M}^h$ on the Hermitian dual $\ms{M}^h$ to $\ms{M}$ determined by
\[
(\ms{M}^h)^{an} = \shom_{\ms{D}_{\bar{X}}}(\overline{\ms{M}}, \ms{D}\mit{b}_{X}) = (\bD \cM^c)^{an}.
\]
From the latter perspective, these data may be expressed explicitly (cf., \cite[\S 5.2.b]{SS}) by a triple
\[ (\ms{M}, \ms{M}', \mf{s}) = ((\ms{M}, F_{\bigcdot}\ms{M}, W_{\bigcdot}\ms{M}), (\ms{M}', F_{\bigcdot}\ms{M}', W_{\bigcdot}\ms{M}'), \mf{s}) \]
where
\begin{itemize}
\item $\ms{M}$ and $\ms{M}'$ are regular holonomic $\ms{D}_X$-modules,
\item $W_{\bigcdot} \ms{M}$ and $W_{\bigcdot}\ms{M}'$ are filtrations by $\ms{D}_X$-submodules,
\item $F_{\bigcdot}\ms{M}$ and $F_{\bigcdot} \ms{M}'$ are $\ms{D}_X$-filtrations,
\item $\mf{s} \colon \ms{M} \otimes \overline{\ms{M}'} \to \Db_X$ is a perfect sesquilinear pairing (i.e., it induces an isomorphism $\ms{M} \cong (\ms{M}')^h$ of the underlying $\ms{D}$-modules) compatible with $W_{\bigcdot}$.
\end{itemize}
Morphisms of triples are defined covariantly in $\ms{M}$ and contravariantly in $\ms{M}'$. We will realize mixed Hodge modules as triples in this way for the rest of this section.

For example \cite[12.7.8]{SS}, the trivial variation of Hodge structure on $X$ of dimension $n$ corresponds to the triple $(\ms{M}, \ms{M}', \mf{s})$, where $\ms{M} = \mc{O}_X$ (resp., $\ms{M}' = \mc{O}_X$) with Hodge and weight filtrations jumping in degrees $0$ and $n$ (resp., $n$ and $-n$) and $\mf{s}$ is given by
\[ \mf{s}(f, \bar{g}) = f\bar{g} \in \mc{C}^\infty_X \subset \Db_X.\]

In order to define a mixed Hodge module, the data $(\ms{M}, \ms{M}', \mf{s})$ are required to satisfy a complicated set of axioms. Since, at the time of writing, \cite{SS} does not contain the definition of \emph{mixed} Hodge modules or a discussion of the algebraic case, let us recall briefly from \cite[\S 4]{S2} how to build this theory from the notion of polarizable Hodge module \cite[14.2]{SS}. The full subcategory $\mhm(X)$ is defined for all $X$ as the largest system of categories satisfying the following conditions for $\ms{T} = (\ms{M}, \ms{M}', \mf{s}) \in \mhm(X)$:
\begin{enumerate}
\item  $\mrm{Gr}^W_w \ms{T}$ is a polarizable Hodge module of weight $w$ for all $w \in \mb{Z}$,
\item if $Y$ is a smooth variety, then $\ms{T} \boxtimes \mc{O}_Y \in \mhm(X \times Y)$,
\item if $U \subset X$ is open, then $\ms{T}|_U \in \mhm(U)$,
\item if $j \colon X \to Y$ is the inclusion of the complement of a weakly locally principal divisor, then there exist weight filtrations on the triples $j_!\ms{T}$ and $j_*\ms{T}$ so that these lie in $\mhm(Y)$, and
\item \label{itm:vanishing cycles exist} if $f \colon X \to \mb{C}$ is a regular function, then the nearby and vanishing cycles for $\ms{T}$ along $f$ are well-defined (i.e., the Hodge, weight and $V$-filtrations on $\ms{M}$ and $\ms{M}'$ are compatible, and the relative monodromy weight filtrations exist on $\mrm{Gr}_V^\alpha$ \cite[\S 2.2]{S2}), and are objects in $\mhm(X)$.
\end{enumerate}
Given this definition, it is a theorem (cf., \cite[Theorem 3.9]{S2}) that $\mhm(\mrm{pt})$ is the whole category of complex mixed Hodge structures, and more generally, that any graded-polarizable admissible variation of mixed Hodge structure defines a mixed Hodge module (cf., \cite[Theorem 3.27]{S2}).\footnote{Strictly speaking, there is not yet a proof of these theorems for complex mixed Hodge modules in the literature. The concerned reader may choose to work instead in the version of complex Hodge modules outlined in Appendix \ref{appendix}, in which the necessary statements follow from Saito's theorems.}

One of the main points of the theory is that mixed Hodge modules come with functors
\[ f_!, f_* \colon\oD^b(\mhm(X)) \to \oD^b(\mhm(Y)) \]
and
\[ f^!, f^* \colon \oD^b(\mhm(Y)) \to \oD^b(\mhm(X)) \]
for $f \colon X \to Y$ a morphism of smooth varieties, compatible with the same operations for regular holonomic $\ms{D}$-modules. They also admit a lift of Hermitian duality
\[ (-)^h \colon \mhm(X)^{op} \to \mhm(X), \]
and Hodge twist functors $\otimes \mb{C}_{p, q}$. Let us briefly recall how the most important of these (for our purposes) are defined in terms of triples.

The Hermitian duality and Hodge twists are simple: they are defined by the formulas
\[ (\ms{M}, \ms{M}', \mf{s})^h = (\ms{M}', \ms{M}, \mf{s}^h) \]
and
\[ (\ms{M}, \ms{M}', \mf{s}) \otimes \mb{C}_{p, q} = ((\ms{M}, F_{{\bigcdot} - p}\ms{M}, W_{{\bigcdot} + p + q}\ms{M}), (\ms{M}', F_{{\bigcdot} + q}\ms{M}', W_{{\bigcdot} - p - q}\ms{M}')),\]
where
\[ \mf{s}^h(m', \bar{m}) := \overline{\mf{s}(m, \bar{m}')}.\]
Let us also remark here that if $\ms{T} = (\ms{M}, \ms{M}', \mf{s})$ and $S \colon \ms{T} \to \ms{T}^h(-w)$ is a map determined by $S \colon \ms{M} \to \ms{M}'$, then the associated form on the underlying $\ms{D}$-module $\ms{M}$ is given by
\begin{align*}
S \colon \ms{M} \otimes \overline{\ms{M}} &\to \Db_X \\
m \otimes \overline{m'} &\mapsto \overline{\mf{s}(m', \overline{S(m)})}.
\end{align*}
The complex conjugates are arranged so that the formula above defines $\mb{C}$-linear map from $\hom(\ms{M}, \ms{M}')$ to sesquilinear forms on $\ms{M}$.

For closed immersions $i \colon X \to Y$, the pushforward of a triple $\ms{T} = (\ms{M}, \ms{M}', \mf{s})$ is defined by
\[ i_*\ms{T} = (i_*\ms{M}, i_*\ms{M}', i_*\mf{s}),\]
where $i_*\mf{s}$ is defined as in \eqref{eq:polarization closed pushforward} (cf., \cite[\S\S 0.2, 12.3.3, 12.4.a]{SS}), and $i_*\ms{M}$ and $i_*\ms{M}'$ are endowed with the Hodge filtrations of \S\ref{subsec:hodge filtration}. The weight filtration is defined by naively pushing forward the weight filtration for $\ms{T}$ in the obvious way. More general proper pushforwards are similarly explicit, with a little delicacy required in passing to the derived category (cf., \cite[Theorem 4.3]{S2}).

For the complement $j \colon U \to X$ of a principal divisor $f^{-1}(0)$, the $!$ and $*$ pushforwards are given by
\[ j_*\ms{T} = (j_*\ms{M}, j_!\ms{M}', j_*\mf{s}), \quad \text{and} \quad j_!\ms{T} = (j_!\ms{M}, j_*\ms{M}', j_!\mf{s}),\]
where $j_*\ms{M}$ and $j_!\ms{M}$ are endowed with the filtrations \eqref{open embedding formula} and \eqref{dual open embedding formula}. The pairings $j_*\mf{s}$ and $j_!\mf{s}$ are by definition the unique pairings restricting to $\mf{s}$ on $U$. (The existence of such a pairing is a non-trivial fact, but follows immediately from the defining adjunctions between $j^*$ and $j_!$, $j_*$ and Kashiwara's theorem \cite{kashiwara2} that Hermitian duality is an anti-equivalence of the category of regular holonomic $\ms{D}$-modules commuting with $j^*$.) The weight filtrations on $j_*\ms{T}$ and $j_!\ms{T}$ are defined implicitly by the condition \eqref{itm:vanishing cycles exist}.

\subsection{Induced pairings on Beilinson functors} \label{subsec:beilinson pairing}

Let $X$ be a smooth variety, $f \colon X \to \mb{C}$ a regular function and $j \colon U = f^{-1}(\mb{C}^\times) \to X$ the inclusion of the complement of the zero locus. Recall the notation of \S\ref{subsec:nearby cycles}, in particular Definition \ref{defn:beilinson functor} of the Beilinson functors $\pi_f^a$ from regular holonomic $\ms{D}$-modules on $U$ to regular holonomic $\ms{D}$-modules on $X$.

The functor $\pi_f^a$ lifts to a functor on complex mixed Hodge modules. We give the construction below in terms of triples as in \S\ref{subsec:SS MHM}.

Suppose that $\ms{T} = (\ms{M}, \ms{M}', \mf{s})$ is a mixed Hodge module on $U$. Endow $f^s\ms{M}((s))$ with Hodge filtration
\[ F_p f^s\ms{M}((s)) = \sum_{k + l \leq p} s^kf^s(F_l\ms{M}) \]
and weight filtration
\[ W_w f^s\ms{M}((s)) = \sum_{2k + l \leq w} s^{-k}f^s(W_l\ms{M})[[s]] \]
and similarly for $\ms{M}'$. The inherited filtrations endow the triple
\begin{equation} \label{eq:truncated fs triple}
\left(\frac{s^{-m}f^s\ms{M}[s]}{s^nf^s\ms{M}[s]}, \frac{s^{-n + 1}f^s\ms{M}'[s]}{s^{m + 1}f^s\ms{M}'[s]}, \Res s^{-1}\mf{s}\right)
\end{equation}
with the structure of a mixed Hodge module for all $m, n$ (\eqref{eq:truncated fs triple} is the tensor product of $\mc{M}$ with the pullback from $\mb{C}^\times$ of the admissible variation of mixed Hodge structure $t^s\mc{O}_{\mb{C}^\times}[s]/s^n$), where we write $\Res s^{-1}\mf{s}$ for the perfect pairing
\begin{align*}
 (\Res s^{-1}\mf{s})(f^s m(s), \overline{f^s m'(s)}) :=& \Res_{s = 0} s^{-1} |f|^{2s}\mf{s}(m(s), \overline{m'(s)}) \\
 =& \sum_{i, j} \Res_{s = 0} s^{i + j - 1} |f|^{2s}\mf{s}(m_i, \overline{m_j'})
\end{align*}
for Laurent series $m(s) = \sum_i m_i s^i \in \ms{M}((s))$ and $m'(s) = \sum_j m_j's^j\in \ms{M}'((s))$. We set
\begin{align*}
f^s\ms{T}[[s]] &= \varprojlim_{n} \left(\frac{f^s\ms{M}[s]}{s^nf^s\ms{M}[s]}, \frac{s^{-n + 1}f^s\ms{M}'[s]}{sf^s\ms{M}'[s]}, \Res s^{-1}\mf{s}\right) \\
&= \left(f^s\ms{M}[[s]], \frac{f^s\ms{M}'((s))}{sf^s\ms{M}'[[s]]}, \Res s^{-1}\mf{s}\right)
\end{align*}
and
\begin{align*}
  f^s\ms{T}((s)) &= \varinjlim_m \varprojlim_n \left(\frac{s^{-m}f^s\ms{M}[s]}{s^nf^s\ms{M}[s]}, \frac{s^{-n + 1}f^s\ms{M}'[s]}{s^{m + 1}f^s\ms{M}'[s]}, \Res s^{-1}\mf{s}\right) \\
&= (f^s\ms{M}((s)), f^s\ms{M}'((s)), \Res s^{-1}\mf{s}).
\end{align*}

\begin{defn}
The \emph{Beilinson functor $\pi_f^a$ for mixed Hodge modules} is defined by
\begin{align*}
\pi_f^a(\ms{T}) &= \coker( j_!f^s\ms{T}[[s]](a) \xrightarrow{s^a} j_*f^s\ms{T}[[s]]) \\
&= \left(\frac{j_*f^s\ms{M}[[s]]}{s^aj_!f^s\ms{M}[[s]]}, \frac{s^{-a + 1}j_*f^s\ms{M}'[[s]]}{sj_!f^s\ms{M}'[[s]]}, j_*(\Res s^{-1}\mf{s})\right)
\end{align*}
for a mixed Hodge module $\ms{T} = (\ms{M}, \ms{M}', \mf{s})$ as above. Here we interpret the above fomula by truncating at a high power of $s$ and performing the $j_!$ and $j_*$ operations in the category of mixed Hodge modules.
\end{defn}

Now suppose that $\ms{T}$ is pure of weight $w$ and $S \colon \ms{T} \to \ms{T}^h(-w)$ is a polarization. We define the induced pairing on $\pi_f^a\ms{T}$ to be the map
\[ \pi_f^a(S) \colon \pi_f^a\ms{T} \to (\pi_f^a\ms{T})^h(-w + a - 1) \]
induced by
\[ s^{-a + 1}S \colon f^s\ms{M}((s)) \to f^s\ms{M}'((s)).\]
At the level of pairings on $\ms{D}$-modules, $\pi_f^a(S)$ corresponds to the pairing
\[ \pi_f^a(S)(m, \overline{m'}) = j_*(\Res s^{-1}S)(m, s^{-a + 1}\overline{m'}) = j_*(\Res S)(m, s^{-a}\overline{m'}) \]
for $m, m' \in j_*f^s\ms{M}[[s]]$.

\subsection{Polarizations on nearby cycles according to Sabbah and Schnell} \label{subsec:SS polarization}

In this subsection, we recall the definition of induced polarization on nearby cycles from \cite{SS} and relate it to the one described above on $\pi_f^0$.

Continuing in the setting of \S\ref{subsec:beilinson pairing}, suppose that $\ms{T} = (\ms{M}, \ms{M}', \mf{s})$ is a mixed Hodge module on $U$, which we assume pure of weight $w$ for simplicity. In the notation of \S\ref{subsec:nearby cycles}, the nearby cycles $\psi_f^{un}\ms{M} := \mrm{Gr}_V^0i_{f*}j_*\ms{M}$ and $\psi_f^{un}\ms{M}'$ carry Hodge filtrations defined in the obvious way. The nearby cycles as a mixed Hodge module are defined by
\[ \psi_f^{un}\ms{T} = ((\psi_f^{un}\ms{M}, F_{\bigcdot}), (\psi_f^{un}\ms{M}', F_{\bigcdot + 1}), \psi_f^{un}\mf{s})\]
where the weight filtrations are given by the monodromy weight filtrations centered at $w - 1$ with respect to the operator $\mrm{N} = -t\partial_t$, and the pairing $\psi_f^{un}\mf{s}$ is defined as follows. This is \cite[Definitions 12.5.10 and 12.5.19]{SS}.

Given $m \in V^0i_{f*}j_*\ms{M}$ and $m' \in V^0i_{f*}j_*\ms{M}'$, we may lift the distribution $i_{f*}\mf{s}(m, \overline{m'})$ on $X \times \mb{C}^\times$ to a distribution on $X \times \mb{C}$, which we denote by the same symbol. (Note that this is a non-trivial condition on the distribution on $X \times \mb{C}^\times$, which follows ultimately from the assumption that $j_*\ms{M}$ and $j_*\ms{M}'$ are regular holonomic on $X$.) The lift is not unique, but for any test form $\tilde{\eta}$ on $X \times \mb{C}$, the function
\begin{equation} \label{eq:SS polarization 1}
 s \mapsto \langle |t|^{2s}\tilde{\eta}, i_{f*}\mf{s}(m, \overline{m'}) \rangle
\end{equation}
is analytic and independent of any choices for $\operatorname{Re} s \gg 0$. (Here the right hand side is defined for such $s \in \mb{C}$ by the observation that any smooth distribution is well-defined on $C^p$ test forms with a fixed compact support for $p \gg 0$.) It follows from the $b$-function lemma that \eqref{eq:SS polarization 1} analytically continues to a meromorphic function of $s \in \mb{C}$ \cite[Proposition 12.5.4]{SS}. The sesquilinear pairing $\psi_f^{un}\mf{s}$ is defined by
\[ \langle \eta, \psi_f^{un}\mf{s}(m, \overline{m'}) \rangle  = \Res_{s = -1} \langle |t|^{2s}\tilde{\eta}, i_{f*}\mf{s}(m, \overline{m'}) \rangle,\]
where $\tilde{\eta}$ is any test form satisfying
\[ \eta = 2\pi i(\partial_t \wedge \partial_{\bar{t}}) \contract \tilde{\eta}|_X.\]

Now suppose that $\ms{T}$ is equipped with a polarization $S \colon \ms{T} \to \ms{T}^h(-w)$. By construction, we have $\psi_f^{un}(\ms{T}^h) = (\psi_f^{un}\ms{T})^h(1)$, so we get an induced pairing
\[ \psi_f^{un}S \colon \psi_f^{un}\ms{T} \to (\psi_f^{un}\ms{T})^h(-w + 1).\]
At the level of pairings of $\ms{D}$-modules, this is given by
\[ \langle \eta, \psi_f^{un}S(m, \overline{m'}) \rangle = \Res_{s = -1}\langle |t|^{2s}\tilde{\eta}, i_{f*}S(m, \overline{m'})\rangle \]
for $m, m' \in V^0i_{f*}j_*\ms{M}$ and $\tilde{\eta}$ as above.

Recall from Proposition \ref{prop:beilinson nearby cycles} that $\psi_f^{un}$ is isomorphic to the Beilinson functor $\pi_f^0$ at the level of $\ms{D}$-modules. This upgrades to an isomorphism of mixed Hodge modules as follows.

\begin{prop} \label{prop:beilinson nearby cycles hodge}
Let $\ms{T} = (\ms{M}, \ms{M}', \mf{s})$ be a pure Hodge module of weight $w$ on $U$. Then we have an isomorphism
\[ \psi_f^{un}\ms{T} \cong \pi_f^0(\ms{T})(1) \]
of mixed Hodge modules on $X$, given by the isomorphisms $\psi_f^{un}\ms{M} \cong \pi_f^0(\ms{M})$ and $\psi_f^{un}\ms{M}' \cong \pi_f^0(\ms{M}')$ of Proposition \ref{prop:beilinson nearby cycles}. Moreover, if $S \colon \ms{T} \to \ms{T}^h(-w)$ is a polarization, then the above isomorphism identifies the pairings $\psi_f^{un}S$ and $\pi_f^0(S)$.
\end{prop}
\begin{proof}
One easily checks that the isomorphisms given by Proposition \ref{prop:beilinson nearby cycles} respect the Hodge filtrations. We next show that they are compatible with the pairings $j_*(\Res s^{-1}\mf{s})$ and $\psi_f^{un}\mf{s}$. It suffices to check that
\begin{equation} \label{eq:beilinson nearby cycles hodge 1}
 \psi_f^{un}\mf{s}(m \otimes \partial_t, \overline{m' \otimes \partial_t}) = j_*(\Res s^{-1}\mf{s})(f^s(f^{-1}m), \overline{f^s(f^{-1}m')})
\end{equation}
for $m, m' \in (V^0i_{f*}j_*\ms{M} \otimes \partial_t^{-1}) \cap i_{f\bigcdot}j_*\ms{M}$. Recall that $j_*(\Res s^{-1}\mf{s})$ is defined to be the unique ($\ms{D}_X \otimes \ms{D}_{\bar{X}}$-linear) pairing between $j_*f^s\ms{M}((s))$ and $j_!f^s\ms{M}'((s))$ whose restriction to $U$ is $\Res s^{-1}\mf{s}$. The formula
\begin{equation} \label{eq:beilinson nearby cycles hodge 2}
 \langle \eta, j_*(\Res s^{-1}\mf{s})(f^sm(s), \overline{f^sm'(s)}) \rangle = \sum_{i, j} \Res_{s = 0} s^{i + j - 1} \langle |f|^{2s}\eta, \mf{s}(m_i, \overline{m'_j})\rangle 
\end{equation}
defines such a pairing, where
\[ m(s) = \sum_i m_is^i \in j_*\ms{M}((s)) \quad \text{and} \quad m'(s) = \sum_j m'_j s^j \in j_*\ms{M}'((s))\]
and the function inside the residue is defined by analytic continuation from $\operatorname{Re} s \gg 0$ as in the definition of $\psi_f^{un}\mf{s}$. (Note that the order of the poles of $\langle |f|^{2s}\eta, \mf{s}(m, \overline{m'})\rangle$ is bounded independent of $m$ and $m'$ by \cite[Proposition 12.5.4]{SS}, so the sum in \eqref{eq:beilinson nearby cycles hodge 2} is finite.) So the right hand side of \eqref{eq:beilinson nearby cycles hodge 1} becomes
\[ \langle \eta, j_*(\Res s^{-1}\mf{s})(f^s(f^{-1}m), \overline{f^s(f^{-1}m')}) \rangle = \Res_{s = -1}\langle |f|^{2s}\eta, \mf{s}(m, \overline{m'})\rangle.\]
But this is equal to the left hand side of \eqref{eq:beilinson nearby cycles hodge 1} by applying the definition and the formula \eqref{eq:polarization closed pushforward} for $i_{f*}\mf{s}$.

This gives the desired isomorphism of triples (modulo the weight filtrations) and the asserted identification of the pairings induced from a polarization. It remains to show that the weight filtrations agree. Consider the unipotent vanishing cycles functor $\phi_f^{un} = \Gr_V^{-1}i_{f_*}$. By definition of the category of mixed Hodge modules, this lifts to a functor $\mhm(X) \to \mhm(X)$ such that
\[ t \colon \phi_f^{un}\mc{N} \to \psi_f^{un}j^*\mc{N}(-1) \quad \text{and} \quad \partial_t \colon \psi_f^{un}j^*\mc{N} \to \phi_f^{un}\mc{N} \]
are morphisms of mixed Hodge modules for all $\mc{N} \in \mhm(X)$, and $\phi_f^{un}\mc{N} = \mc{N}$ as mixed Hodge modules if $\mc{N}$ is supported in $f^{-1}(0)$. Hence, for $n \gg 0$, we have
\begin{align*}
\pi_f^0(\mc{T}) &= \phi_f^{un}\pi_f^0(\mc{T}) \\
&= \coker\left(\phi_f^{un}j_!\left(\frac{f^s\mc{T}[s]}{(s^n)}\right) \to \phi_f^{un}j_*\left(\frac{f^s\mc{T}[s]}{(s^n)}\right)\right) \\
&= \coker\left(t\partial_t \colon \psi_f^{un}\left(\frac{f^s\mc{T}[s]}{(s^n)}\right) \to \psi_f^{un}\left(\frac{f^s\mc{T}[s]}{(s^n)}\right)(-1)\right)
\end{align*}
as mixed Hodge modules. Now, the $\mc{D}$-module underlying $\psi_f^{un}(f^s\mc{T}[s]/(s^n))$ is
\[ \mc{P} := \psi_f^{un}\left(\frac{f^s\mc{M}[s]}{(s^n)}\right) = \frac{\psi_f^{un}\mc{M}[s]}{(s^n)},\]
with $t\partial_t$ acting by $s - \mrm{N}$. Filtering $f^s\mc{T}[s]/(s^n)$ by the subquotients $s^k\mc{T}$, we see that the weight filtration on $\mc{P}$ satisfies
\begin{equation} \label{eq:beilinson nearby cycles hodge 3}
(\mrm{N} - s)(W_r\mc{P}) \subset W_{r - 2}\mc{P} \quad \text{and} \quad \frac{W_r\mc{P} \cap s^k\psi_f^{un}\mc{M}[s]}{W_r\mc{P} \cap s^{k + 1}\psi_f^{un}\mc{M}[s]} = s^k W_{r + 2k}\psi_f^{un}\mc{M}.
\end{equation}
In other words, $W_{\bigcdot}\mc{P}$ is, up to a shift, the relative monodromy weight filtration with respect to $\mrm{N} - s$ and the filtration defined by powers of $s$. By uniqueness of the relative mondromy weight filtration,
\[ W_r\mc{P} = \sum_k s^k W_{r + 2k}\psi_f^{un}\mc{M}\]
is the unique filtration satisfying \eqref{eq:beilinson nearby cycles hodge 3}. Taking the image modulo $t\partial_t = s - \mrm{N}$, we find that
\[ W_r\pi_f^0(\mc{M})(1) = \sum_k \mrm{N}^k W_{r + 2k} \psi_f^{un}\mc{M} = W_{r}\psi_f^{un}\mc{M}\]
as claimed.
\end{proof}

Finally, let us recall how to obtain polarized Hodge modules from the nearby cycles construction. By construction of the monodromy weight filtration, the monodromy operator $\mrm{N} \colon \psi_f^{un}\ms{T} \to \psi_f^{un}\ms{T}(-1)$ induces isomorphisms
\[ \mrm{N}^n \colon \mrm{Gr}^W_{w - 1 + n} \psi_f^{un}\ms{T} \overset{\sim}\to \mrm{Gr}^W_{w - 1 - n} \psi_f^{un}\ms{T}(-n)\]
for all $n \geq 0$. The subobjects of Lefschetz primitive parts are defined by
\[ \mrm{P}_n\psi_f^{un}\ms{T} = \ker(\mrm{N}^{n + 1} \colon \mrm{Gr}^W_{w - 1 + n} \psi_f^{un}\ms{T} \to \mrm{Gr}^W_{w - 3 - n}\psi_f^{un}\ms{T}(-n - 1)) \subset \mrm{Gr}^W_{w - 1 + n}\psi_f^{un}\ms{T}.\]
One of the axioms for polarized Hodge modules \cite[Definition 14.2.2]{SS} is that the pairing
\begin{equation} \label{eq:SS primitive polarization}
(-1)^n \mrm{N}^n \circ \psi_f^{un}S \colon \mrm{P}_n\psi_f^{un}\ms{T} \to (\mrm{P}_n\psi_f^{un}\ms{T})^h(-w + 1 - n)
\end{equation}
is a polarization on $\mrm{P}_n\psi_f^{un}\ms{T}$. We therefore deduce the following from Proposition \ref{prop:beilinson nearby cycles hodge}.

\begin{cor} \label{cor:beilinson polarization}
Let $\ms{T} = (\ms{M}, \ms{M}', \mf{s})$ be a pure Hodge module of weight $w$ on $U$ and $S \colon \ms{T} \to \ms{T}^h(-w)$ a polarization. Then $\mrm{P}_n\pi_f^0(\ms{T})$ is pure of weight $w + 1 + n$, and the pairing
\[ (-1)^{n - 1} s^n \circ \pi_f^0(S) \colon \mrm{P}_n\pi_f^0(\ms{T}) \to (\mrm{P}_n\pi_f^0(\ms{T}))^h(-w - 1 - n) \]
is a polarization. Here $\mrm{P}_n$ denotes the Lefschetz primitive parts with respect to the nilpotent operator $s$.
\end{cor}

As a pairing on the underlying $\ms{D}$-modules, the polarization of Corollary \ref{cor:beilinson polarization} is given by
\[ m \otimes \overline{m'} \mapsto (-1)^{n - 1} j_*(\Res s^{-1} S)(m, s^{n + 1}\overline{m'})\]
for $m, m' \in \mrm{P}_n\pi_f^0(\ms{M})$.

\subsection{Proof of Theorem \ref{thm:main theorem 2}} \label{subsec:proof of theorem 2}

In this subsection, we complete the proof of Theorem \ref{thm:main theorem 2}. The statement in fact holds for polarized Hodge modules on any smooth variety, so we will prove it in this context.

As in \S\ref{subsec:beilinson pairing}, let $X$ be a smooth variety, $f \colon X \to \mb{C}$ a regular function and set $U = f^{-1}(\mb{C}^\times)$. If $\ms{T} = (\ms{M}, \ms{M}', \mf{s})$ is a mixed Hodge module on $U$, pure of weight $w$, and $S \colon \ms{T} \to \ms{T}^h(-w)$ is a polarization, then we have Jantzen filtrations
\[ J_n j_!\ms{T} = (j_!f^s\ms{T}[[s]] \cap s^{-n} j_*f^s\ms{T}[[s]])/(s) \]
and
\[ J_n j_*\ms{T} = (s^{-n}j_!f^s\ms{T}[[s]] \cap j_*f^s\ms{T}[[s]])/(s),\]
where the intersections are taken inside $j_!\ms{T}((s)) = j_*\ms{T}((s))$, and forms
\[ s^{-n}\mrm{Gr}_{-n}^J(S) \colon \mrm{Gr}_{-n}^J(j_!\ms{T}) \overset{\sim}\to \mrm{Gr}_{-n}^J(j_!\ms{T})^h(-w + n) \]
and
\[ s^n\mrm{Gr}_n^J(S) \colon \mrm{Gr}_n^J(j_*\ms{T}) \overset{\sim}\to \mrm{Gr}_n^J(j_*\ms{T})(- w - n).\]
We will show:

\begin{thm} \label{thm:jantzen polarization}
In the setting above, for all $n \geq 0$, $\mrm{Gr}_{-n}^J(j_!\ms{T})$ is a pure Hodge module on $X$ of weight $w - n$, and the Jantzen form
\[ s^{-n}\mrm{Gr}_{-n}^J(S) \colon \mrm{Gr}_{-n}^J(j_!\ms{T}) \to \mrm{Gr}_{-n}^J(j_!\ms{T})^h(- w + n) \]
is a polarization.
\end{thm}

\begin{proof}[Proof of Theorem \ref{thm:main theorem 2}]
Since the claim is local on $\tilde{\ms{B}}$, it suffices to show that the claim holds on any open $X \subset \tilde{\ms{B}}$ such that $f_\varphi$ is the restriction of $f \colon X \to \mb{C}$. Now apply Theorem \ref{thm:jantzen polarization} to the polarized Hodge module $(i_*\gamma, i_*S)$, where $i \colon Q \cap X \to f^{-1}(\mb{C}^\times)$ is the inclusion.
\end{proof}

It remains to prove Theorem \ref{thm:jantzen polarization}. To this end, consider the Beilinson functors $\pi_f^a(\ms{T})$ for $a \geq 0$. The nilpotent endomorphism $s \colon \pi_f^a(\ms{T}) \to \pi_f^a(\ms{T})(-1)$ defines increasing image and kernel filtrations $I$ and $K$ given by
\[ I_n \pi_f^a(\ms{T}) = s^{-n}\pi_f^a(\ms{T}) \quad \text{and} \quad K_n \pi^a_f(\ms{T}) = \ker(s^{n + 1}).\]
We also have the monodromy weight filtration
\[ M_n\pi_f^a(\ms{T}) = \sum_{p + q = n} I_p\pi_f^a(\ms{T}) \cap K_q\pi^a_f(\ms{T}),\]
and the Lefschetz primitive parts
\[ \mrm{P}_n\pi_f^a(\ms{T}) = \ker(s^{n + 1} \colon \mrm{Gr}_n^M\pi_f^a(\ms{T}) \to \mrm{Gr}_{-n - 2}^M \pi_f^a(\ms{T})(- n - 1)).\]
The endomorphism $s$ is self-adjoint with respect to $\pi_f^a(S)$, and hence the filtration $M$ is self-dual. We thus obtain Hermitian forms
\begin{align*}
s^n \mrm{Gr}_n^M\pi_f^a(S) \colon \mrm{Gr}_n^M \pi_f^a(\ms{T}) &\xrightarrow{\pi_f^a(S)} (\mrm{Gr}_{-n}^M\pi_f^a(\ms{T}))^h(-w + a - 1) \\
&\xrightarrow{s^n} (\mrm{Gr}_n^M\pi_f^a(\ms{T}))^h(-w - n + a - 1)
\end{align*}
for all $n \geq 0$.

Now consider the maximal extension $\Xi_f(\ms{T}) := \pi_f^1(\ms{T})$. Note that we have a canonical factorization
\[ j_!\ms{T} = \ker(s) \to \Xi_f(\ms{T}) \to \coker(s) = j_*\ms{T}.\]

\begin{prop} \label{prop:jantzen monodromy}
We have the following.
\begin{enumerate}
\item \label{itm:jantzen monodromy 1} The Jantzen filtrations are given by
\[ J_nj_!\ms{T} = j_!\ms{T} \cap M_n\Xi_f(\ms{T}) = K_0\Xi_f(\ms{T}) \cap M_n\Xi_f(\ms{T}) \]
and
\[ J_nj_*\ms{T} = \mrm{im}(M_n\Xi_f(\ms{T}) \to j_*\ms{M}) = \frac{M_n\Xi_f(\ms{T})}{I_{-1}\Xi_f(\ms{T}) \cap M_n\Xi_f(\ms{T})}.\]
\item \label{itm:jantzen monodromy 2} The morphism
\begin{equation} \label{eq:jantzen monodromy 1}
 \mrm{P}_n \Xi_f(\ms{T}) \subset \mrm{Gr}_n^M\Xi_f(\ms{T}) \to \mrm{Gr}_n^Jj_*\ms{T} 
\end{equation}
is an isomorphism.
\item \label{itm:jantzen monodromy 3} The Jantzen forms $s^n\mrm{Gr}_n^J(S)$ agree with the restriction of $s^n\mrm{Gr}_n^M\Xi_f(S)$ under \eqref{eq:jantzen monodromy 1}.
\end{enumerate}
\end{prop}
\begin{proof}
Unwinding the definitions, one finds that the Jantzen filtrations are given by
\[ J_n j_!\ms{T} = j_!\ms{T} \cap I_n\Xi_f(\ms{T}) \quad \text{and} \quad J_n j_*\ms{T} = \mrm{im}(K_n\Xi_f(\ms{T}) \to j_*\ms{T}).\]
Straightforward linear algebra now shows \eqref{itm:jantzen monodromy 1}. For \eqref{itm:jantzen monodromy 2}, note that
\[ \mrm{Gr}_n^M\Xi_f(\ms{T}) = \mrm{P}_n\Xi_f(\ms{T}) \oplus s\mrm{Gr}_{n + 2}^M\Xi_f(\ms{T}) \]
by the Lefschetz decomposition. Since the $J$ filtration on $j_*\ms{T}$ is the image of the $M$ filtration on $\Xi_f(\ms{T})$, the second term is the kernel of the surjection $\mrm{Gr}_n^M\Xi_f(\ms{T}) \to \mrm{Gr}_n^J j_*\ms{T}$, so the result follows. Finally, \eqref{itm:jantzen monodromy 3} follows easily from the claim that $\pi_f^0(S)$ restricts to the pairing $j_*(S)$ between $j_*\ms{T}$ and $j_!\ms{T}$, which is clear from the definition.
\end{proof}

\begin{prop} \label{prop:jantzen monodromy 2}
For $n > 0$, we have 
\[ \mrm{Gr}_n^J j_*\ms{T} \cong \mrm{P}_{n - 1} \pi_f^0(\ms{T}) \]
and the isomorphism identifies the Jantzen form on the left with the restriction of $s^{n - 1}\mrm{Gr}_{n - 1}^M\pi_f^0(S)$ on the right.
\end{prop}
\begin{proof}
From the definitions, we have
\[ \pi_f^0(\ms{T}) = \frac{\Xi_f(\ms{T})}{K_0\Xi_f(\ms{T})}. \]
It follows formally that, for $n > 0$, the image of $M_n\Xi_f(\ms{T})$ under the quotient map is $M_{n - 1}\pi_f^0(\ms{T})$, and that the induced maps
\[ \mrm{P}_n\Xi_f(\ms{T}) \to \mrm{P}_{n - 1}\pi_f^0(\ms{T}) \]
are isomorphisms. These isomorphisms clearly send $s^n\mrm{Gr}_n^M\Xi_f(S)$ to $s^{n - 1}\mrm{Gr}_{n - 1}^M\pi_f^0(S)$, so we conclude using Proposition \ref{prop:jantzen monodromy}.
\end{proof}

\begin{proof}[Proof of Theorem \ref{thm:jantzen polarization}]
By construction,
\[ s^n \colon \mrm{Gr}_n^Jj_*\ms{T} \to \mrm{Gr}_{-n}^Jj_!\ms{T}(-n) \]
is an isomorphism, so the statement is equivalent to showing that $(-1)^n s^n\mrm{Gr}_n^J(S)$ is a polarization. For $n = 0$, we have
\[ \mrm{Gr}_0^J(j_*\ms{T}) = j_{!*}\ms{T} \quad \text{and} \quad \mrm{Gr}_0^J(S) = j_{!*}S,\]
so the statement is standard in the theory of Hodge modules (cf., \cite[Theorem 15.3.1]{SS}). For $n > 0$, we have
\[ (-1)^n s^n \mrm{Gr}_n^J(S) = (-1)^n s^{n - 1}\mrm{Gr}_{n - 1}^M\pi_f^0(S)|_{\mrm{P}_{n - 1}\pi_f^0(S)}\]
by Proposition \ref{prop:jantzen monodromy 2}. But the latter is precisely the polarization given by Corollary \ref{cor:beilinson polarization} via nearby cycles, so we are done.
\end{proof}

\appendix
\section{Complex mixed Hodge modules via Saito's theory} \label{appendix}

In this appendix, we explain how to recover the theory of complex mixed Hodge modules from the theory of real mixed Hodge modules of Saito. This is an expanded version of the idea explained in \cite[Appendix A]{schmid-vilonen}. In \S\ref{subsec:Saito MHM}, we define the category of Saito complex mixed Hodge modules and the functor relating these to the category of triples discussed in \S\ref{subsec:SS MHM}. We discuss polarizations in this context in \S\ref{subsec:mhm polarization}, and the nearby cycles construction in \S\ref{subsec:Saito polarization}.

The key mixed Hodge-theoretic results we have used in this paper are:
\begin{enumerate}
\item \label{itm:hodge filtration} The Hodge filtration on a pushforward is given as in \S\ref{subsec:hodge filtration}.
\item \label{itm:vhs} Every polarized variation of Hodge structures defines a polarized Hodge module.
\item \label{itm:closed pushforward} If $(\ms{M}, S)$ is a polarized Hodge module on $X$ and $i \colon X \to Y$ is a closed embedding, then $(i_*\ms{M}, i_*S)$ is a polarized Hodge module on $Y$, where $i_*S$ is defined by \eqref{eq:polarization closed pushforward}.
\item \label{itm:ic pushforward} If $(\ms{M}, S)$ is a polarized Hodge module on $X$ and $j \colon X \to Y$ is an open embedding with complement a divisor, then $(j_{!*}\ms{M}, j_{!*}S)$ is a polarized Hodge module on $Y$, where $j_{!*}S$ is the unique pairing restricting to $S$ on $U$.
\item \label{itm:nearby cycles} In the setting of \S\ref{subsec:SS polarization}, if $(\ms{T}, S)$ is a polarized Hodge module, then so is $\mrm{P}_n\psi_f^{un}\ms{T}$ with the polarization \eqref{eq:SS primitive polarization}.
\end{enumerate}
Items \eqref{itm:hodge filtration} and \eqref{itm:ic pushforward} are obvious given the definitions in \S\ref{subsec:Saito MHM} and the analogous claims for real Hodge modules. Items \eqref{itm:vhs} and \eqref{itm:closed pushforward} require a simple check of the sign conventions in \cite{S1} and \S\ref{subsec:mhm polarization}. It turns out that \eqref{itm:nearby cycles} is non-trivial, however, so we provide a sketch of a proof in \S\ref{subsec:Saito polarization}. Since \eqref{itm:hodge filtration}--\eqref{itm:nearby cycles} all hold in the theory presented here, the results of this paper also hold in this context.

\subsection{Complex mixed Hodge modules via Saito's theory}
\label{subsec:Saito MHM}

In \cite{S1, S2}, Saito constructs a theory of rational mixed Hodge modules on complex analytic spaces, and on quasi-projective algebraic varieties in particular. Since we are interested only in the algebraic case, we will denote by $\mhm_{\mb{Q}}(X)$ the category of (algebraic) mixed Hodge modules defined in \cite[\S 4]{S2} for $X$ quasi-projective. If $X$ is smooth, the objects in this category are (certain very special) tuples
\[ ((\ms{M}, F_{\bigcdot}\ms{M}), (K, W_{\bigcdot}K), \alpha),\]
where
\begin{itemize}
\item $\ms{M}$ is a regular holonomic $\ms{D}_X$-module with quasi-unipotent monodromy around all hypersurfaces,
\item $F_{\bigcdot}$ is a $\ms{D}_X$-filtration on $\ms{M}$,
\item $K$ is a perverse sheaf of $\mb{Q}$-vector spaces on $X$,
\item $W_{\bigcdot}$ is a filtration on $K$ by perverse subsheaves, and
\item $\alpha \colon \DR(\ms{M}) \to K \otimes_{\mb{Q}} \mb{C}$ is an isomorphism of complex perverse sheaves.
\end{itemize}
One may modify Saito's definitions ever so slightly by replacing $\mb{Q}$ with $\mb{R}$ and the assumption of quasi-unipotence with $\mb{R}$-specializability (see \cite[Chapters 9 and 10]{SS}); the latter assumption amounts to allowing $V$-filtrations indexed by $\mb{R}$ instead of $\mb{Q}$.\footnote{For the purposes of this paper, this generalization is only necessary to work with irrational twistings $\lambda$, as for $\lambda \in \mf{h}^*_\mb{Q}$ all our $\ms{D}$-modules are quasi-unipotent and even defined over a number field.} These changes require generalizations of some foundational results on degenerations of Hodge structures such as the $SL(2)$-orbit theorem, but otherwise do not introduce any significant difficulties to the theory. The relevant results are proved for example in \cite{SS2}, so this yields categories $\mhm_{\mb{R}}(X)$ of real mixed Hodge modules for quasi-projective varieties $X$ with the same functoriality properties as $\mhm_{\mb{Q}}(X)$.

Given the $\mb{R}$-linear abelian category $\mhm_{\mb{R}}(X)$, we may now define a complex mixed Hodge module $\ms{T}$ to be a real mixed Hodge module $\ms{T} = (\ms{M}, K, \alpha) \in \mhm_{\mb{R}}(X)$ equipped with an action $\mb{C} \otimes_{\mb{R}} \ms{T} \to \ms{T}$ of the field of complex numbers. This defines a $\mb{C}$-linear category
\[ \mhm_{\mb{C}}(X) = \mhm_\mb{R}(X) \otimes_\mb{R} \mb{C} \]
of complex mixed Hodge modules on $X$.

The relation to the theory of Sabbah and Schnell described in \S\ref{subsec:SS MHM} is as follows. Given a triple $\ms{T} = (\ms{M}, \ms{M}', \mf{s}) \in \mhm(X)$ as in \S\ref{subsec:SS MHM}, we set
\[ \saito{\ms{T}} = (\ms{M} \oplus \mb{D}\ms{M}', \DR(\ms{M}), \alpha) \in \mhm_\mb{C}(X),\]
where the $\mb{C}$-structure is given by the usual complex multiplication on $\ms{M}$ and $\DR(\ms{M})$ and the conjugate one on $\mb{D}\ms{M}'$. Here the isomorphism
\[ \alpha \colon \DR(\ms{M}) \otimes_{\mb{R}} \mb{C} = \DR(\ms{M}) \oplus \overline{\DR(\ms{M})} \to \DR(\ms{M}) \oplus \mb{D}\DR(\ms{M}') = \DR(\ms{M} \oplus \mb{D}\ms{M}') \]
is given by the identity on the first summand and on the second by the map
\[ \overline{\DR(\ms{M})} \to \mb{D}\DR(\ms{M}') \]
classifying the pairing $\overline{\DR(\mf{s})}$, where $\DR(\mf{s})$ is defined by
\begin{align}
\DR(\mf{s}) \colon \DR(\ms{M}) \otimes \overline{\DR(\ms{M}')} &\to \Db_X \otimes \mc{E}^{\bigcdot}_X[2n] \cong \underline{\mb{C}}_X(n)[2n] \label{eq:hermitian de Rham}\\
(m \otimes \alpha) \otimes \overline{(m' \otimes \beta)} &\mapsto (-1)^{\frac{n(n - 1)}{2} + n \cdot \deg \alpha} \mf{s}(m, \overline{m'}) \otimes (\alpha \wedge \bar{\beta}), \nonumber
\end{align}
where $n = \dim X$. Here $\ms{E}_X^{\bigcdot}$ denote the sheaves of smooth differential forms on $X$ and the right hand side of \eqref{eq:hermitian de Rham} is equipped with the de Rham differential. This construction identifies $\mhm_\mb{C}(X)$ with a full subcategory of triples $(\ms{M}, \ms{M}', \mf{s})$. Allowing that this may a priori differ from the category $\mhm(X)$ defined via Sabbah and Schnell's theory, we will call the objects in this subcategory Saito mixed Hodge modules.

\subsection{Polarizations} \label{subsec:mhm polarization}

In this subsection, we explain how Saito's notion of polarized Hodge module is related to the $\ms{D}$-module notion used in this paper and \cite{SS}.

Suppose that $\ms{T} = (\ms{M}, \ms{M}', \mf{s}) \in \mhm(X)$ is a pure Hodge module of weight $w$, and let $S \colon \ms{T} \to \ms{T}^h(-w)$ be a morphism satisfying $S^h = S$. For simplicity, we will identify $S$ with the associated Hermitian form 
\begin{align*}
S \colon \ms{M} \otimes \overline{\ms{M}} &\to \Db_X \\
m \otimes \overline{m'} &\mapsto \mf{s}(m, \overline{S(m')}) = \overline{\mf{s}(m', \overline{S(m)})}
\end{align*}
on the underlying $\ms{D}$-module $\mc{M}$. Applying $\DR$, we obtain an induced pairing
\[ \DR(S) \colon \DR(\ms{M}) \otimes \overline{\DR(\ms{M})} \to \underline{\mb{C}}_X(n - w)[2n]\]
on perverse sheaves via \eqref{eq:hermitian de Rham}, which one easily checks is $(-1)^n$-Hermitian with respect to the usual sign rules for interchanging tensor products of complexes. The real bilinear form
\[ S_\mb{R} := \operatorname{Re}^{(n - w)}\DR(S) \colon \DR(\ms{M}) \otimes \DR(\ms{M}) \to \underline{\mb{R}}_X(n - w)[2n] \]
is therefore $(-1)^w$-symmetric, where
\[ \operatorname{Re}^{(n - w)} = (2 \pi i)^{n - w} \operatorname{Re}((2 \pi i)^{w - n} \cdot) \colon \mb{C}(n - w) \to \mb{R}(n - w) = (2 \pi i)^{n - w}\mb{R} \]
is the real part with respect to the real structure on the Tate module.

\begin{defn}
Let $\ms{T} = (\ms{M}, \ms{M}', \mf{s})$ be a triple in the sense of \S\ref{subsec:SS MHM} such that $\saito{\ms{T}} \in \mhm_\mb{C}(X)$ is pure of weight $w$. We say that $S \colon \ms{T} \to \ms{T}^h(-w)$ is a \emph{Saito polarization} if $S_\mb{R}$ is a polarization of $\saito{\ms{T}}$ as an algebraic real mixed Hodge module \cite[5.2.10]{S1}\cite[4.1]{S2}.
\end{defn}

We make the following claim, although this is not needed for our purposes. See also \cite[\S 13.5]{mochizuki}, especially \cite[Proposition 13.5.4]{mochizuki} for a similar result in the related setting of mixed twistor $\ms{D}$-modules.

\begin{prop} \label{prop:theories agree}
A pair $(\ms{T}, S)$ of $\ms{T} = (\ms{M}, \ms{M}', \mf{s})$ and $S \colon \ms{T} \to \ms{T}^h(-w)$ is a polarized Hodge module in the sense of \cite[Definition 14.2.2]{SS} if and only if $\ms{T}$ is a Saito mixed Hodge module and $S$ is a Saito polarization, i.e., the pair $(\saito{\ms{T}}, S_\mb{R}) \in \mhm_\mb{R}(X)$ is a polarized Hodge module in the sense of Saito.
\end{prop}
\begin{proof}[Proof sketch]
Since the definitions have a similar inductive form, one only has to check that the functor $(\ms{T}, S) \mapsto (\saito{\ms{T}}, S_\mb{R})$ is compatible with the operations of pushforward under closed embeddings and nearby and vanishing cycles defined in \cite{SS} and \cite{S1}. Compatibility with pushforwards is a simple calculation using \eqref{eq:hermitian de Rham} and compatibility with nearby cycles is discussed in \S\ref{subsec:Saito polarization}. We omit the discussion of vanishing cycles, which is not used directly in this paper.
\end{proof}

\subsection{Polarizations on nearby cycles according to Saito} \label{subsec:Saito polarization}

In this subsection, we recall Saito's construction of the polarization on the primitive parts of nearby cycles, and show that it is the real part of the polarization of \S\ref{subsec:SS polarization}. In particular, we conclude (Corollary \ref{cor:polarizations agree}) that Sabbah and Schnell's construction does indeed produce a Saito polarization, with correct signs.

Take $X$, $f$ and $U$ as in \S\ref{subsec:beilinson pairing}. We will assume for simplicity that the divisor $D = f^{-1}(0)$ is smooth: the general case is reduced to this one by the usual trick of embedding via the graph of $f$.

Let $K$ be a real perverse sheaf on $U$. The (perverse) nearby cycles of the perverse sheaf $K$ are defined by
\[ \psi_f K = i_*i^*j_*\pi_*\pi^* K[-1],\]
where
\[ \pi \colon \tilde{U} = U \times_{\mb{C}^\times} \widetilde{\mb{C}^\times} \to U \]
is the pullback of the universal cover of $\mb{C}^\times$ and $i \colon D = f^{-1}(0) \to X$ is the inclusion of the divisor defined by $f$. (Note: the perverse sheaf $\psi_f K$ is denoted by $\vphantom{\psi_f}^p\psi_f K$ in \cite{S1}.) The monodromy around an oriented loop in $\mb{C}^\times$ defines an operator $\mrm{T} \colon \psi_f K \to \psi_f K$. The unipotent nearby cycles are the sub-perverse sheaf
\[ \psi_f^{un} K = \bigcup_n \ker(\mrm{T} -  1)^n \subset \psi_f K.\]

\begin{defn}[{\cite[5.1.3.3]{S1}}]
Suppose that $\ms{T} = (\ms{M}, K, \alpha)$ is a real mixed Hodge module on $U$. The \emph{unipotent nearby cycles of $\ms{T}$} are
\[ \psi_f^{un}\ms{T} = (\psi_f^{un}\ms{M}, \psi_f^{un}K, \psi_f^{un}\alpha)\]
where $\psi_f^{un}\ms{M}$ is the functor defined in \S\ref{subsec:nearby cycles} and the isomorphism $\psi_f^{un}\alpha \colon \DR(\psi_f^{un}\ms{M}) \to \psi_f^{un}K \otimes_{\mb{R}} \mb{C}$ is defined by composing the isomorphism $\psi_f^{un}\DR(\ms{M}) \cong \psi_f^{un}K \otimes_{\mb{R}} \mb{C}$ coming from $\alpha$ with the natural isomorphism
\begin{equation} \label{eq:nearby cycles de rham}
\DR(\psi_f^{un}\ms{M}) \cong \psi_f^{un}\DR(\ms{M})
\end{equation}
defined below.
\end{defn}

The isomorphism \eqref{eq:nearby cycles de rham} is constructed as follows. First, make the identification
\[ \psi_f^{un}\ms{M} \cong \pi_f^0\ms{M} = \frac{j_*f^s\ms{M}[[s]]}{j_!f^s\ms{M}[[s]]} \]
via Proposition \ref{prop:beilinson nearby cycles}. By the $b$-function lemma, the middle term of the exact triangle
\begin{equation} \label{eq:beilinson triangle}
 f^s\ms{M}[[s]] \to f^s\ms{M}((s)) \to \frac{f^s\ms{M}((s))}{f^s\ms{M}[[s]]} \xrightarrow{\beta} f^s\ms{M}[[s]][1].
\end{equation}
is sent to zero under the functor $C = \operatorname{Cone}(j_! \to j_*)$, so $C(\beta)$ is an isomorphism.
We therefore have a quasi-isomorphism
\[ \DR(\psi_f^{un}\ms{M}) \cong \DR \circ C(f^s\ms{M}[[s]]) \xrightarrow{-\DR(C(\beta))^{-1}[-1]} \DR \circ C\left(\frac{f^s\ms{M}((s))}{f^s\ms{M}[[s]]}\right)[-1].\]
Composing with the quasi-isomorphisms (note that $\DR \circ j_! \cong j_! \circ \DR$ and $\DR \circ j_* \cong j_* \circ \DR$ for regular holonomic $\mc{D}$-modules and hence for inverse limits of such by the Riemann-Hilbert correspondence)
\[ \DR \circ C \cong \operatorname{Cone}(j_! \DR \to j_*\DR) \cong i_*i^*j_* \DR \]
and taking residues at $s = 0$ provides a map from this to $\psi_f\DR(\ms{M})$, which factors through the desired isomorphism to $\psi_f^{un}\DR(\ms{M})$.

Now suppose that the real mixed Hodge module $(\ms{M}, K, \alpha)$ is pure of weight $w$ and that $S \colon K \otimes K \to \underline{\mb{R}}_U(d - w)[2d]$ is a polarization, where $d = \dim X$. Saito \cite[5.2.3]{S1} defines an induced pairing
\[ \psi_f S \colon \psi_f K \otimes \psi_f K \to \underline{\mb{R}}_X(d - w + 1)[2d] \]
 as follows. First, we have the morphism
\begin{align}
  \psi_f K \otimes \psi_f K = i_*&i^*j_*\pi_*\pi^* K [-1] \otimes i_*i^*j_*\pi_*\pi^* K [-1]  \nonumber \\
&\to i_*i^*j_*\pi_*\pi^* (K \otimes K)[-2]\xrightarrow{-S} i_*i^*j_*\pi_*\pi^* \underline{\mb{R}}_U(d - w)[2d-2].\label{eq:saito polarization 1}
\end{align}
Note that the sign rules for interchanging shifts and tensor products produce an implicit sign $(-1)^{\deg}$ on the first tensor factor in \eqref{eq:saito polarization 1}. Note that we take $-S$ rather than $S$ above to ensure that the shift of \eqref{eq:saito polarization 1} by $[1]$ recovers the obvious pairing on $i_*i^*j_*\pi_*\pi^* K$. Next, we have a canonical isomorphism
\begin{equation} \label{eq:saito polarization 2}
 i_*i^*j_*\pi_*\pi^* \underline{\mb{R}}_U(d - w)[2d - 2] \cong i_*\underline{\mb{R}}_D(d - w)[2d - 2].
\end{equation}
Explicitly, \eqref{eq:saito polarization 2} may be realized as the pullback from germs of differential forms on a neighborhood of $D$ in $X$ (the right hand side) into germs on the covering space of a punctured neighborhood induced by $\pi \colon \tilde{U} \to U$ (the left hand side). The pairing $\psi_f S$ is given by composing \eqref{eq:saito polarization 1}, \eqref{eq:saito polarization 2} and the canonical trace morphism
\[ \mrm{Tr} \colon i_*\underline{\mb{R}}_D(d - w)[2d - 2] \to \underline{\mb{R}}_X(d - w + 1)[2d].\]

As in \S\ref{subsec:SS polarization}, the pairing $\psi_f^{un} S$ induces a polarization on primitive parts as follows. The nilpotent operators $\mrm{N} \colon \psi_f^{un}\ms{M} \to \psi_f^{un}\ms{M}$ and
\[ \mrm{N} = \frac{1}{2\pi i} \mrm{T} \colon \psi_f^{un} K \to \psi_f^{un} K(-1) \]
are compatible under the isomorphism $\psi_f^{un}\alpha$, so they define a nilpotent morphism $\mrm{N} \colon \psi_f^{un}\ms{T} \to \psi_f^{un}\ms{T}(-1)$ of real mixed Hodge module. The weight filtration on $\psi_f^{un}\ms{T}$ is defined to be the monodromy weight filtration with respect to $\mrm{N}$ centered at $w - 1$, and we obtain pure Lefschetz primitive parts
\[ \mrm{P}_n \psi_f^{un}\ms{T} = (\mrm{P}_n\psi_f^{un}\ms{M}, \mrm{P}_n \psi_f^{un} K, \mrm{P}_n\psi_f^{un}\alpha) \subset \mrm{Gr}_{w - 1 + n}^W\psi_f^{un}\ms{T}.\]
The pairing $\psi_f S$ restricts to a pairing $\psi_f^{un}S$ on $\psi_f^{un}K$, and the induced pairings
\begin{equation} \label{eq:saito primitive polarization}
\psi_f^{un} S \circ (\id \otimes \mrm{N}^n) \colon \mrm{P}_n \psi_f^{un} K \otimes \mrm{P}_n \psi_f^{un} K \to \underline{\mb{R}}_X(d - w + 1 - n)[2d]
\end{equation}
are polarizations on $\mrm{P}_n\psi_f^{un}\ms{T}$ by definition \cite[5.2.10.2]{S1}.

We now relate the above constructions for real Hodge modules to the complex constructions of \S\ref{subsec:beilinson pairing}-\ref{subsec:SS polarization}. The key result is the following.

\begin{prop} \label{prop:nearby cycles pairings}
Let $\ms{T} = (\ms{M}, \ms{M}', \mf{s})$ be a complex mixed Hodge module on $U$, pure of weight $w$, and let $S \colon \ms{T} \to \ms{T}^h(-w)$ be a polarization. Then the isomorphism \eqref{eq:nearby cycles de rham} provides an isomorphism $\psi_f^{un}(\saito{\ms{T}}) \cong \saito{(\psi_f^{un}\ms{T})}$ such that $(\psi_f^{un}S)_\mb{R} = \psi_f^{un}(S_\mb{R})$ in the notation of \S\ref{subsec:mhm polarization}.
\end{prop}

\begin{rmk}
Given the importance of sign in the theory of polarizations, some readers may, like the authors, feel anxious about the many opportunities for sign errors in the following proof of Proposition \ref{prop:nearby cycles pairings}. As reassurance, one may check the statement for $X = \mb{C}^d$, $f$ a coordinate function, and $\ms{M} = \mc{O}_U$ with its usual Hodge structure and polarization. Since any sign error in the statement will be uniform across all examples, this is enough to rigorously ensure that none is present.
\end{rmk}

\begin{proof}
In light of the definitions, proving that $\psi_f^{un}(\saito{\ms{T}}) \cong \saito{(\psi_f^{un}\ms{T})}$ as Hodge modules amounts to checking that the Hodge filtrations on $\mb{D}\psi_f^{un}\ms{M}'$ and $\psi_f^{un}\mb{D}\ms{M}'$ agree under the isomorphism given implicitly by \eqref{eq:nearby cycles de rham} and $\psi_f^{un}\mf{s}$. This can be checked directly, or one can observe that since $(\psi_f^{un}S)_{\mb{R}}$ is compatible with the former Hodge filtration and $\psi_f^{un}(S_\mb{R})$ is compatible with the latter, it will follow once we have checked that these pairings agree.

To prove that the pairings agree, we first consider the complex pairing
\begin{align}
\DR\left(\frac{f^s\ms{M}((s))}{f^s\ms{M}[[s]]}\right) \otimes \overline{\DR\left(\frac{f^s\ms{M}((s))}{f^s\ms{M}[[s]]}\right)} &\xrightarrow{\DR(\beta) \otimes 1} \DR(f^s\ms{M}[[s]])[1] \otimes \overline{\DR\left(\frac{f^s\ms{M}((s))}{f^s\ms{M}[[s]]}\right)} \nonumber \\
&\xrightarrow{\DR(\Res S)} \underline{\mb{C}}_U[2d + 1], \label{eq:nearby cycles pairings 1}
\end{align}
where $\beta$ is the connecting homomorphism in \eqref{eq:beilinson triangle} and we write, as usual, $\overline{\,\cdot\,}$ for the operation of conjugating the scalar multiplication on a $\mb{C}$-sheaf. As in the construction of \eqref{eq:nearby cycles de rham}, we have a quasi-isomorphism
\[ \DR(\psi_f^{un}\ms{M}) \xrightarrow{-\DR(C(\beta))^{-1}[-1]} i_*i^*j_*\DR\left(\frac{f^s\ms{M}((s))}{f^s\ms{M}[[s]]}\right)[-1] \]
in the derived category, and hence a pairing
\begin{align} 
\DR(\psi_f^{un}\ms{M}) \otimes &\overline{\DR(\psi_f^{un}\ms{M})} \nonumber \\
& = i_*i^*j_*\DR\left(\frac{f^s\ms{M}((s))}{f^s\ms{M}[[s]]}\right)[-1] \otimes i_*i^*j_*\overline{\DR\left(\frac{f^s\ms{M}((s))}{f^s\ms{M}[[s]]}\right)}[-1] \nonumber \\
& \to i_*i^*j_*\left(\DR\left(\frac{f^s\ms{M}((s))}{f^s\ms{M}[[s]]}\right) \otimes \overline{\DR\left(\frac{f^s\ms{M}((s))}{f^s\ms{M}[[s]]}\right)}\right)[-2]  \label{eq:nearby cycles pairings 2}\\
& \xrightarrow{-\eqref{eq:nearby cycles pairings 1}} i_*i^*j_*\underline{\mb{C}}_U[2d - 1], \nonumber
\end{align}
where the sign in the last morphism is present for the same reason as in \eqref{eq:saito polarization 1}.

\begin{lem} \label{lem:nearby cycles pairings 1}
The pairing $\DR(\psi_f^{un}S)$ agrees with \eqref{eq:nearby cycles pairings 2} composed with
\[ i_*i^*j_*\underline{\mb{C}}_U[2d - 1] \xrightarrow{-\gamma} j_!\underline{\mb{C}}_U[2d] \rightarrow \underline{\mb{C}}_X[2d],\]
where $\gamma$ is the connecting homomorphism in the functorial exact triangle
\[ j_!j^* \to \mrm{id} \to i_*i^* \xrightarrow{\gamma} j_!j^*[1] \]
applied to $j_*\underline{\mb{C}}_U[2d - 1]$.
\end{lem}
\begin{proof}
A straightforward check, paying careful attention to signs.
\end{proof}

Next, observe that
\[ \psi_f^{un}(S_\mb{R}) = \mrm{Re}^{(d - w + 1)} \psi_f^{un}\DR(S),\]
where the $\mb{C}$-linear pairing
\[ \psi_f^{un}\DR(S) \colon \psi_f^{un}\DR(\ms{M}) \otimes \overline{\psi_f^{un}\DR(\ms{M})} \to \underline{\mb{C}}_X(d - w + 1)[2d]\]
is defined as in the real case.

\begin{lem} \label{lem:nearby cycles pairings 2}
Under the isomorphism \eqref{eq:nearby cycles de rham}, $\psi_f^{un}\DR(S)$ agrees with \eqref{eq:nearby cycles pairings 2} composed with
\[ i_*i^*j_*\underline{\mb{C}}_U[2d - 1] \xrightarrow{i_*i^*\delta[-1]} i_*i^!\underline{\mb{C}}_X[2d] \xrightarrow{\mrm{Tr}} \underline{\mb{C}}_X[2d],\]
where $\delta$ and $\mrm{Tr}$ are the morphisms in the canonical exact triangle
\begin{align*}
 i_*i^!\underline{\mb{C}}_X[2d] &\xrightarrow{\mrm{Tr}} \underline{\mb{C}}_X[2d] \\
& \to j_*\underline{\mb{C}}_U[2d] \xrightarrow{\delta} i_*i^!\underline{\mb{C}}_X[2d + 1].
\end{align*}
\end{lem}
\begin{proof}
Consider the exact triangle
\begin{align}
 \underline{\mb{C}}_U[2d] \to &\pi_*\pi^*\underline{\mb{C}}_U[2d] \nonumber \\
& \xrightarrow{\frac{1}{2\pi i}(\mrm{T} - 1)} \pi_*\pi^*\underline{\mb{C}}_U[2d] \xrightarrow{\beta'} \underline{\mb{C}}_U[2d + 1], \label{eq:monodromy triangle}
\end{align}
where $\mrm{T}$ is the monodromy operator. The statement of the lemma follows from commutativity of the diagrams
\[
\begin{tikzcd}
\DR\left(\frac{f^s\ms{M}((s))}{f^s\ms{M}[[s]]}\right) \otimes \overline{\DR\left(\frac{f^s\ms{M}((s))}{f^s\ms{M}[[s]]}\right)} \ar[r, "\Res"] \ar[d, "\eqref{eq:nearby cycles pairings 1}"] & \pi_*\pi^*\DR(\ms{M}) \otimes \pi_*\pi^*\overline{\DR(\ms{M})} \ar[d, "\pi_*\pi^*\DR(S)"] \\
\underline{\mb{C}}_U[2d + 1] & \ar[l, "\beta'"'] \pi_*\pi^*\underline{\mb{C}}_U[2d]
\end{tikzcd}
\]
and
\[
\begin{tikzcd}
i_*i^*j_*\pi_*\pi^*\underline{\mb{C}}_U[2d - 2] \ar[r, "\eqref{eq:saito polarization 2}"] \ar[d, "{\beta'[-2]}"] & i_*\underline{\mb{C}}_D[2d - 2] \ar[d, equal] \\
i_*i^*j_*\underline{\mb{C}}_U[2d - 1] \ar[r, "{i_*i^*\delta[-1]}"] & i_*i^!\underline{\mb{C}}_X[2d], 
\end{tikzcd}
\]
which may be checked by direct (if not completely trivial) calculation.
\end{proof}

Continuing with the proof of Proposition \ref{prop:nearby cycles pairings}, by Lemmas \ref{lem:nearby cycles pairings 1} and \ref{lem:nearby cycles pairings 2}, it remains to prove commutativity of the diagram
\begin{equation} \label{eq:nearby cycles pairings 5}
\begin{tikzcd}
i_*i^*j_*\underline{\mb{C}}_U[2d - 1] \ar[r, "-\gamma"] \ar[d, "{i_*i^*\delta[-1]}"] & j_!\underline{\mb{C}}_U[2d] \ar[d] \\
i_*i^!\underline{\mb{C}}_X[2d] \ar[r, "\mrm{Tr}"] & \underline{\mb{C}}_X[2d].
\end{tikzcd}
\end{equation}
To see this, consider the exact triangle
\begin{align*}
\underline{\mb{C}}_X[2d - 1] &\to j_*\underline{\mb{C}}_U[2d - 1] \\
&\xrightarrow{\delta[-1]} i_*i^!\underline{\mb{C}}_X[2d] \xrightarrow{\mrm{Tr}} \underline{\mb{C}}_X[2d].
\end{align*}
Applying the functorial exact triangle
\[ \id \to i_*i^* \xrightarrow{-\gamma} j_!j^*[1] \to \id[1] \]
to each term provides a diagram
\[
\begin{tikzcd}
\underline{\mb{C}}_X[2d - 1] \ar[r] \ar[d] & j_*\underline{\mb{C}}_U[2d - 1] \ar[r, "{\delta[-1]}"] \ar[d] & i_*i^!\underline{\mb{C}}_X[2d] \ar[r, "\mrm{Tr}"] \ar[d] & \underline{\mb{C}}_X[2d] \ar[d] \\
i_*i^*\underline{\mb{C}}_X[2d - 1] \ar[r] \ar[d, "-\gamma"] & i_*i^*j_*\underline{\mb{C}}_U[2d - 1] \ar[r, "{i_*i^*\delta[-1]}"] \ar[d, "-\gamma"] & i_*i^!\underline{\mb{C}}_X[2d] \ar[r, "i_*i^*\mrm{Tr}"] \ar[d, "-\gamma"] & i_*i^*\underline{\mb{C}}_X[2d] \ar[d, "-\gamma"] \\
j_!\underline{\mb{C}}_U[2d] \ar[r] \ar[d] & j_!\underline{\mb{C}}_U[2d] \ar[r] \ar[d] & 0 \ar[r] \ar[d] & j_!\underline{\mb{C}}_U[2d + 1] \ar[d] \\
\underline{\mb{C}}_X[2d] \ar[r] & j_*\underline{\mb{C}}_U[2d] \ar[r, "{\delta}"] & i_*i^!\underline{\mb{C}}_X[2d + 1] \ar[r, "{\mrm{Tr}[1]}"] & \underline{\mb{C}}_X[2d + 1]
\end{tikzcd}
\]
in which each row and column is an exact triangle. Since there is a functorial quasi-isomorphism $\operatorname{Cone}(\id \to i_*i^*) \to j_!j^*[1]$ at the level of complexes, the diagram above may be identified with a diagram as in Lemma \ref{lem:nearby cycles pairings 3} below, so \eqref{eq:nearby cycles pairings 5} commutes and we are done.
\end{proof}

\begin{lem}\label{lem:nearby cycles pairings 3}
Let
\[
\begin{tikzcd}
A \ar[r, "f"] \ar[d, "p"] & B \ar[d, "p'"] \\
A' \ar[r, "f'"] & B'
\end{tikzcd}
\]
be a commutative square of complexes in an abelian category $\mc{A}$, and let
\[
\begin{tikzcd}
A \ar[r, "f"] \ar[d, "p"] & B \ar[r, "g"] \ar[d, "p'"] & C \ar[r, "h"] \ar[d, "p''"] & A[1] \ar[d, "{p[1]}"] \\
A' \ar[r, "f'"] \ar[d, "q"] & B' \ar[r, "g'"] \ar[d, "q'"] & C' \ar[r, "h'"] \ar[d, "q''"] & A'[1] \ar[d, "{q[1]}"] \\
A'' \ar[r, "f''"] \ar[d, "r"] & B'' \ar[r, "g''"] \ar[d, "r'"] & C'' \ar[r, "h''"] \ar[d, "r''"] & A''[1] \ar[d, "{r[1]}"] \\
A[1] \ar[r, "{f[1]}"] & B[1] \ar[r, "{g[1]}"] & C[1] \ar[r, "{h[1]}"] & A[2]
\end{tikzcd}
\]
be the diagram obtained by taking cones horizontally then vertically, i.e., $C := \operatorname{Cone}(g)$, $C'' := \operatorname{Cone}(p'')$ etc. Assume that $C''$ is acyclic, so that $p''$ and $g''$ are quasi-isomorphisms. Then we have
\[ h \circ (p'')^{-1} \circ g' = r \circ (f'')^{-1} \circ q'  \]
as morphisms $B' \to A[1]$ in the derived category $\oD(\mc{A})$.
\end{lem}
\begin{proof}
We have
\[ C'' \cong \operatorname{Cone}(\operatorname{Cone}(A \xrightarrow{(p, -f)} A' \oplus B) \xrightarrow{(0, f'\pi_{A'} + p'\pi_B)} B'),\]
so the assumption implies $D:= \operatorname{Cone}(A \to A' \oplus B) \to B'$ is a quasi-isomorphism. The statement of the lemma now follows from the commutativity (up to homotopy) of the diagram below.
\[
\begin{tikzcd}
& C' & & C \ar[ll, "p''"'] \ar[dr, "h"] \\
B' \ar[ur, "g"] \ar[dr, "q'"'] & & D \ar[ll, "{(0, f'\pi_{A'} + p'\pi_B)}"'] \ar[ur, "{(\id, -\pi_B)}"] \ar[dr, "{(\id, \pi_{A'})}"] & & A[1] \\
& B' & & A'' \ar[ll, "f''"] \ar[ur, "r"']
\end{tikzcd}
\]
\end{proof}

\begin{cor} \label{cor:polarizations agree}
Assume in the setting of \S\ref{subsec:SS polarization} that $\ms{T} = (\ms{M}, \ms{M}', \mf{s})$ is Saito mixed Hodge module on $U$, pure of weight $w$, and $S \colon \ms{T} \to \ms{T}^h(-w)$ is a Saito polarization. Then $\mrm{P}_n \psi_f^{un}\ms{T}$ is a pure Saito mixed Hodge module of weight $w - 1 + n$, and \eqref{eq:SS primitive polarization} is a Saito polarization.
\end{cor}
\begin{proof}
The statement is equivalent to the assertion that
\[ (-1)^n (\psi_f^{un}(S) \circ (\id \otimes \mrm{N}^n))_\mb{R} \colon \DR(\psi_f^{un}\ms{M}) \otimes \DR(\psi_f^{un}\ms{M}) \to \underline{\mb{R}}_X(d - w + 1 - n)[2d] \]
is a polarization of $\saito{(\mrm{P}_n\psi_f^{un}\ms{T})}$. Since the operator $\mrm{N}$ is purely imaginary and $S$ is Hermitian, we have 
\[ (-1)^n (\psi_{f}^{un}(S) \circ (\id \otimes \mrm{N}^n))_\mb{R} = \psi_{f}^{un}(S)_\mb{R} \circ (\id \otimes \mrm{N}^n) = \psi_{f}^{un}(S_\mb{R}) \circ (\id \otimes \mrm{N}^n)\]
by Proposition \ref{prop:nearby cycles pairings}. The right hand side is a polarization by definition, so we are done.
\end{proof}

\end{document}